\documentclass[notitlepage]{scrartcl}

\usepackage{amsfonts}

\usepackage{graphicx}
\usepackage{float}
\usepackage{amsmath}
\usepackage{amsthm}
\usepackage{bbm}
\usepackage{hyperref}
\usepackage{enumitem} 
\usepackage[utf8]{inputenc}
\usepackage[T1]{fontenc}
\usepackage{lmodern}
\usepackage{mathrsfs}
\usepackage{changepage}
\usepackage{tikz-cd} 
\usepackage{multirow}
\usepackage[mathscr]{eucal}
\usepackage[toc,page]{appendix}
\usepackage{wasysym}
\usepackage{adjustbox}
\usepackage{blindtext}
\usepackage{makecell}
\usepackage{url}
\usepackage{pgfplots}
\usepackage{tikz}
\usepackage{tikz-3dplot}
\pgfplotsset{compat=1.15}
\usepackage{mathrsfs}
\usetikzlibrary{arrows}

\usepackage{titling}

\pretitle{\begin{center}\Huge\bfseries}
	\posttitle{\par\end{center}\vskip 0.5em}
\preauthor{\begin{center}\Large\ttfamily}
	\postauthor{\end{center}}
\predate{\par\large\centering}
\postdate{\par}

\title{Curved $\infty$-Local Systems And Projectively Flat Riemann-Hilbert Correspondence}
\author{Patrick Antweiler}

\begin{document}
	
	\maketitle
	\thispagestyle{empty}
	
	\begin{abstract}
		We generalize the higher Riemann-Hilbert correspondence in the presence of scalar curvature for a (possibly non-compact) smooth manifold $M$. We show that the dg-category of curved $\infty$-local systems, the dg-category of graded vector bundles with projectively flat $\mathbb Z$-graded connections and the dg-category of curved representations of the singular simplicial set of the based loop space of $M$ are all $A_\infty$-quasi equivalent. They provide dg-enhancements of the subcategory of the bounded derived category of twisted sheaves whose cohomology sheaves are locally constant of bounded rank. In the ungraded case, we reduce to an equivalence between projectively flat vector bundles and a subcategory of projective representations of $\pi_1(M; x_0)$. As an application of our general framework, we also prove that the category of cohesive modules over the curved Dolbeault algebra of a complex manifold $X$ is equivalent to a subcategory of the bounded derived category of twisted sheaves of $\mathcal O_X$-modules which generalizes a theorem due to Block to possibly non-compact complex manifolds.
	\end{abstract}

	\newcommand{\quot}[2]{{\raisebox{.2em}{$#1$}\left/\raisebox{-.2em}{$#2$}\right.}}
	\newcommand{\cats}[2]{\mathscr{#1}\mathrm{#2}}
	\newcommand{\cat}{\cats{C}{at}}
	\newcommand{\vect}{\cats{V}{ect}_\mathbb{K}}
	\newcommand{\vectr}{\cats{V}{ect}_\mathbb{R}}
	\newcommand{\vectc}{\cats{V}{ect}_\mathbb{C}}
	\newcommand{\vectcat}{\vect \cat}
	\newcommand{\set}{\cats{S}{et}}
	\newcount\colveccount
	\newcommand*\colvec[1]{
		\global\colveccount#1
		\begin{pmatrix}
			\colvecnext
		}
		\def\colvecnext#1{
			#1
			\global\advance\colveccount-1
			\ifnum\colveccount>0
			\\
			\expandafter\colvecnext
			\else
		\end{pmatrix}
		\fi
	}
	
	\newenvironment{nalign}{
		\begin{equation}
			\begin{aligned}
			}{
			\end{aligned}
		\end{equation}
		\ignorespacesafterend
	}
	
	\theoremstyle{definition}
	
	\newtheorem{def1}{Definition}[subsection]
	\newtheorem{notation}[def1]{Notation}
	\newtheorem{observation}[def1]{Observation}
	\newtheorem{remark}[def1]{Remark}
	\newtheorem{example}[def1]{Example}
	\theoremstyle{theorem}
	\newtheorem{lemma}[def1]{Lemma}
	\newtheorem{corollary}[def1]{Corollary}
	\newtheorem{theorem}[def1]{Theorem}
	\newtheorem{proposition}[def1]{Proposition}
	\newtheorem*{thm*}{Theorem}
	\newtheorem*{prop*}{Proposition}
	
	\tableofcontents

	\section{Introduction and outline}
	For a real smooth manifold $M$ the classical Riemann-Hilbert correspondence asserts that there are (categorical) equivalences between
	\begin{enumerate}
		\item $\mathbb K$-linear vector bundles on $M$ together with flat connections,
		\item $\mathbb K$-linear representations of the fundamental group of $M$, and
		\item local systems on $M$; either realized as locally constant sheaves of $\mathbb K$-modules or in a combinatorial manner.
	\end{enumerate}
	where $\mathbb K = \mathbb R$ or $\mathbb C$.
	These notions have been upgraded to a differential $\mathbb Z$-graded setting in \cite{BlockRiemann}, \cite{Mor1}, \cite{Mor2}, \cite{RiemannHilbert} and \cite{AbadSchaetz}. Vector bundles with flat connections are replaced by complexes of vector bundles together with a flat superconnection. A superconnection on a graded vector bundle $E^*$ is a finite sum
	\begin{align*} \mathbb E  = \nabla + \mathbb E_0 + \mathbb E_2 + \mathbb E_3 + \mathbb E_4 + ...\end{align*}
	where $\nabla$ is an ordinary connection on each $E^i$ and
	\begin{align*} \mathbb E_n \in \bigoplus_{k} \Omega^n(M, \text{Hom}(E^{k}, E^{k-n+1})).\end{align*}
	The superconnection is called flat if $\mathbb E^2 = 0$, in particular $E^*$ together with $\mathbb E_0$ is a complex of vector bundles. These objects can be arranged into a dg-category which can equivalently be described as the dg-category of cohesive modules over the de Rham Algebra (see $\cite{BlockRiemann}$). Representations of the fundamental group are replaced by representations of the dg-algebra of chains on the based loop space of $M$ (\cite{AbadSchaetz}, Section 4.2). The appropriate dg-version of local systems are called $\infty$-local systems (\cite{BlockRiemann}, Section 2.1). All these dg-categories provide models for the bounded derived category of sheaves of $\mathbb K$-modules such that their cohomology sheaves are locally constant and locally finitely generated. 
	Cohesive modules over the (uncurved) de Rham algebra have been shown to be equivalent to $\infty$-local systems in \cite{BlockRiemann}. The equivalence between representations and $\infty$-local systems has been shown in $\cite{Mor1}$. Moreover, in $\cite{Mor2}$ it is proven that $\infty$-local systems are equivalent to perfect complexes of sheaves (see Definition \ref{PerfectDefinition}). In $\cite{RiemannHilbert}$ cohesive modules are directly identified with perfect complexes of sheaves. Independently, in the paper $\cite{AbadSchaetz}$ the authors show that cohesive modules over the de Rham algebra are equivalent to representations. 
	\\ \\
	We will generalise these $A_\infty$-quasi equivalences, allowing for a curvature $h \in \Omega_{cl}^2(M)$, and show that these dg-categories provide dg-enhancements of the bounded derived category of twisted sheaves of $\mathbb K$-modules such that their cohomology sheaves are locally constant and have finitely generated fibers. \\ \\
	We will proceed as follows. In Section \ref{sectionDefs}, we fix the necessary notation to work with connections on vector bundles. We recall some basic results and define projectively flat vector bundles. \\ \\
	In the following section \ref{sectionCDG} we will develop a rather general framework to deal with dg-categories whose objects are complexes of objects belonging to a cdg-category. Such categories of twisted complexes were defined in \cite{Kapranov} (with the restriction that the initial category is a dg-category). There is also a version of twisted complexes for curved $A_\infty$-categories, e.g., in \cite{Wendy}. It is important that we will always consider so-called \emph{two-sided} twisted complexes. We follow this approach since the proofs of equivalence between (curved) cohesive modules and (curved) $\infty$-local systems on the one hand, and (curved) cohesive modules and (curved) representations of $\text{Sing}(\Omega_{x_0} M)$ on the other, are formally very similar. This approach both simplifies the definition of both $A_\infty$-functors and also reduces the proof of equivalence to a much simpler situation. To be precise, it suffices to show that the restriction of such an $A_\infty$-functor to complexes concentrated in degree $0$ is an $A_\infty$-quasi equivalence (under certain assumptions) which is made precise in the first main theorem of our paper:
	\begin{thm*}[\ref{EquivalenceTheorem}]
		Let $\mathcal F : \mathcal C\to \mathcal D$ be a proper $A_\infty$-functor between (non-negatively graded) cdg-categories such that both $\mathcal C$ and $\mathcal D$  are sufficiently twist-closed and split.  Assume that the induced $A_\infty$-functor $\mathcal F^{0} : \mathcal C^{0} \to \mathcal D^{0}$ is a quasi-equivalence. Then $\text{Tw}(\mathcal F): \text{Tw}(\mathcal C) \to \text{Tw}(\mathcal C)$ is an $A_\infty$-quasi equivalence of dg-categories.
	\end{thm*}
	The twist-closed condition refers to a property of the categories that one can twist differentials by endomorphisms of degree $1$. The assumption that both categories are split mean that certain short exact sequences split, where crucially, such splittings are not required to be closed. We believe that this Theorem is of separate interest and may be applied to other problems. \\ \\
	For a fixed curvature $h \in \Omega^2(M)$ we define the cdg-categories of curved $\infty$-local systems $\text{Loc}(M)^{pre}[H]$ (where  $H$ is the image of $h$ under the de-Rham map) and of curved cohesive modules $\mathcal P(M)^{pre}[h]$ in section \ref{sectionInfinityLocalSystems}. Their associated twisted dg-categories (see Definition \ref{defAssociatedDG}) $\text{Loc}(M)^{\infty}[H] := \text{Tw}(\text{Loc}(M)^{pre}[H])$ of $H$-curved $\infty$-local systems and $\mathcal P(M)^{\infty}[h] := \text{Tw}(\mathcal P(M)^{pre}[h])$ of $h$-curved cohesive modules will be the dg-categories of interest.
	Let us briefly recall why these categories are dg-categories rather than cdg-categories. As stated before, we consider a fixed curvature $h \in \Omega_{cl}^2(M)$. Consider two vector bundles $(E, \nabla)$ and $(F, \nabla')$. The main observation is that even though these vector bundles are curved, the vector bundle of homomorphisms $\text{Hom}(E,F)$ is again flat which allows us to do homological algebra; in particular we can consider the complex $\Omega^*(M, \text{Hom}(E,F))$ with its exterior differential squaring to $0$. Thus, we can arrange vector bundles with curvature $h$ again into a dg-category. \\ \\
	We define an $A_\infty$-functor of cdg-categories which induces an $A_\infty$-functor of dg-categories
	\begin{align*} \mathcal{RH}(M)^{\infty}[h] : \mathcal P(M)^{\infty}[h] \to \text{Loc}(M)^{\infty}[H].\end{align*} Once this functor is defined, it is relatively simple to apply the results from the preceding section to show that this functor is indeed a quasi-equivalence (Theorem \ref{hCurvedDGEquivalence}). However, to define this functor there are significant problems, some of which already arise in much simpler cases. Consider the de Rham integration map:
	\begin{align*} DR : \Omega^*(M, \mathbb K) \to C^*(M, \mathbb K), DR(\omega)(\sigma) := \int_{\Delta^*} \sigma^* \omega.\end{align*}
	This map is a map of complexes and known to be a quasi-isomorphism but it is not a map of dg-algebras because it is not compatible with the products. However, this issue can be fixed by defining an $A_\infty$-map of dg-algebras whose first component is equal to $DR$ (see \cite{Gugenheim}). The Gugenheim map will appear as the restriction of the functor $\mathcal{RH}(M)^\infty[h = 0]$ to the constant local system $\underline{\mathbb K}$. Therefore, we may only expect to obtain an $A_\infty$-functor rather than a dg-functor. The other problem lies in the fact that we have to define the functor for all vector bundles, not just the flat ones. Given a vector bundle $(E, \nabla)$ there is a priori no canonical way to integrate over sections of $E$. In case the bundle is flat and we integrate over simplices there is no problem since the parallel transport canonically trivializes $E$. If the bundle is not flat some choice has to be made. To do this in a homotopy coherent manner, we define a family of smooth maps $\smiley^n : I \times I^{n} \to \Delta^{n+1}$ satisfying certain coherency axioms (Definition \ref{AdmissibleAxioms}).\\ \\ For composable morphisms in $\mathcal P(M)^{pre}[h]$ we introduce forms $\text{hol}_{f_1,...,f_n}$ in Proposition \ref{holProp} which we call holonomy forms. These are vector bundle-valued forms on the path-space of $M$. Typically these forms are defined in terms of iterated integrals. We believe it is more straightforward to define them as solutions to differential equations. This is analogous to how parallel transport is defined: Either locally as an iterated integral or as a solution to a differential equation. A very simple case of holonomy forms appear in the standard proof that de Rham cohomology is homotopy invariant in the following way: Consider the inclusions $\iota_0 : M \to M \times I$ and $\iota_1 : M \to M \times I$. $\Omega(\iota_0)$ and $\Omega(\iota_1)$ are homotopic maps of complexes witnessed by the homotopy
	\begin{align*} H(\omega) := \int_I \iota_{\partial_t}(\omega) dt.\end{align*}
	This $H(\omega)$ is, in our formulation, precisely the pullback of the holonomy form $\text{hol}_\omega$ (up to sign) to $M \times I$. Then we prove:
	\begin{thm*}[\ref{hCurvedDGEquivalence}]
		The induced $A_\infty$-functor 
		\begin{align*}\mathcal{RH}(M)^{\infty}[h] : \mathcal P(M)^{\infty}[h] \to \text{Loc}(M)^{\infty}[H]\end{align*}
		is a quasi-equivalence of dg-categories. 
	\end{thm*}
	The following section \ref{sectionProjRepMonoids} deals with representations of the singular simplicial set of the Moore loop space $\Omega_{x_0}(M)$ which is a simplicial monoid under concatentation. The Moore loop space is more convenient than the actual loop space because the composition of paths is associative. The notion of curved representations will be defined for arbitrary simplicial monoids. This construction takes as input a simplicial monoid $S$ and a Hochschild 2-class $[H]$ with values in the bimodule $\mathbb C^*$
	\[ HH^2(C_*(S), \mathbb C^*) \cong H^2(RHom_{C_*(S)^{op} \otimes C_*(S) }(C_*(S), \mathbb C^*)) \cong H^2(\text{Hom}_{\mathbb Z}^*(B(C_*(S)), \mathbb C^*))\]
	where $C_*(S)$ is considered as a dg-algebra whose multiplication is induced by the composition of the Eilenberg-Zilber map with the concatenation in $S$ and $B(-)$ denotes the bar construction of an augmented dg-algebra. Moreover, $\mathbb C^*$ is considered as a bimodule over $C_*(S)$ via the augmentation $C_*(S) \to \mathbb Z$. To $S$ and $H$ we associated another dg-algebra $C_*^H(S)$ whose underlying vector space agrees with that of $C_*(S) \otimes \mathbb C$ but its differential and multiplication are twisted by $H$. The dg-algebra $C_*^H(S)$ admits an augmenation if and only if the cohomology class $H$ is trivial. We will define a certain category of representations of the pair $(S,H)$ denoted by $\text{Rep}(S)^{\infty}[H]$ which we will realise again as the twist-construction $\text{Tw}(-)$ of a certain cdg-category $\text{Rep}(S)^{pre}[H]$. Ignoring the latter, the homotopy category of $\text{Rep}(S)^{\infty}[H]$ can be identified with the full subcategory of $D(C_*^H(S))$ consisting of dg-modules that have finite-dimensional cohomology. In more precise terms, the dg-category $\text{Rep}(S)^{\infty}[H]$ coincides with the dg-category of $A_\infty$-modules over $C_*^H(S)$ which are finite-dimensional. \\ \\
	A special class of curvature elements, which we call \emph{logarithmic}, are those contained in the image of the exponential mapping
	\[ HH^2(C_*(S), \mathbb C) \to HH^2(C_*(S), \mathbb C^*)\]
	Letting $h$ be a closed 2-form on $M$ and $H$ the corresponding element in $HH^2(C_*(\Omega_{x_0}(M)), \mathbb C)$ we again use the holonomy forms of the previous section, and define an $A_\infty$-functor to obtain:
	\begin{thm*}[\ref{hTheorem2}]
		The induced $A_\infty$-functor
		\begin{align*}\text{Rep}(M)[h] : \mathcal P(M)^{\infty}[h] \to \text{Rep}(\Omega_{x_0}(M))^{\infty}[e^{-H}] \end{align*}
		of dg-categories is a quasi-equivalence.
	\end{thm*}
	We will then consider the category $\text{PF}_{\infty}(M)$ consisting of graded vector bundles with arbitrary (scalar) curvatures and the category $\text{LPRep}(\Omega_{x_0}(M))$ of representations with arbitrary logarithmic curvatures. In this category, there can be morphisms between vector bundles having different curvatures. Furthermore, for any line bundle $L$ we have that any vector bundle $E$ is isomorphic to $E \otimes L$ in this category. In the latter category, roughly speaking, representations are isomorphic if they are projectively the same. We conclude the section with our main theorem:
	\begin{thm*}[\ref{PFPrep}]
		\begin{align*}\text{PRep}(M) : \text{PF}_{\infty}(M) \to \text{LPRep}(\Omega_{x_0}(M))\end{align*}
		is an equivalence of categories.
	\end{thm*}
	On the right side, we restrict to logarithmic projective representations, which is a restriction only on the curvature. Theorem \ref{PFPrep} is thus a version of Theorem \ref{hTheorem2} where we assembled curved representations of different curvatures into one category.\\ \\
	In the next section, we want to reduce our results to the 1-categorical case, i.e., we restrict to (ungraded) projectively flat vector bundles. In this case, projectively flat vector bundles correspond to projective representations of the fundamental group. Similarly to the graded case, we have to restrict to certain curvature elements, see below for more details. The discussion will amount to the following theorem:
	\begin{thm*}[\ref{TheoremUngraded}]
		The category of projectively flat vector bundles $\text{PF}_{0}(M)$ is equivalent to the category $\text{PRep}^0_{H^2_B(\pi_1(M), \mathbb C^*)}(\pi_1(M))$ of projective representations of $\pi_1(M)$ whose curvatures belong to $H^2_B(\pi_1(M), \mathbb C^*)$. These are precisely the classes in $H^2(\pi_1(M), \mathbb C^*)$ that get mapped to $0$ under the composition $H^2(\pi_1(M), \mathbb C^*) \to H^2(M, \mathbb C^*) \to H^3(M, \mathbb Z)$ where the first map is induced by the natural map $M \to K(\pi_1(M),1)$ to the Eilenberg-MacLane space of $\pi_1(M)$ and the latter is the connecting morphism induced by the short exact sequence $0 \to \mathbb Z \xrightarrow{2 \pi i} \mathbb C \xrightarrow{\exp} \mathbb C^* \to 0$.
	\end{thm*}
	In Example \ref{counterexampleProj} we will see that this restriction is necessary. In particular, this shows that not every projectively flat representation of $\pi_1(M; x_0)$ arises from a projectively flat vector bundle. We will also see that projectively flat vector bundles with coefficients in $\mathbb R$ are no more interesting than flat vector bundles:
	\begin{prop*}[\ref{realBoring}]
		Let $\mathbb K = \mathbb R$. Let $(E, \nabla)$ be a non-zero projectively flat vector bundle. Then there is a 1-form $\omega \in \Omega^1(M)$ such that
		\begin{align*} \nabla' := \nabla - \omega\end{align*}
		is a flat connection on $E$. In particular, the curvature $h$ of $\nabla$ is equal to $d \omega$, i.e., $h$ is an exact 2-form.
	\end{prop*}
	Interestingly, even though this chapter concerns ungraded objects, its results have consequences for cohesive modules too. Namely from the previous Proposition we obtain:
	\begin{thm*}[\ref{ThmRealBoring}]
		Let $\mathbb K = \mathbb R$. Let $h$ be a closed real 2-form on $M$, then we have the following quasi-equivalences:
		\begin{align*}
			\mathcal P(M)^{\infty}[h] \simeq \begin{cases}
				\mathcal P(M)^{\infty}[0] \text{ if } [h] = 0 \text{ in } H^2_{DR}(M, \mathbb R)\\
				0 \text{ else }
			\end{cases}
		\end{align*}
	\end{thm*}
	
	In the last section we show that twisted sheaves over $\mathbb K$ correspond to cohesive modules over the cdga $(\Omega^*(M),h)$. The method of proof is largely motivated by the paper \cite{RiemannHilbert} (which however does not consider curvature). Our approach will be slightly different as we do not use model category theory; instead we will work with the notion and existence of $K$-projective, $K$-injective and $K$-flat resolutions due to \cite{Spaltenstein}. The obtained correspondence is given by an equivalence between the homotopy category of cohesive modules and a subcategory of the derived category of sheaves of modules. Roughly speaking, the general situation is as follows. We consider a sheaf $\mathcal R$ of commutative algebras and a soft resolution of this sheaf $\mathcal R \to \mathcal A^*$ where $\mathcal A^*$ is a graded commutative algebra. Letting $h \in A^2_{cl}$ we can twist the algebra $A^*$ by $h$ to obtain a curved differentially graded algebra $(A^*, h)$. We show that twisting by $h$ corresponds to a gerbe in $\mathcal R$ and thus consider the derived category of sheaves, twisted by this gerbe. This section can be read independently of the previous ones and is of separate interest. We show: 
	\begin{thm*}[\ref{MainEq}]
		Let $\mathscr R$ be a sheaf of $k$-algebras and let $\mathscr R \to \mathscr A^*$ be a quasi-isomorphism of sheaves of dg-algebras on a topological space $X$. Set $A^* := \mathscr A^*(X)$ and let $h \in A^2(X)$ be closed. Under the assumptions \ref{ListAssumptions} we obtain an equivalence of categories
		\begin{align*}H^0((A^*, h)\text{-Mod}^{coh}) \to D^B_{perf}(\mathscr R)^{\check h}\end{align*}
		induced by the functor $J$ as in Definition \ref{DefJFunctor}.
	\end{thm*}
	The first category is the homotopy category of cohesive modules over the cdg-algebra $(A^*, h)$ and the latter is a full subcategory of the derived category of twisted sheaves of $\mathcal R$-modules whose objects are globally bounded perfect (see Definition \ref{PerfectDefinition}).
	Applying this to our main case of interest, we obtain:
	\begin{thm*}[\ref{MainEqM}]
		Let $M$ be a real, connected smooth manifold. Let $\mathbb K = \mathbb R$ or $\mathbb K = \mathbb C$ and let $h \in \Omega_{cl}^2(M, \mathbb K)$ be a closed 2-form. Then there are equivalences
		\begin{align*}H^0(\mathcal P(M)^{\infty}[h])) \to H^0((\underline \Omega^*(M, \mathbb K), h)\text{-Mod}^{\text{coh}}) \to D_{\text{fclc}}(\underline{\mathbb K})^{\check h} \end{align*}
	\end{thm*}
	The second category is the homotopy category of the dg-category of sheaves of cohesive modules over the sheaf of cdg-algebras $(\underline \Omega^{0,*}(X), h)$. The latter category is a subcategory of the derived category of twisted sheaves of $\mathbb K$-vector spaces such that the cohomology is concentrated in finitely many degrees and each cohomology sheaf is locally constant and is locally finitely generated.
	We also apply our general framework to a complex manifold $X$ to recover the following theorem, which has been claimed by Block in \cite{Block} in the case of a compact complex manifold $X$:
	\begin{thm*}[\ref{MainEqX}]
		Let $X$ be a connected complex manifold. Let $h \in \Omega_{cl}^{0,2}(M)$ be a $\bar \partial$-closed $(0,2)$-form. Then there are equivalences
		\begin{align*}H^0((\Omega^{0,*}(X), h)\text{-Mod}^{\text{coh}}) \to H^0((\underline \Omega^{0,*}(X), h)\text{-Mod}^{\text{coh}}) \to D^B_{\text{perf}}(X)^{\check h}.\end{align*}
		The latter category embeds fully faithfully into the bounded derived category $D^B_{\text{coh}}(X)^{\check h}$ of objects whose cohomology sheaves are coherent. This embedding is an equivalence if $X$ is compact.
	\end{thm*}
	
	Here is a summary of different equivalent notions of curved local systems and the corresponding $\infty$-local systems together with a description of how the curvatures are naturally parametrized in each case.
	\begin{center}
		\footnotesize
		\begin{tabular}{ |c|c|c|c|c| } 
			\hline
			1-categorical & higher categorical & \makecell{(1-categorical)\\curvature\\ parametrized by} & \makecell{(higher categorical)\\curvature\\parametrized by}   \\
			\hline
			\makecell{(subclass of) projective \\representations\\ of $\pi_1(M)$} & \makecell{logarithmic projective \\ representations of \\$\text{Sing}(\Omega_{x_0}(M))$} & \makecell{$H^2_{B}(\pi_1(M), \mathbb C) \subset$\\ $H^2_{grp}(\pi_1(M), \mathbb C^*)$} & $H^2(B^*(C_*(\text{Sing}(\Omega_{x_0}(M)), \mathbb C))$\\ 
			\hline
			\makecell{projectively flat \\ vector bundles} & \makecell{cohesive modules over \\the cdga $(\Omega^* (M), h)$/\\supervectorbundles \\together with \\ projectively flat \\superconnections} & $H^2_{DR}(M, \mathbb C)$ & $H^2_{DR}(M, \mathbb C)$\\ 
			\hline
			\makecell{curved local systems \\(combinatorially)} & curved $\infty$-local systems & $H^2_{sing}(M, \mathbb C)$ &  $H^2_{sing}(M, \mathbb C)$\\
			\hline
			\makecell{twisted locally\\ constant and\\ finitely generated\\ sheaves} & \makecell{bounded complexes \\of twisted sheaves\\ with locally constant\\ cohomology sheaves\\ of bounded rank} & $\check H^2(M; \mathbb C)$ & $\check H^2(M; \mathbb C)$\\ 
			\hline
		\end{tabular}
	\end{center}
	All the higher categorical groups of curvature elements are isomorphic.
	We further note that in all cases, except for vector bundles, one can consider more general curvature elements by replacing $\mathbb C$ with $\mathbb C^*$. For example, a projective representation of $\text{Sing}(\Omega_{x_0}(M))$ has curvature contained in $H^2(B^*(C_*(\text{Sing}(\Omega_{x_0}(M)), \mathbb C^*))$. We say that a projective representation is logarithmic projective, if its curvature is in the image of the exponential map $H^2(B^*(C_*(\text{Sing}(\Omega_{x_0}(M)), \mathbb C)) \to H^2(B^*(C_*(\text{Sing}(\Omega_{x_0}(M)), \mathbb C^*))$. Equivalently, this means that the curvature element is mapped to $0$ under the connecting morphism $H^2(B^*(C_*(\text{Sing}(\Omega_{x_0}(M)), \mathbb C)) \to H^3(B^*(C_*(\text{Sing}(\Omega_{x_0}(M)), \mathbb Z))$. In the 1-categorical case, there is an exact sequence
	\[ H^2(M, \mathbb Z) \to H_B^2(M, \mathbb C) \to H^2(\pi_1(M), \mathbb C^*) \to H^3(M, \mathbb Z)\]
	where $H_B^2(M, \mathbb C)$ can be described as the subgroup of $H_{sing}^2(M, \mathbb C)$ consisting of those classes which can be represented by a 2-cochain $f$ such that $f(\sigma) = f(\sigma') \text{ Mod } 2\pi i$ whenever $\partial \sigma = \partial \sigma'$. $H^2_B(\pi_1(M), \mathbb C)$ is then defined to be the image of $H_B^2(M, \mathbb C) \to H^2(\pi_1(M), \mathbb C^*)$. In particular, we obtain that a projective representation of $\pi_1(M)$ corresponds to a projectively flat vector bundle if and only if its curvature is contained in the subgroup $H^2_B(\pi_1(M), \mathbb C)$ which is equivalently characterised by an obstruction map $H^2(\pi_1(M), \mathbb C^*) \to H^3(M, \mathbb Z)$ (cf. Proposition \ref{propLongExact}).

	\section{Preliminaries}
	\label{sectionDefs}
	The aim of this section is to fix some notations and conventions. Throughout the paper, we will consider a real smooth connected manifold $M$ without boundary. Only in the last section we will consider more general topological spaces. $\mathbb K$ will either denote the field $\mathbb R$ of real numbers or the field $\mathbb C$ of complex numbers. The set of sections of differential $n$-forms with coefficients in $\mathbb K$ is denoted by 
	\begin{align*} \Omega^n(M) := \Omega^n(M; \mathbb K) := \Gamma( \bigwedge^n T^*M  \otimes_{\mathbb R} \mathbb K).\end{align*}
	If $E \to M$ is a vector bundle over $\mathbb K$ we denote the sections of differential $n$-forms with values in $E$ by
	\begin{align*} \Omega^n(M; E) := \Gamma( \bigwedge^n T^*M \otimes_{\mathbb R} E) \cong \Omega^n(M; \mathbb K) \otimes_{C^{\infty}_{\mathbb K}(M) } \Gamma(E). \end{align*}
	Note that in our case, the vector bundle $E$ is on the right. In particular, $\Omega^*(M; E)$ is a left $\mathbb Z$-graded module over $\Omega^*(M)$. With this notation a connection on $E$ is a $\mathbb K$-linear map
	\begin{align*} \nabla : \Omega^0(M; E) \to \Omega^1(M; E)\end{align*}
	satisfying the Leibniz rule
	\begin{align*} \nabla(f \otimes e) = df \otimes e + f \otimes \nabla e\end{align*}
	for all $f \in \Omega^0(M)$ and $e \in \Gamma(E)$. We extend $\nabla$ to a $\mathbb K$-linear map
	\begin{align*} \nabla : \Omega^i(M; E) \to \Omega^{i+1}(M; E)\end{align*}
	for all $i > 0$ by the graded Leibniz-rule
	\begin{align*} \nabla(\omega \otimes e) = d\omega \otimes e + (-1)^{|\omega|} \omega \wedge \nabla e.\end{align*}
	In equations as the one above, we will often omit wedge and tensor symbols.
	\begin{def1}
		The \emph{curvature} of a connection $\nabla$ on a vector bundle $E$ is given by 
		\begin{align*}\nabla^2 := \nabla \circ \nabla : \Omega^0(M; E) \to \Omega^2(M; E).\end{align*}
		Noting that $\nabla^2$ is $C^{\infty}$-linear, we will view $\nabla^2$ as an element of $\Omega^2(M, \text{End}(E))$.
	\end{def1}
	
	\begin{def1}
		\begin{enumerate}
			\item A connection $\nabla$ is called \emph{flat} if $\nabla^2 = 0$.
			\item A connection $\nabla$ is called \emph{projectively flat} if $\nabla^2$ has the form
			\begin{align*} \nabla^2 = h \otimes \mathbbm 1_{E}\end{align*}
			for some $h \in \Omega^2(M)$.
			We will refer to $h$ as the curvature of the projectively flat connection $\nabla$.
		\end{enumerate}
	\end{def1}
	\begin{remark}
		\label{ZeroConvention}
		If $\nabla$ is a projectively flat connection, then $h$ is uniquely determined by the equation \begin{align*} \nabla^2 = h \otimes \mathbbm 1_{E}\end{align*} unless the vector bundle is the zero bundle. In the case of the zero bundle, we allow any closed 2-form $h$ to be referred to as the curvature of the connection.
	\end{remark}
	\begin{def1}
		Let $(E,\nabla)$ be a vector bundle with connection. We use the symbol $PT$ to denote the parallel transport of $(E, \nabla)$, i.e., it assigns to each path $\gamma : [a,b] \to M$ the parallel transport along the path $\gamma$: 
		\begin{align*}PT(\gamma) : E_{\gamma(a)} \to E_{\gamma(b)}.\end{align*}
		In the last chapter we will consider Moore loops $(\gamma, l)$. In particular, $\gamma$ is a path $\gamma : [0, \infty) \to M$ and $l > 0$ is a number such that $\gamma(t) = x_0$ for all $t \geq l$ and $\gamma(0) = x_0$. Then the parallel transport along $\gamma$ is defined to be the parallel transport along $\gamma|_{[0, l]}$.
	\end{def1}
	\begin{def1}
		As usual, given a vector field $X \in \mathfrak{X}(M)$ and a section $e \in \Gamma(E)$ of a vector bundle $E$ with connection $\nabla$, we can consider the section $\nabla_X e := (\nabla e)(X)$ of $E$. We also define how $\nabla_X$ acts on elements $A \in \Omega^n(M, E)$. We set
		\begin{align*} \nabla_X A := \iota_X(\nabla A) + \nabla(\iota_X A).\end{align*}
	\end{def1}
	
	\label{sectionProjFlat}
	In our case, it turns out to be useful to introduce so called time-ordered simplices.
	\begin{def1}
		\label{DefSimplices}
		The \emph{(time-ordered) standard $n$-simplex} is the smooth $n$-dimensional manifold with corners
		\begin{align*} \Delta^n := \{(t_1,...,t_n) | 1 \geq t_1 \geq t_2 \geq ... \geq t_n \geq 0 \} \subset \mathbb R^n\end{align*}
		with its induced orientation from $\mathbb R^n$. For $n = 0$, we set $\Delta^0$ equal to the point $\{0\}$. An $n$-simplex $\sigma$ in $M$ is a smooth map $\sigma : \Delta^n \to M$. The $i$-th face ($i = 0,...,n$) of $\Delta^n$ is given by 
		\begin{align*}\partial_i \Delta^n  = \{(t_1,...,t_n) \in \Delta^n | t_i = t_{i+1}\}\end{align*}
		where we have set $t_0 = 1$ and $t_{n+1} = 0$. For each such face there is a diffeomorphism
		$q_i : \Delta^{n-1} \to \partial_i \Delta^n$ given by $(t_1,...,t_{n-1}) \mapsto (t_1,...,t_i,t_i,...,t_{n-1})$ for $i = 1,...,n-1$ and $(t_1,...,t_{n-1}) \mapsto (1,t_1,...,t_{n-1})$ for $i = 0$ and $(t_1,...,t_{n-1}) \mapsto (t_1,...,t_{n-1},0)$ for $i = n$.\\
		For an $n$-simplex $\sigma$ in $M$ we define the $i$-th face of $\sigma$ to be the restriction \begin{align*}\partial_i \sigma = \sigma \circ q_i : \Delta^{n-1} \to M. \end{align*}
		We also define maps $p_i : \Delta^{n+1} \to \Delta^{n}$ for $i = 0,...,n$ by 
		\begin{align*} p_i(t_1,...,t_{n+1}) = (t_1,...,\hat{t_{i+1}},...,t_{n+1}).\end{align*} Using the $p_i$, we may now define the degeneracy maps $s_i$ by
		\begin{align*}s_i \sigma = \sigma \circ p_i : \Delta^{n+1} \to M.\end{align*}
		We define $\text{Sing}(M)_n$ as the set of $n$-simplices in $M$.
	\end{def1}
	\begin{remark}
		Typically, the standard $n$-simplex is defined as the space
		\begin{align*} \Delta^n_{std} := \{(x_0,...,x_n) | \sum x_i = 1, x_i \geq 0\} \subset \mathbb R^{n+1}\end{align*}
		or equivalently as the convex hull of the vectors $e_0,....,e_n$. Now set 
		\begin{align*} v_0 := (0,...,0) \in \mathbb R^n, v_1 := (1,0,...,0), v_{n-1} := (1,...,1,0), v_n := (1,...,1).\end{align*}
		Then $\Delta^n$ is the convex hull of $v_0,...,v_n$ and there is an orientation-preserving diffeomorphism
		\begin{align*} \Delta^n \to \Delta^{n}_{std}, \sum t_i v_i \mapsto \sum t_i e_i.\end{align*}
		It is straightforward to check that the degeneracy and face maps are compatible with these diffeomorphisms. We will freely use this observation and use results proved for the usual simplicial set $(\text{Sing}(M))_{std})_*$. This also shows that $\text{Sing}(M)_*$ is a simplicial set.
	\end{remark}
	
	\section{Curved and uncurved dg-categories}
	\subsection{Curved dg-categories and the dg-category of twisted complexes}
	\label{sectionCDG}
	In this section, $\mathbb K$ may also be a commutative unital ring. We will first recall the definition of a dg-category. For more details, see \cite{Toen}.
	\begin{def1}
		A \emph{dg-category} (= differential graded category) over $\mathbb K$ is a category $\mathcal C$ carrying the following additional data:
		\begin{enumerate}
			\item The morphism sets $\text{Hom}(E,F)$ are endowed with the structure of a graded vector space over $\mathbb K$ and a $\mathbb K$-linear map $d$ which will be referred to as the \emph{differential} \begin{align*}d : \text{Hom}^*(E,F) \to \text{Hom}^{*+1}(E,F). \end{align*}
		\end{enumerate}
		Furthermore it satisfies the following properties:
		\begin{enumerate}
			\item The composition of morphisms is bilinear, respects the degrees, and fulfills the graded Leibniz rule
			\begin{align*} d(f \circ g) = d(f) \circ g + (-1)^{|f|} f \circ d(g).\end{align*}
			\item For any $f \in \text{Hom}^{*}(E,F)$ the following identity holds
			\begin{align*} d^2(f) = 0.\end{align*}
		\end{enumerate}
	\end{def1}
	Equivalently a dg-category is just a category enriched over the category of chain complexes over $\mathbb K$.
	\begin{def1}
		Let $\mathcal C$ be a dg-category. To $\mathcal C$ we can associate its \emph{homotopy category} $H^0(\mathcal C)$. It has the same objects as $\mathcal C$. The morphisms are
		\begin{align*} \text{Hom}_{H^0(\mathcal C)}(E,F) := H^0(\text{Hom}^*_{\mathcal C}(E,F), d).\end{align*}
		By the Leibniz rule, the composition in $\mathcal C$ induces a composition in $H^0(\mathcal C)$.
	\end{def1}
	
	\begin{def1}
		A \emph{(non-negatively graded) cdg-category} (= curved differential graded category) over $\mathbb K$ is a category $\mathcal C$ carrying the following additional data:
		\begin{enumerate}
			\item The morphism sets $\text{Hom}(E,F)$ are endowed with the structure of a non-negatively graded vector space over $\mathbb K$ and a $\mathbb K$-linear map $\nabla$ which we will be referred to as the \emph{connection} \begin{align*}\nabla : \text{Hom}^*(E,F) \to \text{Hom}^{*+1}(E,F) \text{ and }\end{align*}
			\item Each object $E$ is equipped with a $\nabla$-closed endomorphism $R_E \in \text{Hom}^2(E,E)$, called the \emph{curvature} of $E$, of degree 2.
		\end{enumerate}
		Furthermore, it satisfies the following properties:
		\begin{enumerate}
			\item The composition of morphisms is bilinear, respects the degrees, and fulfills the graded Leibniz rule
			\begin{align*} \nabla(f \circ g) = \nabla(f) \circ g + (-1)^{|f|} f \circ \nabla(g).\end{align*}
			\item For any $f \in \text{Hom}^{*}(E,F)$, the following identity holds
			\begin{align*} \nabla^2(f) = R_F \circ f - f \circ R_E.\end{align*}
			\item $\mathcal C$ is pointed. We write this object as $0$.
		\end{enumerate}
	\end{def1}
	\begin{remark}
		In the following we will always assume that the occuring cdg-categories are non-negatively graded without explicitly spelling this out.
	\end{remark}
	
	\begin{remark}
		There is a notion of curved $A_\infty$-algebras and curved $A_\infty$-categories (e.g., in \cite{Armstrong}) which Keller calls weak $A_\infty$-categories, see \cite{Keller}. In \cite{Wendy}, such categories are called $A_{[0,\infty)}$-categories. A cdg-category is then just a curved $A_\infty$-category such that all higher compositions are trivial. For simplicity we have chosen not to work in this generality; however, we will need $A_\infty$-functors between cdg-categories. At the same time, this generality is not irrelevant for the matter at hand: In the last section, we are interested in the loop space. If we had chosen to work with the actual loop space, we would have needed the notion of curved $A_\infty$-categories as the concatenation of paths is not associative. Instead, we decided to work with the space of Moore loops in which the concatenation is associative, and we are thus again only dealing with cdg-categories.
	\end{remark}
	
	\begin{example}
		Let $G$ be a group (or more generally a monoid) and set $\mathbb K := \mathbb Z$. We define the \emph{cdg-category of pre-representations} $\text{Rep}^{pre}(G)$ of $G$. An object of $\text{Rep}^{pre}(G)$ is an abelian group $M$ together with a map 
		\begin{align*}\mu_M : G \to \text{Hom}_{\mathbb Z}(M, M) \end{align*}
		satisfying $\mu_M(e) = \mathbbm 1$. For two objects $M$ and $N$, the morphism spaces are given by 
		\begin{align*} \text{Hom}^n(M, N) := \text{Hom}_{\text{Set}}(G^n, \text{Hom}_{\mathbb Z}(M, N)).\end{align*}
		The connection is defined as
		\begin{align*} \nabla(f)(g_1,...,g_n) &:= \mu_N(g_1) \circ f(g_2,...,g_n) + \sum_{i = 1}^{n-1} (-1)^i f(g_1,...,g_{i}g_{i+1},...,g_n) \\
			&\quad + (-1)^n f(g_1,...,g_{n-1}) \circ \mu_M(g_n).\end{align*}
		One computes (a similar but more involved computation can be found in the proof of Lemma \ref{RepIsCdg})
		\begin{align*}
			\nabla^2(f)(g_1,...,g_n) &= (\mu_N(g_1)\mu_N(g_2) - \mu_N(g_1 g_2)) \circ f(g_3,...,g_n) \\
			&\quad- f(g_1,...,g_{n-2})\circ  (\mu_M(g_{n-1})\mu_M(g_{n}) - \mu_M(g_{n-1} g_{n})).
		\end{align*}
		We therefore define the curvature of $M$ to be
		\begin{align*} R_M(g,h) := \mu_M(g)\mu_M(h) - \mu_M(gh).\end{align*}
		The curvature thus measures the failure of $\mu_M$ to define a left action of $G$ on $M$.
	\end{example}
	
	\begin{def1}
		Let $\mathcal C$ be a cdg-category.
		\begin{enumerate}
			\item We say that $\mathcal C$ is \emph{twist-closed} if for each object $E$ and endomorphism $a \in \text{Hom}^1(E,E)$ there is an object $E^a$, and morphisms $f \in \text{Hom}^0(E,E^a)$,\\ $g \in \text{Hom}^0(E^a,E)$ which are mutually inverse and satisfy
			\begin{align*}\nabla(f) = f \circ a.\end{align*}
			\item We say that  $\mathcal C$ is \emph{sufficiently twist-closed} if for each object $E$ and endomorphism $a \in \text{Hom}^1(E,E)$ of the form $a = b \circ \nabla(d)$ where $b$ and $d$ of degree $0$, there is an object $E^a$ and morphisms $f \in \text{Hom}^0(E,E^a), g \in \text{Hom}^0(E^a,E)$ which are mutually inverse and satisfy
			\begin{align*}\nabla(f) = f \circ a.\end{align*}
		\end{enumerate}
	\end{def1}
	
	\begin{example}
		$\text{Rep}^{pre}(G)$ is sufficiently twist-closed but not twist-closed. One could consider a reduced version of $\text{Rep}^{pre}(G)$ similar to the normalized bar complex of a group. This reduced cdg-category is twist-closed.
	\end{example}
	
	\begin{remark}
		\label{remarkIdentify}
		
		Let $\mathcal C$ be a cdg-category and let $E$ be an object and $a \in \text{Hom}^1(E,E)$. Assume either that $\mathcal C$ is twist-closed or sufficiently twist-closed with $a$ of the above form. Using the notation from the definition, we have that
		\begin{align*} \nabla(g) = -a \circ g\end{align*}
		and the curvature $R_{E^a}$ of $E^a$ is given by
		\begin{align*} R_{E^a} = f \circ (R_E + \nabla(a) + a^2) \circ g\end{align*}
		Using $f$ and $g$, we will identify $\text{Hom}(F,E^a) \cong \text{Hom}(F,E)$ and $\text{Hom}(E^a,F) \cong \text{Hom}(E,F)$ as graded spaces without explicitly writing out $f$ and $g$. With this notation, the complexes only differ by a twisting of $a$ in the connection $\nabla$.
		
	\end{remark}
	
	\begin{def1}
		Let $\mathcal C$ be a cdg-category.
		We define the  \emph{dg-category of flat objects} $\mathcal C^0$ to be the full subcategory of $\mathcal C$ consisting of those objects which have $0$ curvature, i.e., those $E$ for which $R_E = 0$.
	\end{def1}
	\begin{example}
		For a group $G$, the objects of the dg-category of $\text{Rep}^{0}(G) := (\text{Rep}^{pre}(G))^0$ are precisely the $\mathbb Z[G]$-modules. $H^0(\text{Rep}^{0}(G))$ is then just the category of $\mathbb Z[G]$-modules. 
	\end{example}
	
	\begin{def1}
		\label{defAssociatedDG}
		Let $\mathcal C$ be a (non-negatively graded) cdg-category. We define its associated \emph{dg-category of twisted complexes} $\text{Tw}(\mathcal C)$.
		An object of $\text{Tw}(\mathcal C)$ is a pair $(E^*, \mathbb E)$ where 
		\begin{enumerate}
			\item $E^* := (E^n)_{n \in \mathbb Z}$ is a sequence such that each $E^n$ is an object of $\mathcal C$ and $E^n = 0$ for almost all $n \in \mathbb Z$, and
			\item $\mathbb E$ is an element of $\bigoplus_{a+b = 1, a \neq 0} \text{Hom}^{a,b}(E^*, E^*)$ (defined below) such that 
			\begin{nalign} \label{EquationTrivialCurvature} \nabla(\mathbb E) + \mathbb E^2 + R_E = 0\end{nalign}
			where $R_E := (...,{R_{E^{-1}}}, {R_{E^{0}}},...)$.
		\end{enumerate}
		For objects $(E^*, \mathbb E)$ and $(F^*, \mathbb F)$, we define the morphism complexes to be 
		\begin{align*} \text{Hom}^n((E^*, \mathbb E), (F^*, \mathbb F)) := \bigoplus_{a+b = n} \text{Hom}^{a,b}(E^*, F^*) := \bigoplus_{a+b = n, p \in \mathbb Z} \text{Hom}^b(E^p, F^{p+a}).\end{align*} For $f \in \bigoplus_{p \in \mathbb Z}\text{Hom}^b(E^p, F^{p+a})$ and $g \in \bigoplus_{p \in \mathbb Z}\text{Hom}^{b'}(F^{p}, G^{p+a'})$, we define the composition using the composition $\circ'$ in $\mathcal C$ by
		\begin{align*} g \circ f := (-1)^{a' b} g \circ' f \in \bigoplus_{p \in \mathbb Z}\text{Hom}^{b+b'}(E^p, F^{p+{a+a'}}),\end{align*}
		and then we extend bilinearly.
		We endow $\text{Hom}^*((E^*, \mathbb E), (F^*, \mathbb F))$
		with the differential \begin{nalign} \label{DiffDefinition} d(f) = \nabla(f) + \mathbb F \circ f - (-1)^{|f|} f \circ \mathbb E\end{nalign} which squares to $0$ by Equation \eqref{EquationTrivialCurvature}. The composition satisfies the Leibniz rule with respect to $d$. The degree $n$ as in the above formula will be called the total degree, and we write $|f| = n$ for a morphism $f$ of total degree $n$. The degree $b$ will be referred to as the auxiliary degree, and $a$ is the internal degree.
	\end{def1}
	\begin{remark}
		In the case that the cdg-category $\mathcal C$ is twist-closed there is another definition of the category $\text{Tw}(\mathcal C)$ available using the terminology of \emph{twists}, \emph{shifts} and \emph{finite direct sums} as defined in \cite[Section 1.2]{HH2}. Namely, denote by $\text{Tw}^{cdg}(\mathcal C)$ the \emph{cdg}-category obtained by closing $\mathcal C$ under twists, shifts and finite direct sums. The \emph{dg}-category $\text{Tw}(\mathcal C)$ is then equivalent to the full subcategory of $\text{Tw}^{cdg}(\mathcal C)$ consisting of objects with vanishing curvature $R = 0$. \\ \\
		In the general case, the reason that we only close under twists $\mathbb E$ contained in \\
		$\bigoplus_{a+b = 1, a \neq 0} \text{Hom}^{a,b}(E^*, E^*)$ rather than $\bigoplus_{a+b = 1} \text{Hom}^{a,b}(E^*, E^*)$ is that certain formulas containing a-priori infinite sums are then finite (e.g., see definition \ref{DefAssociatedFunctor}).\\ \\
		Another related construction can be found in \cite{Kapranov}. The difference between their and our definition is that we allow the category to start with to be a cdg-category, rather than a dg-category. In particular, in our formulation, the components $\mathbb E_0$ and $\mathbb E_2$ of $\mathbb E$ are used to get rid of the curvature $R_E$. There is also a more general approach using curved $A_\infty$-categories in \cite{Wendy}.
	\end{remark}
	\begin{remark}
		We defined $\mathcal C^0$ as a full subcategory of $\mathcal C$. Equivalently, $\mathcal C^0$ can be viewed as the full subcategory of $\text{Tw}(\mathcal C)$ consisting of those objects that are concentrated in degree $0$.\\ \\	
		For an object $(E^*, \mathbb E)$ in $\text{Tw}(\mathcal C)$, we will often decompose 
		\begin{align*} \mathbb E =: \mathbb E_0 + \mathbb E_2 + \mathbb E_3 + ... \end{align*}
		where $\mathbb E_i \in \text{Hom}^{1-i,i}(E^*, F^*) = \bigoplus_{p \in \mathbb Z} \text{Hom}^{i}(E^p, F^{p+1-i})$. Recall that, per definition of the objects in $\mathcal C$, we have $\mathbb E_1 = 0$ and furthermore $\mathbb E_k = 0$ for all $k < 0$ as we always assume that the cdg-category $\mathcal C$ is non-negatively graded. 
		Equation \ref{EquationTrivialCurvature} decomposes into the following equations:
		\begin{nalign}
			\label{EqDecomposition}
			&\text{degree } (2,0): &&\mathbb E_0^2 = 0\\
			&\text{degree } (1,1): &&\nabla(\mathbb E_0) = 0\\
			&\text{degree } (0,2): &&\mathbb E_0 \mathbb E_2 + \mathbb E_2 \mathbb E_0 +  R_E  = 0\\
			&\text{degree } (-1,3): &&\nabla(\mathbb E_2) + \mathbb E_0 \mathbb E_3 + \mathbb E_3 \mathbb E_0 = 0\\
			&...\\
			&\text{degree } (2-i,i): &&\nabla(\mathbb E_{i-1}) + \sum_{k = 0}^i \mathbb E_k \mathbb E_{i-k} = 0 \text{ } (\text{for } i \neq 2).
		\end{nalign}
	\end{remark}
	\begin{example}
		We again consider the cdg-category $\text{Rep}^{pre}(G)$ and its associated dg-category $\text{Rep}^{\infty}(G) := \text{Tw}(\text{Rep}^{pre}(G))$. 
		Let $(E^*, \mathbb E)$ and $(F^*, \mathbb F)$ be objects of $\text{Rep}^{\infty}(G)$.
		The components of $d$ on $\text{Hom}(E,F)$ act as depicted in the following diagram
		\[\begin{tikzcd}[row sep=large, column sep=large]
			{\text{Hom}_{\text{Set}}(G^b, \text{Hom}_{\mathbb Z}(E^*, F^{*+a}))} \arrow[d, "\nabla"] \arrow[rr, "\mathbb F_0 \circ (-) \pm (-)\circ \mathbb E_0"]   &  & {\text{Hom}_{\text{Set}}(G^b, \text{Hom}_{\mathbb Z}(E^*, F^{*+a+1}))} \arrow[d, "\nabla"] \arrow[to=lldd, near end, "\mathbb F_2 \circ (-) \pm (-)\circ \mathbb E_2"]  \\
			{\text{Hom}_{\text{Set}}(G^{b+1}, \text{Hom}_{\mathbb Z}(E^*, F^{*+a}))} \arrow[d, "\nabla"] \arrow[rr, crossing over, "\mathbb F_0 \circ (-) \pm (-)\circ \mathbb E_0"] &  & {\text{Hom}_{\text{Set}}(G^{b+1}, \text{Hom}_{\mathbb Z}(E^*, F^{*+a+1}))} \arrow[d, "\nabla"]                                                          \\
			{\text{Hom}_{\text{Set}}(G^{b+2}, \text{Hom}_{\mathbb Z}(E^*, F^{*+a}))} 
			\arrow[rr, "\mathbb F_0 \circ (-) \pm (-)\circ \mathbb E_0"]                     &  & {\text{Hom}_{\text{Set}}(G^{b+2}, \text{Hom}_{\mathbb Z}(E^*, F^{*+a+1}))}          .                                                                   
		\end{tikzcd}
		\]
		In general the expression $\mathbb F_i \circ (-) \pm (-) \circ \mathbb E_i$ moves $1-i$ steps to the right and $i$ steps down.
		
	\end{example}

	\subsection{$A_\infty$-functors and natural transformations}
	\begin{def1}
		\label{Ainfty}
		Let $\mathcal C$ and $\mathcal D$ be cdg-categories. An \emph{$A_\infty$-functor} $\mathcal F : \mathcal C  \to \mathcal D$ assigns to every object $E$ of $\mathcal C$ an object $\mathcal F(E)$ of $\mathcal D$ and to every sequence $f_1 \in \text{Hom}_{\mathcal C}^{|f_1|}(E_2,E_1),...,f_n \in \text{Hom}_{\mathcal C}^{|f_n|}(E_{n+1},E_n), n \geq 1$ of composable morphisms in $\mathcal C$ a morphism
		\begin{align*} \mathcal F(f_1,...,f_n) \in \text{Hom}_{\mathcal D}^{|f_1|+...+|f_n|-n+1}(\mathcal F E_{n+1},\mathcal F E_1)\end{align*}
		such that \begin{enumerate}
			\item $\mathcal F$ is multilinear,
			\item $\mathcal F$ is compatible with the identity, i.e., $\mathcal F(\mathbbm 1_{E_k})= \mathbbm 1_{\mathcal F(E_k)}$, $\mathcal F_n(...,\mathbbm 1,...) = 0$ for all $n \geq 2$
			\item $\mathcal F$ is compatible with the curvatures, i.e., $\mathcal F(R_{E_k}) = R_{\mathcal F(E_k)}$, and
			\item the equations \begin{nalign}
				\label{AinfinityEquation}
				&(-1)^{|f_1|+...+|f_n|+1-n}\nabla(\mathcal F(f_1,...,f_n)) \\
				&\quad+ \sum_{k = 1}^{n-1} (-1)^{|f_{k+1}| + ... + |f_n|-n+k+1}\mathcal F(f_1,...,f_k)\mathcal F(f_{k+1},...,f_n) =\\
				&= \sum_{k = 1}^n (-1)^{|f_k|+...+|f_n| - n+k}\mathcal F(f_1,...,\nabla f_k,...,f_n)\\
				&\quad+ \sum_{k = 1}^n (-1)^{|f_{k+1}|+...+|f_n| - n+k-1}\mathcal F(f_1,...,f_k \circ f_{k+1},...,f_n)\\
				&\quad+ \sum_{k = 1}^{n+1} (-1)^{|f_k|+...+|f_n| - n+k}\mathcal F(f_1,...,f_{k-1},R_{E_k},f_k,...,f_n)
			\end{nalign}
			hold.
		\end{enumerate}
		We say that $\mathcal F$ is \emph{proper} if for each $n \geq 2$ we have
		\begin{align*} \mathcal F(f_1,...,f_n) = 0\end{align*}
		if there is some $i = 1,...,n$ such that $|f_i| = 0$. \\ \\
		We call $\mathcal F$ a \emph{cdg-functor} if 
		\begin{align*} \mathcal F(f_1,...,f_n) = 0\end{align*} for all $n \geq 2$.
	\end{def1}
	We will only consider proper $A_\infty$-functors. The reason for that is that we need this condition to ensure that certain sums are finite. If one were to deal with more general $A_\infty$-functors (possibly also allowing the cdg-categories to have negative degrees), we would need to put a topology on the morphism spaces to be able to talk about convergence. It should also be mentioned that more generally, a curved $A_\infty$-functor between cdg-categories may not strictly preserve curvature elements, but rather involve the additional data of \emph{change-of-connection} elements $\mathcal F_E \in \text{Hom}_{\mathcal D}^1(\mathcal F E, \mathcal F E)$ which measure the failure of $\mathcal F$ to be compatible with curvatures (in the case of cdg-functors, see e.g., \cite[Section 1.3]{HH2}). Since all functors appearing in the present paper will preserve curvature elements, we will only deal with such functors.
	\begin{def1}
		\label{Ainftynatural}
		Let $\mathcal F, \mathcal G : \mathcal C \to \mathcal D$ be $A_\infty$-functors of cdg-categories. An \emph{$A_\infty$-natural transformation}
		\begin{align*} \eta : \mathcal F \Rightarrow \mathcal G\end{align*}
		from $\mathcal F$ to $\mathcal G$ assigns to each object $E$ of $\mathcal C$ a morphism $\eta_E \in \text{Hom}_{\mathcal D}^{0}(\mathcal F(E),\mathcal G(E))$ and to every sequence $f_1 \in \text{Hom}_{\mathcal C}^{|f_1|}(E_2,E_1),...,f_n \in \text{Hom}_{\mathcal C}^{|f_n|}(E_{n+1},E_n), n \geq 1$ of composable morphisms in $\mathcal C$ a morphism
		\begin{align*} \eta(f_1,...,f_n) \in \text{Hom}_{\mathcal D}^{|f_1|+...+|f_n|-n}(\mathcal F(E_{n+1}),\mathcal G(E_1))\end{align*} satisfying the properties that
		\begin{enumerate}
			\item $\eta$ is multilinear, and
			\item the following identity holds
			\begin{nalign}
				&(-1)^{|f_1|+...+|f_n|-n+1} \nabla(\eta(f_1,...,f_n))\\
				&\quad+ \sum_{k = 0}^{n-1} \eta(f_1,...,f_k) \mathcal F(f_{k+1},...,f_n)\\
				&\quad+ \sum_{k = 1}^{n} (-1)^{|f_{k+1}|+...+|f_n|-n+k+1}\mathcal G(f_1,...,f_k) \eta(f_{k+1},...,f_n)\\
				&= \sum_{k = 1}^n (-1)^{|f_k|+...+|f_n|-n+k} \eta(f_1,...,\nabla f_k,...,f_n)\\
				&\quad+ \sum_{k = 1}^n (-1)^{|f_{k+1}|+...+|f_n|-n+k+1} \eta(f_1,...,f_k \circ f_{k+1},...,f_n)\\
				&\quad+ \sum_{k = 1}^n (-1)^{|f_k|+...+|f_n|-n+k} \eta(f_1,...,R_{E_k},f_k,...,f_n)\\
			\end{nalign}
			for all $n \geq 0$.
		\end{enumerate}
		We say that $\eta$ is \emph{proper} if for each $n \geq 2$ we have
		\begin{align*} \eta(f_1,...,f_n) = 0 \end{align*}
		if there is some $i = 1,...,n$ such that $|f_i| = 0$.
		In the above expression, we set $\eta(f_1,...,f_n) := \eta_E$ when $n = 0$. The equation for $n = 0$ reads
		\begin{align*} \nabla(\eta_E) =  \eta(R_{E}).\end{align*}
	\end{def1}
	
	\begin{def1}
		Let $\mathcal F, \mathcal G : \mathcal C \to \mathcal D$ be $A_\infty$-functors of cdg-categories and let 
		\begin{align*} \eta : \mathcal F \Rightarrow \mathcal G, \zeta : \mathcal G \Rightarrow \mathcal H\end{align*}
		be $A_\infty$-natural transformations of $A_\infty$-functors. We define their composite $\zeta \circ \eta : \mathcal F \Rightarrow \mathcal H$. For objects $E$ of $\mathcal C$, we set
		\begin{align*} (\zeta \circ \eta)_{E} :=  \zeta_E \circ \eta_E.\end{align*}
		For composable morphisms $f_1,...,f_n$ of $\mathcal C$, we define
		\begin{align*} (\zeta \circ \eta)(f_1,...,f_n) := \sum_{k = 0}^n \zeta(f_1,...,f_k)\eta(f_{k+1},...,f_n)\end{align*}
		One may check that this indeed defines an $A_\infty$-natural transformation. 
	\end{def1}
	
	\begin{def1}
		\label{DefAssociatedFunctor}
		Let $\mathcal F : \mathcal C \to \mathcal D$ be a proper $A_\infty$-functor of cdg-categories. We define how $\mathcal F$ induces an ordinary $A_\infty$-functor of dg-categories $\text{Tw}(\mathcal F) : \text{Tw}(\mathcal C) \to \text{Tw}(\mathcal D)$. To do this, we first extend $\mathcal F$ linearly to composable morphisms in $\text{Tw}(\mathcal C)$. We then see that we obtain the $A_\infty$-relations as in \ref{Ainfty}, where the composition is given by $\circ'$ (i.e., the linear extension of the composition in $\mathcal C$ and $\mathcal D$) and the degrees in the signs are the internal degrees instead of the total degrees. To obtain an equation involving the total degrees and the composition in $\text{Tw}(\mathcal C)$ and $\text{Tw}(\mathcal D)$ we define
		\begin{align*} \mathcal F'(f_1,...,f_n) := (-1)^{p_1(q_2+...+q_n+n) + p_2 (q_3+...+q_n+n-1)  +...+p_{n}}   \mathcal F(f_1,...,f_n)\end{align*}
		for $f_1 \in \text{Hom}^{p_1,q_1}((E_2)^*,(E_1)^*),...,f_n \in \text{Hom}^{p_n,q_n}((E_{n+1})^*,(E_n)^*)$.
		Then $\mathcal F'$ fulfills the relations as in \ref{Ainfty} with the composition $\circ$ of $\text{Tw}(\mathcal C)$ and $\text{Tw}(\mathcal D)$ and the degrees are the total degrees. Now let $(E^*, \mathbb E)$ be an object of $\mathcal C$. Then we define $\text{Tw}(\mathcal F)(E^*) := (\mathcal F(E^n))_{n \in \mathbb Z}$ and the corresponding endomorphism $\mathcal F_*(\mathbb E)$ is given by
		\begin{align*} \text{Tw}(\mathcal F)_*(\mathbb E) := \sum_{k \geq 1}(\mathcal F)'(\mathbb E^{\otimes k}).\end{align*} On morphisms, we now set
		\begin{align*} \text{Tw}(\mathcal F)(f_1,...,f_n) := \sum_{k_1,...,k_{n+1} \geq 0} (\mathcal F)'(\mathbb E_1^{\otimes k_1},f_1,\mathbb E_2^{\otimes k_2},f_2,...,\mathbb E_{n}^{\otimes k_n},f_n,\mathbb E_{n+1}^{\otimes k_{n+1}}).\end{align*}
		Note that, by the requirements that the elements $\mathbb E_k$ have no parts that have internal degree $0$, that $\mathcal F$ is proper, and our boundedness assumptions on the sequences $E^*$, both of these sums are in fact finite.
	\end{def1}
	\begin{lemma}
		\label{inducedFunctorLemma}
		In the setting of the previous definition $\text{Tw}(\mathcal F)$ indeed defines an $A_{\infty}$-functor of dg-categories, in particular for $f_1 \in \text{Hom}_{\mathcal C}^{|f_1|}(E^*_2,E^*_1),...,f_n \in \text{Hom}_{\mathcal C}^{|f_n|}(E^*_{n+1},E^*_n)$, $n \geq 1$ the identity
		\begin{nalign}
			&(-1)^{|f_1|+...+|f_n|+1-n}d(\text{Tw}(\mathcal F)(f_1,...,f_n)) \\
			&\quad+ \sum_{k = 1}^{n-1} (-1)^{|f_{k+1}| + ... + |f_n|-n+k+1}\text{Tw}(\mathcal F)(f_1,...,f_k)\text{Tw}(\mathcal F)(f_{k+1},...,f_n) \\
			&= \sum_{k = 1}^n (-1)^{|f_k|+...+|f_n| - n+k}\text{Tw}(\mathcal F)(f_1,...,df_k,...,f_n)\\
			&\quad+ \sum_{k = 1}^n (-1)^{|f_{k+1}|+...+|f_n| - n+k-1}\text{Tw}(\mathcal F)(f_1,...,f_k \circ f_{k+1},...,f_n)
		\end{nalign}
		holds.
		\begin{proof}
			The proof ammounts to a straightforward (but rather lengthy) calculation. The identities one has to use are \ref{EquationTrivialCurvature}, \ref{DiffDefinition}, and \ref{AinfinityEquation} (for $\mathcal F'$ instead of $\mathcal F$).
		\end{proof}
	\end{lemma}
	
	\begin{def1}
		Let $\eta : \mathcal F \to \mathcal G$ be a proper $A_\infty$-natural transformation of proper $A_\infty$-functors between cdg-categories $C$ and $D$. We want to define an $A_\infty$-natural transformation $\text{Tw}(\eta) : \text{Tw}(\mathcal F) \Rightarrow \text{Tw}(\mathcal G)$ of induced $A_\infty$-functors between dg-categories. Similar to before, we first extend $\eta$ to composable morphisms in $\mathcal C$ linearly. Then we define
		\begin{align*} \eta'(f_1,...,f_n) := (-1)^{p_1(q_2+...+q_n+n-1) + p_2 (q_3+...+q_n+n-1)  +...+p_{n-1}} \eta(f_1,...,f_n)\end{align*}
		for $f_1 \in \text{Hom}^{p_1,q_1}((E_2)^*,(E_1)^*),...,f_n \in \text{Hom}^{p_n,q_n}((E_{n+1})^*,(E_n)^*)$. One can check that the $A_{\infty}$-equation as in Definition \ref{Ainftynatural} holds with the composition $\circ$ of $\text{Tw}(\mathcal C)$ and $\text{Tw}(\mathcal D)$, and the degrees are the total degrees. On objects $(E^*, \mathbb E)$, we now set
		\begin{align*} \text{Tw}(\eta)_{(E^*, \mathbb E)} := \eta_{E^*} + \sum_{k \geq 1} \eta'(\mathbb E^{\otimes k}) \end{align*}
		and for composable morphisms, we set
		\begin{align*} \text{Tw}(\eta)(f_1,...,f_n) := \sum_{k_1,...,k_{n+1} \geq 0} \eta'(\mathbb E_1^{\otimes k_1},f_1,\mathbb E_2^{\otimes k_2},f_2,...,\mathbb E_{n}^{\otimes k_n},f_n,\mathbb E_{n+1}^{\otimes k_{n+1}}). \end{align*}
		As before, it follows from the properness assumption that all the occurring sums are finite.
	\end{def1}
	
	\begin{lemma}
		In the setting of the previous definition, we have that $\text{Tw}(\eta) : \text{Tw}(\mathcal F) \Rightarrow \text{Tw}(\mathcal G)$ indeed defines an $A_\infty$-natural transformation, i.e., the identities 
		\begin{nalign}
			&(-1)^{|f_1|+...+|f_n|-n+1} d(\text{Tw}(\eta)(f_1,...,f_n))\\
			&\quad+ \sum_{k = 0}^{n-1} \text{Tw}(\eta)(f_1,...,f_k) \text{Tw}(\mathcal F)(f_{k+1},...,f_n)\\
			&\quad+ \sum_{k = 1}^{n} (-1)^{|f_{k+1}|+...+|f_n|-n+k+1}\text{Tw}(\mathcal G)(f_1,...,f_k) \text{Tw}(\eta)(f_{k+1},...,f_n)\\
			&= \sum_{k = 1}^n (-1)^{|f_k|+...+|f_n|-n+k} \text{Tw}(\eta)(f_1,...,d f_k,...,f_n)\\
			&\quad+ \sum_{k = 1}^n (-1)^{|f_{k+1}|+...+|f_n|-n+k+1} \text{Tw}(\eta)(f_1,...,f_k \circ f_{k+1},...,f_n)\\
		\end{nalign}
		hold for composable morphisms $f_1,...,f_n$.
	\end{lemma}
	
	\begin{remark}
		$A_\infty$-functors and $A_\infty$-natural transformations of cdg-categories (or more generally of curved $A_\infty$-categories) can be composed (see \cite{Armstrong}). We remark that this composition is compatible with taking the induced functor and induced natural transformations of dg-categories.
	\end{remark}
	
	\subsection{$A_\infty$-quasi equivalences}
	\label{quasiEqSection}
	Suppose we are given a proper $A_\infty$-functor $\mathcal F : \mathcal C \to \mathcal D$ of cdg-categories such that the restricted functor $\mathcal F^0 : \mathcal C^0 \to \mathcal D^0$ is a quasi-equivalence. The goal of this subsection is to find conditions on the categories $\mathcal C$ and $\mathcal D$ such that also $\text{Tw}(\mathcal F) : \text{Tw}(\mathcal C) \to \text{Tw}(\mathcal D)$ is a quasi-equivalence.
	\begin{def1}
		Let $\mathcal C$ be a cdg-category. Note that by definition $\mathcal C$ is an additive category. 
		\begin{enumerate}
			\item Let $Z^0(\mathcal C)$ be the category consisting of the same objects as $\mathcal C$ and the morphisms are closed degree $0$ morphisms of $\mathcal C$.
			\item We say that $\mathcal C$ is \emph{abelian} if $Z^0(\mathcal C)$ is an abelian category.
			\item Now assume that $\mathcal C$ is abelian. Then we can consider short exact sequences 
			\begin{align*} 0 \to E \to F \to G \to 0\end{align*}
			in $\mathcal C$. We always understand that the occuring morphisms are closed and have degree $0$, i.e., the short exact sequence is a short exact sequence in $Z^0(\mathcal C)$. We say that $\mathcal C$ is \emph{split} if every short exact sequence in $\mathcal C$ splits. Here, the splittings are required to have degree $0$ but they are \textbf{not} assumed to be closed. We stress that this is \textbf{not} saying that $Z^0(\mathcal C)$ is split (which is not the case for our applications).
			\item Note that an $A_\infty$-functor of cdg-categories $\mathcal F : \mathcal C  \to \mathcal D$ naturally induces an ordinary $A_\infty$-functor $\mathcal F^0 : \mathcal C^0 \to \mathcal D^0$ of dg-categories.
		\end{enumerate}
	\end{def1}
	
	\begin{def1}
		\label{filtration}
		Let $\mathcal C$ be a cdg-category. We define a decreasing filtration on the morphism complexes of $\text{Tw}(\mathcal C)$ by
		\begin{align*} F^b \text{Hom}^n(E,F) := \bigoplus_{a+b' = n, b' \geq b} \text{Hom}^{a,b'}(E,F).\end{align*}
		Since the differential does not decrease the auxiliary degree, this is indeed a filtration of complexes. We obtain an associated spectral sequence whose first three pages are given by 
		\begin{nalign}
			\label{SpectralSequenceEq}
			&E_0^{p,q}(E,F) = \text{Hom}^{p,q}(E,F)\\
			&E_1^{p,q}(E,F) = H^p(\text{Hom}^{*,q}(E,F), d^0)\\
			&E_2^{p,q}(E,F) = H^q(H^p(\text{Hom}^{*,*}(E,F), d^0), \nabla)
		\end{nalign}
		where $d^0(f) = \mathbb F_0 \circ f - (-1)^{|f|} f \circ \mathbb E_0$.
		Moreover, note that it is exhaustive since ${F^0 \text{Hom}^n(E,F) = \text{Hom}^n(E,F)}$ and it is bounded above since for any $n$,
		${\text{Hom}^{n-b,b}(E,F) = 0}$ for $b \gg 0$. Therefore, the spectral sequence converges
		\begin{nalign}
			E_0^{p,q}(E,F) \Rightarrow H^n(\text{Hom}^*(E,F), d),
		\end{nalign}
		by Theorem 5.5.1 in \cite{Weibel}.
		Now let $\mathcal F : \mathcal C \to \mathcal D$ be a proper $A_\infty$-functor of cdg-categories. By the properness assumption it follows that 
		\begin{align*} \text{Tw}(\mathcal F)(E^*, F^*) : \text{Hom}^*(E^*, F^*) \to \text{Hom}^*( \mathcal F E^*,  \mathcal F F^*) \end{align*}
		is compatible with the filtrations.
	\end{def1}
	The following Proposition is the key ingredient for the proof of Theorem \ref{EquivalenceTheorem}.
	\begin{proposition}
		\label{trivialDiff}
		Let $\mathcal C$ be a sufficiently twist-closed cdg-category which is split. Then any object $(E^*, \mathbb E)$ of $\text{Tw}(\mathcal C)$ is homotopy equivalent to an object $(H^*, \mathbb H)$ such that $\mathbb H^0 = 0$, i.e., the auxiliary degree of every part of $\mathbb H$ is positive. In particular, $R_{H^i} = 0$ for all $i$, that is $H^i \in \mathcal C^0$.
		\begin{proof}
			Let $(E^*, \mathbb E)$ be an object of $\text{Tw}(\mathcal C)$. We decompose $\mathbb E = \mathbb E_0 + \mathbb E_2 + \mathbb E_3 +...$ where each $\mathbb E_i$ has auxiliary degree $i$. As in equation \eqref{EqDecomposition}, we obtain
			\begin{align*}  (\mathbb E_0)^2 = 0 \text{ and } \nabla(\mathbb E_0) = 0\end{align*}
			For the moment we forget about all the higher differentials $\nabla + \mathbb E_2 +...$ and only restrict our attention to $\mathbb E_0$. Let $H^*$ be the cohomology complex of $E^*$ with respect to $\mathbb E_0$. Using that $\mathcal C$ is split, we can construct (possibly non $\nabla$-closed) degree $0$ morphisms of complexes $g^0 : (E^*, \mathbb E_0) \to (H^*, 0)$ and $f^0 : (H^*, 0) \to (E^*, \mathbb E_0)$ such that $g^0 \circ f^0 = \text{id}_{H^*}$ and $f^0 \circ g^0$ is homotopic to the identity on $E^*$ (with respect to $\mathbb E_0$). By the next proposition, we can find a differential $\mathbb H$ on $H^*$ with $\mathbb H_0 = 0$ and an extension of $f^0$ to a homotopy equivalence $f$ in the dg-category $\text{Tw}(\mathcal C)$.
		\end{proof}
	\end{proposition}
	The next Proposition can be understood as a Homotopy Transfer Theorem, i.e, it allows us to transport the differential on $(E^*, \mathbb E)$ to a differential $\mathbb F$ on $F^*$ under mild assumptions.
	\begin{proposition}
		\label{ExtendLemma}
		Let $\mathcal C$ be a cdg-category. Let $(E^*, \mathbb E)$ be some object in $\text{Tw}(\mathcal C)$ and let $F^*$ be a bounded sequence of objects in $\mathcal C$ with a $\nabla$-closed morphism $\mathbb F_0 \in \text{Hom}^0(F^i,F^{i+1})$ for each $i$ such that $ (\mathbb F_0)^2 = 0$. Moreover, let $f^0 : (F^*, \mathbb F_0) \to (E^*, \mathbb E_0) $ be a (not necessarily $\nabla$-closed) morphism of complexes of bidegree $(0,0)$. Assume either that 
		\begin{enumerate}
			\item $\mathcal C$ is sufficiently twist-closed and $f^0$ has a homotopy inverse $g^0 : (E^*, \mathbb E_0) \to (F^*, \mathbb F_0)$, or
			\item $\mathcal C$ is twist-closed and the map \begin{align*}(f_*^0)^i : \text{Hom}^{*, i}((F^*, \mathbb F_0), (F^*, \mathbb F_0)) \to \text{Hom}^{*, i}((F^*, \mathbb F_0), (E^*, \mathbb E_0))\end{align*} is a quasi-isomorphism of complexes for every $i \in \mathbb Z$.
		\end{enumerate}
		Then there are elements $a_i \in \text{Hom}^1(F^i,F^i)$, an object $(((F^i)^{a_i})_{i \in \mathbb Z}, \mathbb F)$ of $\mathcal C$ extending $\mathbb F_0$, and a closed morphism $f \in \text{Hom}^0((F^a)^*, E^*)$ in $\text{Tw}(\mathcal C)$ extending $f^0$. In the first case $f$ is a homotopy equivalence.
		\begin{proof}
			In both cases we have that each $(f_*^0)^i$ is a quasi-isomorphism. The only difference in the proof between 1 and 2 is the fact that we can choose the elements $a_i$ to be of the required form in the case that $\mathcal C$ is only sufficiently twist-closed. We begin by constructing $a = (a_i)_{i \in \mathbb Z}$. By assumption, we have $\mathbb E_0 f^0 - f^0 \mathbb F_0 = 0$. Furthermore, since $\nabla_F(\mathbb F_0) = 0 = \nabla_E(\mathbb E_0)$ we have $\mathbb E_0 \nabla(f^0) + \nabla(f^0) \mathbb F_0 = 0$. As $(f_*^0)^1$ is a quasi-isomorphism, we obtain a closed $a \in  \text{Hom}^{0, 1}((F^*, \mathbb F_0), (F^*, \mathbb F_0))$ and a map $f^1 \in \text{Hom}^{-1, 1}((F^*, \mathbb F_0), (E^*, \mathbb E_0))$ such that $\nabla(f^0) - f^0 \circ a + \mathbb E_0 f^1 - f^1 \mathbb F_0 = 0$. We remark that in the first case we can choose $a = g^0 \circ \nabla(f^0)$, thus being of the required form for the sufficiently twist-closed property. Let $\nabla'$ denote the by-$a_i$-twisted connection (compare Remark \ref{remarkIdentify}) for each $i$. Note that since $a$ is closed, we still have $\nabla'(\mathbb F_0) = 0$. Together with the fact that $f^0$ is $(\mathbb E_0$, $\mathbb F_0)$ closed, we can summarize
			\begin{align*} \mathbb E  f^{\leq 1} - f^{\leq 1} \mathbb F_{\leq 1} + \nabla'(f^{\leq 1}) = 0 \text{ Mod } F^2.\end{align*}
			Here, $f^{\leq i} := f^0 + ... + f^i$ and $F^*$ is the filtration as in Definition \ref{filtration}. We also obtain
			\begin{align*} \nabla'(\mathbb F_{\leq 1}) + \mathbb F_{\leq 1} \circ \mathbb F_{\leq 1} + R_{F^a} = 0 \text{ Mod } F^2\end{align*}
			(recall that $\mathbb F_1 = \mathbb E_1 = 0$). We now proceed inductively. Assume for a given $i \geq 1$ that the equations
			\begin{enumerate}
				\item \begin{align*} \nabla'(\mathbb F_{\leq i}) + \mathbb F_{\leq i} \mathbb F_{\leq i} + R_{F^a} = 0 \text{ Mod } F^{i+1}\end{align*}
				\item \begin{align*} \mathbb E  f^{\leq i} - f^{\leq i} \mathbb F_{\leq i} + \nabla'(f^{\leq i}) = 0 \text{ Mod } F^{i+1}\end{align*}
			\end{enumerate}
			hold. We will then construct $\mathbb F_{\leq i+1} = \mathbb F_{\leq i} + \mathbb F_{i+1}$ and $f^{\leq i+1} = f^{\leq i} + f^{i+1}$ such that the above equations hold for modulo $F^{i+2}$ instead of $F^{i+1}$ .  We consider the failure of the above equations to hold modulo $F^{i+2}$. Thus, we set
			\begin{align*} P := \nabla'(\mathbb F_{\leq i}) + \mathbb F_{\leq i} \mathbb F_{\leq i} + R_{F^a}  \text{ Mod } F^{i+2}\end{align*}
			and 
			\begin{align*} Q := \mathbb E  f^{\leq i} - f^{\leq i} \mathbb F_{\leq i} + \nabla'(f^{\leq i}) \text{ Mod } F^{i+2}\end{align*}
			We first show that $P$ is closed modulo $F^{i+2}$.
			\begin{nalign}
				\mathbb F_0 P - P \mathbb F_0 &= \mathbb F_0 ( \nabla'(\mathbb F_{\leq i}) + \mathbb F_{\leq i} \mathbb F_{\leq i} + R_{F^a}) - ( \nabla'(\mathbb F_{\leq i}) + \mathbb F_{\leq i} \mathbb F_{\leq i} + R_{F^a}) \mathbb F_0 \\
				&=  -\nabla'( \mathbb F_0 \mathbb F_{\leq i} + \mathbb F_{\leq i} \mathbb F_0) + \mathbb F_0 \mathbb F_{\leq i} \mathbb F_{\leq i} - \mathbb F_{\leq i} \mathbb F_{\leq i} \mathbb F_0
			\end{nalign}
			Since we are computing modulo $F^{i+2}$ and $\nabla'$ increases the degree of the filtration by $1$, we use the induction hypothesis to calculate:
			\begin{nalign}
				&-\nabla'( \mathbb F_0 \mathbb F_{\leq i} + \mathbb F_{\leq i} \mathbb F_0) + \mathbb F_0 \mathbb F_{\leq i} \mathbb F_{\leq i} -\mathbb F_{\leq i} \mathbb F_{\leq i} \mathbb F_0 \\
				&= \nabla'( \mathbb F_{ \geq 1, \leq i} \mathbb F_{ \geq 1, \leq i} + R_{F_a} + \nabla'(\mathbb F_{\leq i} )) + \mathbb F_0 \mathbb F_{\leq i} \mathbb F_{\leq i} -\mathbb F_{\leq i} \mathbb F_{\leq i} \mathbb F_0\\
				&= (-\mathbb F_{\leq 1} \mathbb F_{\leq 1} - R_{F^a}) \mathbb F_{ \geq 1, \leq i}- \mathbb F_{ \geq 1, \leq i} (-\mathbb F_{\leq 1} \mathbb F_{\leq 1} - R_{F^a}) + (\nabla')^2 \mathbb F_{\leq i}\\
				&\quad+  \mathbb F_0 \mathbb F_{\leq i} \mathbb F_{\leq i} -\mathbb F_{\leq i} \mathbb F_{\leq i} \mathbb F_0 \\
				&= 0.
			\end{nalign}
			Thus, $\mathbb F_0 P - P \mathbb F_0 = 0$. Next, we want to compute the differential of $Q$ in $F^{i+1}$. We begin by the following computation modulo $F^{i+2}$ (using similar arguments as before)
			\begin{nalign}
				&\nabla'(- \mathbb E ^0 f^{\leq i} + f^{\leq i} \mathbb F_0) \\
				&= \nabla'(\mathbb E_{\geq 1} f^{\leq i} - f^{\leq i} \mathbb F_{\geq 1, \leq i} + \nabla'(f^{\leq i})) \\
				&= (- \mathbb E_2 - R_E) f^{\leq i} - \mathbb E_{\geq 1}(-\mathbb E f^{\leq i} + f^{\leq i} \mathbb F_{\leq i}) - (-\mathbb E f^{\leq i} + f^{\leq i} \mathbb F_{\leq i})\mathbb F_{\geq 1, \leq i} \\
				&\quad- f^{\leq i, \geq 1} (- \mathbb F_{\leq i} \mathbb F_{\leq i} - R_{F^a}) - f^0 \nabla'(\mathbb F_{\leq i})  + (\nabla')^{2} f^{\leq i}\\
				&= -\mathbb E_0(\mathbb E f^{\leq i} - f^{\leq i} \mathbb F_{\leq i}) - (\mathbb E f^{\leq i} - f^{\leq i} \mathbb F_{\leq i}) \mathbb F_0 - f^0 P
			\end{nalign}
			which we can rewrite as
			\begin{align*} \mathbb E_0 Q + Q \mathbb F_0 + f^0 P = 0 \text{ Mod } F^{i+2}.\end{align*}
			Now consider the (i+1)-th components $P^{i+1}$ and $Q^{i+1}$ so that $\mathbb E_0 Q^{i+1} + Q^{i+1}\mathbb F_0 + f^0P = 0$ and $\mathbb F_0 P^{i+1} - P^{i+1} \mathbb F_{0} = 0$. We consider the pair $(P,Q)$ as an element of the cone of $(f_*^0)^{i+1}$ which is acyclic since $(f_*^0)^{i+1}$ is a quasi-isomorphism. The above equations precisely mean that $(P^{i+1},Q^{i+1})$ is a cycle. Thus, there is a pair $(f^{i+1}, \mathbb F_{i+1})$ such that 
			\begin{align*} \mathbb F_0 \mathbb F_{i+1} + \mathbb F_{i+1} \mathbb F_0 = -P^{i+1}\end{align*}
			and 
			\begin{align*} \mathbb E_0 f^{i+1} - f^{i+1} \mathbb F_{0} - f^0 \mathbb F_{i+1} = -Q^{i+1}.\end{align*}
			Now letting $f^{\leq i+1} := f^{\leq i} + f^{i+1}$ and $\mathbb F_{\leq i+1} := \mathbb F_{\leq i} + \mathbb F_{i+1}$ we conclude the induction.
			Since the filtration is bounded above, it follows that the above process terminates for some $i \gg 0$, and we obtain the two equalities
			\begin{nalign}
				&\nabla'(\mathbb F) + \mathbb F \mathbb F + R_{F^a} = 0\\
				&\mathbb E f - f \mathbb F + \nabla'(f) = 0\
			\end{nalign}
			where $\mathbb F = \mathbb F_0 + ...$ and $f = f^0 + ...$.
			The first equation means that we have constructed an object $(F^a, \mathbb F)$ of $\text{Tw}(\mathcal C)$, and the second equation implies that $f$ is a closed morphism $f \in \text{Hom}^0((F^a)^*, E^*)$. By construction $\mathbb F$ extends $\mathbb F_0$ and $f$ extends $f^0$, and by assumption $f^0 : (F^a)^* \to E^*$ is a homotopy equivalence. To show the remaining statement in the case that $f^0$ has a homotopy inverse, we can directly apply the next Lemma.
			
		\end{proof}
	\end{proposition}
	\begin{lemma}
		\label{HomotopyCriterion}
		Let $\mathcal C$ be a cdg-category. Let $f \in \text{Hom}^0(F^*, E^*)$ be a closed morphism in $\text{Tw}(\mathcal C)$ of total degree $0$. Then following are equivalent
		\begin{enumerate}
			\item $f^0 : (F^*, \mathbb F_0) \to (E^*, \mathbb E_0)$ is a homotopy equivalence of complexes.
			\item $f$ is a homotopy equivalence in $\text{Tw}(\mathcal C)$.
		\end{enumerate}
		\begin{proof}
			It is clear that the second statement implies the first.
			Thus let $f$ be $d$-closed such that $f^0$ is a homotopy equivalence. To prove the claim, it suffices to show that the map 
			\begin{align*} f^* : \text{Hom}^*_{\text{Tw}(\mathcal C)}(E^*,G^*) \to \text{Hom}^*_{\text{Tw}(\mathcal C)}(F^*,G^*)\end{align*}
			is a quasi-isomorphism for every object $G^*$ of $\text{Tw}(\mathcal C)$. On both complexes, we consider the filtrations from \ref{filtration} and note that $f^*$ is compatible with these filtrations as the auxiliary degree of any morphism (in particular of $f$) is non-negative. But since $f^0$ is a homotopy equivalence, the induced map on the $E_1$ page (cf. equation \eqref{SpectralSequenceEq}) is an isomorphism. We have already seen that the spectral sequences converge, so we obtain that $f^*$ is a quasi-isomorphism. 
		\end{proof}
	\end{lemma}
	\begin{proposition}
		\label{PropQFF}
		Let $\mathcal F : \mathcal C \to \mathcal D$ be a proper $A_\infty$-functor between cdg-categories such that $\mathcal C$ is sufficiently twist-closed and split. Assume that the induced $A_\infty$-functor $\mathcal F^{0} : \mathcal C^{0} \to \mathcal D^{0}$ is quasi-fully faithful. Then $\text{Tw}(\mathcal F) : \text{Tw}(\mathcal C) \to \text{Tw}(\mathcal D)$ is quasi-fully faithful.
		\begin{proof}
			By our assumptions, we can apply \ref{trivialDiff}. Therefore, it is enough to show that for objects $(E^*, \mathbb E)$ and $(F^*, \mathbb F)$ with $\mathbb E_0 = 0$ and $\mathbb F_0 = 0$, we have that the induced map
			\begin{align*} \mathcal F : \text{Hom}^*(E^*, F^*) \to \text{Hom}^*(\mathcal F E^*, \mathcal F F^*)  \end{align*}
			is a quasi-isomorphism. Furthermore, note that by the properness assumption of $\mathcal F$, we also obtain that $(\mathcal F_*(\mathbb E))^0 = 0 = (\mathcal F_*(\mathbb F))^0$. We again use the spectral sequence from \ref{filtration} and immediately obtain from our assumption that $\mathcal F^0$ is quasi-fully faithful, that the induced map on the $E_2$ page is an isomorphism. Since the spectral sequences converge, we are done.
		\end{proof}
	\end{proposition}
	\begin{remark}
		\label{remarkMaurer}
		Let $(E^*, \mathbb E)$ be an object of a dg-category $\text{Tw}(\mathcal C)$ associated to a cdg-category $\mathcal C$. Clearly, $\text{End}_{\text{Tw}(\mathcal C)}((E^*, \mathbb E))$ is a dg-algebra. Let $a \in \text{End}_{\text{Tw}(\mathcal C)}^1((E^*, \mathbb E))$ such that \begin{align*} d(a) + a \circ a = 0,\end{align*} i.e., $a$ is a Maurer-Cartan element of the dg-algebra $\text{End}_{\text{Tw}(\mathcal C)}((E^*, \mathbb E))$. Then we see 
		\begin{nalign}\nabla(\mathbb E + a) + (\mathbb E + a)^2 + R_E
			&= d(a) + a^2 = 0.
		\end{nalign}
		If the degree $(0,1)$ part of $a$ is trivial, this means we obtain that $(E^*, \mathbb E + a)$ defines an object of $\text{Tw}(\mathcal C)$. The dg-algebra $\text{End}_{\text{Tw}(\mathcal C)}((E^*, \mathbb E+a))$ is then the by-$a$-twisted dg-algebra of $\text{End}_{\text{Tw}(\mathcal C)}((E^*, \mathbb E))$ (compare Definition 2.2 in \cite{RiemannHilbert}).\\ \\
		Furthermore, consider an invertible element $g \in \text{End}_{\text{Tw}(\mathcal C)}^0((E^*, \mathbb E))$. Then $g$ acts on $a$ by $g.a := g \circ a \circ g^{-1} - d_{\mathbb E}(g) \circ g^{-1}$. One can easily check that $g.a$ is again a Maurer-Cartan element. If we assume that the degree $(0,1)$ part of $g.a$ is zero, then we can consider the object $(E^*, \mathbb E + g.a)$. Now regard $g$ as an element $g \in \text{Hom}^0((E^*, \mathbb E+a), (E^*, \mathbb E + g.a))$. Then $g$ is closed using that $d_{\mathbb E}(g) = \nabla(g) + \mathbb E \circ g - g \circ \mathbb E$. Thus, $(E^*, \mathbb E+a)$ and $(E^*, \mathbb E + g.a)$ are isomorphic objects in $\text{Tw}(\mathcal C)$.
	\end{remark}
	\begin{proposition}
		\label{PropQES}
		Let $\mathcal F : \mathcal C \to \mathcal D$ be a proper $A_\infty$-functor between cdg-categories such that $\mathcal D$ is sufficiently twist-closed and split. Assume that the induced $A_\infty$-functor $\mathcal F^{0} : \mathcal C^{0} \to \mathcal D^{0}$ is a quasi-equivalence. Then $\text{Tw}(\mathcal F) : \text{Tw}(\mathcal C) \to \text{Tw}(\mathcal D)$ is quasi-essentially surjective.
		\begin{proof}
			Let $(E^*, \mathbb E)$ be an object in $\text{Tw}(\mathcal D)$. By Lemma \ref{trivialDiff}, we may assume that $\mathbb E_0 = 0$. This implies that the curvature $R_{E^i}$ for each $E^i$ is zero, i.e., the $E^i$ are contained in $\mathcal D^0$. By assumption, there are objects $F^i$ of $\mathcal C^0$ such that $\mathcal F^0 (F^i )$ is isomorphic to $E^i$ for each $i$ (a homotopy equivalence in $\mathcal D^0$ is always an isomorphism since the degrees of the morphism complexes are non-negative). Now, since $R_{F_i} = 0$ for all $i$, we see that the pair $(F^*, 0)$ defines an object of $\text{Tw}(\mathcal C)$. Similarly we obtain the object $(E^*, 0)$ of $\text{Tw}(\mathcal D)$. Furthermore, it follows from the definition of the induced functor $\mathcal F$ that $\mathcal F((F^*, 0)) \cong (E^*, 0)$. Using that $\mathcal F^0$ is quasi-fully faithful, we can reason in the same way as in the proof of the previous proposition that \begin{align*}\text{Tw}(\mathcal F) : \text{Hom}^*((F^*,0), (F^*,0)) \to \text{Hom}^*((E^*,0), (E^*,0))\end{align*} is an $A_\infty$-quasi-isomorphism of dg-algebras. Moreover, by the definition of $\mathcal F$, this map is actually compatible with the auxiliary degrees (this is not in general true). Note that the differentials of the above morphism complexes always strictly increase the auxiliary degree and composition is also compatible with the auxiliary degree. We are now precisely in the setting of Proposition A.9 in \cite{AbadSchaetz} (compare Remark \ref{remarkMaurer}) which allows us to obtain an element \begin{align*} \mathbb F \in \text{Hom}^1((F^*,0), (F^*,0))\end{align*} with $\mathbb F_0 = \mathbb F_1 = 0$ and an invertible $g \in \text{Hom}^0((E^*,0), (E^*,0))$ such that 
			\begin{align*} g\mathcal F_*(\mathbb F)g^{-1} - \nabla(g)g^{-1} = \mathbb E.\end{align*}
			Regarding $g$ as an element of $\text{Hom}^0((E^*,\text{Tw}(\mathcal F)_*(\mathbb F)), (E^*,\mathbb E))$, the above equation implies that $g$ is closed. Thus, $(E^*,\text{Tw}(\mathcal F)_*(\mathbb F))$ and $(E^*,\mathbb E)$ are isomorphic objects in $\text{Tw}(\mathcal D)$. But  since $(E^*,\text{Tw}(\mathcal F)_*(\mathbb F)) \cong \text{Tw}(\mathcal F)((F^*,\mathbb F))$, we are done.
		\end{proof}
	\end{proposition}
	\begin{theorem}
		\label{EquivalenceTheorem}
		Let $\mathcal F : \mathcal C\to \mathcal D$ be a proper $A_\infty$-functor between (non-negatively graded) cdg-categories such that both $\mathcal C$ and $\mathcal D$  are sufficiently twist-closed and split.  Assume that the induced $A_\infty$-functor $\mathcal F^{0} : \mathcal C^{0} \to \mathcal D^{0}$ is a quasi-equivalence. Then $\text{Tw}(\mathcal F): \text{Tw}(\mathcal C) \to \text{Tw}(\mathcal C)$ is an $A_\infty$-quasi equivalence of dg-categories.
		\begin{proof}
			The assumptions on $\mathcal C$ together with the quasi-fully faithfulness of $\mathcal F^0$ imply that $\text{Tw}(\mathcal F)$ is quasi-fully faithful by Proposition \ref{PropQFF}. On the other hand we have that the assumptions on $\mathcal D$ and the requirement that $\mathcal F^0$ is a quasi-equivalence mean that $\text{Tw}(\mathcal F)$ is quasi-essentially surjective by Proposition \ref{PropQES}.
		\end{proof}
	\end{theorem}
	\section{$h$-curved infinity local systems and $h$-curved graded vector bundles}
	From now on, we will always fix either $\mathbb K = \mathbb R$ or $\mathbb K = \mathbb C$.
	\label{sectionInfinityLocalSystems}
	\begin{notation}
		In the following sections we will often consider (non-negatively graded) cdg-categories. We use the following notations:
		\begin{enumerate}
			\item $S^{pre}$ denotes a non-negatively graded cdg-category. 
			\item $S^0$ denotes the full subcategory of $S^{pre}$ consisting of objects which have $0$ curvature.
			\item $S^{\infty}$ is the associated dg-category of twisted complex $S^\infty := \text{Tw}(S^{pre})$ as defined in Definition \ref{defAssociatedDG}. This is the dg-category we are interested in. As remarked before, we can also view $S^0$ is a full subcategory of $S^\infty$.
		\end{enumerate}
		If we are given a proper $A_\infty$-functor $F^{pre} : S^{pre} \to T^{pre}$ of cdg-categories we denote the induced functors by $F^0 := F^{pre}|_{S^0} : S^{0} \to T^{0}$ and $F^{\infty} := \text{Tw}(F^{pre}) : S^{\infty} \to T^{\infty}$.
	\end{notation}
	\begin{def1}
		Let $h$ be a closed 2-form on $M$.
		We define the cdg-category $\mathcal P(M)^{pre}[h]$. Its objects are pairs $(E, \nabla)$ where
		\begin{enumerate}
			\item $E$ is a smooth finite-dimensional vector bundle over $\mathbb K$ on $M$
			\item $\nabla$ is a connection on $E$.
		\end{enumerate}
		The morphism spaces are given by
		\begin{align*} \text{Hom}_{\mathcal P(M)^{pre}[h]}^n((E, \nabla^E),(F, \nabla^F)) := \Omega^n(M, \text{Hom}(E,F) ).\end{align*}
		The connection is given by
		\begin{align*} \nabla(f) := \nabla^F \circ f - (-1)^{|f|}f \circ \nabla^E.\end{align*}
		The composition is induced by the wedge product.
		Let $(R_E)' := (\nabla^E)^2$ be the usual curvature of $E$. We define the curvature of $E$ in the category $\mathcal P(M)^{pre}[h]$ to be 
		\begin{nalign} R_E := (R_E)' - h = (\nabla^E)^2 - h.\end{nalign}
	\end{def1}
	\begin{lemma}
		\label{lemmaP}
		$\mathcal P(M)^{pre}[h]$ is a twist-closed cdg-category.
		\begin{proof}
			The Leibniz rule is straightforward to verify.
			Next, we show that $R_E$ can indeed be regarded as the curvature of $E$. For $f \in \text{Hom}(E,F)$ we have the following:
			\begin{align*}
				\nabla^2(f)
				= (R_F)' \circ f - f \circ (R_E)'
			\end{align*}
			But now note that $h$ has even degree and therefore commutes with $f$, so we conclude
			\begin{align*}
				\nabla^2(f)
				&= (R_F)' \circ f - f \circ (R_E)' - h \circ f + f \circ h\\
				&= R_F \circ f - f \circ R_E.
			\end{align*}
			It remains to show that it is twist-closed. Let $(E, \nabla)$ be an object of $\mathcal P(M)^{pre}[h]$ and let $a \in \Omega^1(M, \text{Hom}(E,E))$. Then the pair $(E, \nabla + a)$ is also an object of $\mathcal P(M)^{pre}[h]$ and can be related to $(E, \nabla)$ via the identity on $E$ in both directions which clearly satisfies the required equation.
		\end{proof}
	\end{lemma}
	\begin{remark}
		\label{remarkCohesive}
		The above category can also be defined as follows. Let $A^*$ be a non-negatively graded, dg-algebra over $\mathbb K$ which is graded commutative. Let $h \in A^2$ be closed. Then we can define the cdg-category $(A^*,h)
		\text{-Mod}^{coh}_{pre}$ of cohesive modules over $(A^*,h)$. Its objects are pairs $(E, \nabla^E)$ where 
		\begin{enumerate}
			\item $E$ is a finitely generated projective $A^0$-module.
			\item $\nabla$ is a $\mathbb K$-linear map $\nabla : E \to A^1 \otimes_{A^0} E$ satisfying the Leibniz rule.
		\end{enumerate}
		The morphism spaces are given by
		\begin{align*} \text{Hom}^n((E, \nabla^E),(F, \nabla^F)) := A^n \otimes_{A^0} \text{Hom}_{A^0}(E,F).\end{align*}
		The connection is given by
		\begin{align*} \nabla(f) := \nabla^F \circ f - (-1)^{|f|}f \circ \nabla^E.\end{align*}
		Here, $\nabla^E$ is extended to $A^* \otimes E$ by the Leibniz rule.
		The curvature of $E$ is given by
		\begin{align*} R_E := (\nabla^E)^2 - h.\end{align*}
		(Note that $E$ being projective and finitely generated implies $\text{Hom}_{A^0}(E, A^i \otimes E) = A^i \otimes \text{Hom}_{A^0}(E, E)$). The associated dg-category $(A^*,h)\text{-Mod}^{coh} := \text{Tw}((A^*,h)\text{-Mod}^{coh}_{pre})$ is precisely the \emph{dg-category of cohesive modules over the cdga $(A^*, h)$}.
		By the Serre-Swan correspondence we obtain that $\mathcal P(M)^{pre}[h]$ is cdg-equivalent to $(A^*,h)
		\text{-Mod}^{coh}_{pre}$ for $A^* = \Omega^*(M, \mathbb K)$. Thus, the associated dg-categories of twisted complexes are also dg-equivalent.
	\end{remark}
	\begin{def1}
		\label{DefInfintyLocal}
		Let $S$ be a simplicial set. 
		Let $H$ be a 2-cocycle $H \in C^2(S; \mathbb K^*)$. We define the cdg-category $\text{Loc}(S)^{pre}[H]$. An object of $\text{Loc}(S)^{pre}[H]$ is a pair $((E_x)_{x \in S_0}, P)$, where 
		\begin{enumerate}
			\item $E_x$ is a finite-dimensional $\mathbb K$-vector space for each $x \in S_0$ and
			\item $P$ assigns to each edge $\gamma \in S_1$ a $\mathbb K$-linear map $P(\gamma) : E_{\gamma(0)} \to E_{\gamma(1)}$
			such that $P(c_{x}) = id_{E_x}$ for every $x \in S_0$ (here $c_x$ is the degenerate path based at $x$).
		\end{enumerate}
		The degree $n$ elements of the morphism spaces $\text{Hom}^*(((E_x)_{x \in S_0}, P^E),((F_x)_{x \in S_0}, P^F))$ are assignments $f$ that map every smooth $n$-simplex $\sigma \in S_n$ to a $\mathbb K$-linear map 
		\begin{align*}f(\sigma) : E_{\sigma_{(0)}} \to F_{\sigma_{(n)}}. \end{align*}
		The connection is given by 
		\begin{nalign} 
			\nabla(f)(\sigma) :&= (-1)^{n+1} H(\sigma_{(0,1,n+1)}) f(\partial^0(\sigma)) P(\sigma_{(0,1)})\\
			&\quad+H(\sigma_{(0,n,n+1)})P(\sigma_{(n, n+1)}) f(\partial^{n+1}(\sigma))\\
			&\quad+ \sum_{i = 1}^{n} (-1)^{i+n+1}  f(\partial^i(\sigma))
		\end{nalign}
		where $\sigma \in S_{n+1}$ is an $(n+1)$-simplex and $f$ has degree $n$.
		We now define the composition $\cup$ in $\text{Loc}(S)^{pre}[H]$. Let $f \in \text{Hom}^{i}(((E_x)_{x \in S_0}, P^E),((F_x)_{x \in S_0}, P^F))$ and $g \in \text{Hom}^{n-i}(((F_x)_{x \in S_0}, P^F),((G_x)_{x \in M}, P^G))$. We set 
		\begin{align*} (g \cup f)(\sigma) := H(\sigma_{(0,i,n)}) g(\sigma_{(i,...,n)})f(\sigma_{(0,...,i)}).\end{align*}
		The curvature of $((E_x)_{x \in S_0}, P)$ is given by 
		\begin{align*} R_E(\sigma) := - P(\sigma_{(0,2)}) + H(\sigma)P(\sigma_{(1,2)})P(\sigma_{(0,1)})\end{align*}
		for $\sigma \in S_2$.
	\end{def1}
	\begin{lemma}
		\label{lemmaLoc}
		$\text{Loc}(S)^{pre}[H]$ is indeed a sufficiently twist-closed cdg-category. 
		\begin{proof}
			It is straightforward using that $H$ is closed to verify that this is a cdg-category. It is sufficiently twist-closed because given any object $(E, P)$ we have that any endomorphism of the form $a = b \circ \nabla(c)$ ($|a| = |b| = 0$) is trivial on degenerate paths. We can thus consider $(E, P + a)$ and see that it fulfills the conditions.
		\end{proof}
	\end{lemma}
	\begin{def1}
		The associated dg-category of \emph{$H$-curved $\infty$-local systems on $S$} will be denoted by $\text{Loc}(S)^{\infty}[H] := \text{Tw}(\text{Loc}(S)^{pre}[H])$.
		We will consider the (smooth) singular simplicial set $S = \text{Sing}(M)$ and use the notation $\text{Loc}(M)^{pre}[H] := \text{Loc}(\text{Sing}(M))^{pre}[H]$. 
	\end{def1}
	Let $h$ be a closed 2-form on $M$ and let $H:= H_h := e^{DR(h)}$ where $DR$ is the de Rham quasi-isomorphism
	\begin{align*} DR : \Omega^*(M, \mathbb K) \to C^*(M, \mathbb K).\end{align*}
	We want to define an $A_\infty$-functor from $\mathcal P(M)^{pre}[h]$ to $\text{Loc}(M)^{pre}[H]$. This will take some preparatory steps.
	\subsection{Diffeological spaces}
	We will be using the notion of diffeological spaces, which allow us to view the path space $PM$ of $M$ as an object with a smooth structure. Chen first introduced the notion of differentiable spaces which are modelled on convex subsets of $\mathbb R^n$ (for example in \cite{Chen}). Diffeological spaces, on the other hand, are modelled on open subsets of $\mathbb R^n$ which tend to be more convenient for our purposes.
	\begin{def1}
		A \emph{diffeological space} $X$ is a set $X$ endowed with a family of maps  $(\phi_\alpha)$, called \emph{plots}, where each $\phi_\alpha$ is a map
		\begin{align*} \phi_\alpha : U_\alpha \to X\end{align*}
		on some oben subset $U_\alpha \subset \mathbb R^n$ such that 
		\begin{enumerate}
			\item each constant map $\phi : U \to X$ is a plot,
			\item for any plot $\phi_\alpha$ and smooth map $f : U \to U_\alpha$, where $U$ is an open subset of some $\mathbb R^n$, we have that $\phi_\alpha \circ f : U \to X$ is a plot, and
			\item for every open $U \subset \mathbb R^n$, map $\phi : U \to X$, and open covering $\bigcup_{i} U_i = U$ such that each $\phi|_{U_i} : U_i \to X$ is a plot, we have that $\phi$ is a plot. 
		\end{enumerate}
	\end{def1}
	We stress that the dimension $n$ of $\mathbb R^n$ in the above definition is not fixed. 
	\begin{def1}
		A map between diffeological spaces $f: X \to Y$ is called \emph{smooth} if for every plot $\phi : U \to X$ of $X$, the composition $f \circ \phi$ is a plot of $Y$.
	\end{def1}
	\begin{remark}
		Any smooth manifold $M$ (possibly with corners) is a diffeological space. The plots are all smooth maps $U \to M$ for open subsets $U \subset \mathbb R^n$ ($n$ is not fixed to be the dimension of the manifold).\\ \\
		For smooth manifolds $M$ and $N$, a map $M \to N$ is smooth as a map of smooth manifolds if and only if it is smooth as a map of diffeological spaces. Thus, the category of smooth manifolds is a full subcategory of the category of diffeological spaces.\\ \\
		For diffeological spaces $X$ and $Y$, the set of all smooth maps $X \to Y$ is a diffeological space by defining $\phi : U \to Map(X,Y)$ to be a plot if the adjoint $U \times X \to Y$ is smooth. In particular, for $I = X$ and $M = Y$, we obtain the diffeological space of smooth paths $Map(I,M) =: PM$. \\ \\
		If $X$ is a diffeological space, any subset $Y \subset X$ is naturally a diffeological space by declaring $\phi : U \to Y$ be a plot if $\iota \circ \phi$ is a plot of $X$ where $\iota$ is the inclusion $Y \to X$.
		
	\end{remark}
	\begin{def1}
		A \emph{smooth vector bundle on a diffeological space $X$} is a family of smooth vector bundles $(E_\alpha)$ such that
		\begin{enumerate}
			\item each $E_\alpha$ is a smooth vector bundle on $U_\alpha$, and
			\item for each $\alpha$ and $\beta$, and smooth map $f : U_\alpha \to U_\beta$ with $\phi_\beta \circ f = \phi_\alpha$ we have 
			\begin{align*} E_\alpha = f^*(E_\beta).\end{align*}
		\end{enumerate}
		Similarly, we can define what a connection on $E$ is.
	\end{def1}
	\begin{def1}
		Let $X$ be a diffeological space and $E$ a vector bundle on $M$.
		An \emph{$E$-valued $k$-form on $X$} $\omega \in \Omega^k(X, E)$ is a family $\omega = (\omega_\alpha)$ such that 
		\begin{enumerate}
			\item for each $\alpha$, $\omega_\alpha \in \Omega^k(U_\alpha, E_\alpha)$ and
			\item for each $\alpha$ and $\beta$ and smooth map $f : U_\alpha \to U_\beta$ with $\phi_\beta \circ f = \phi_\alpha$ we have that 
			\begin{align*} \omega_\alpha = f^*(\omega_\beta).\end{align*} 
		\end{enumerate}
		
	\end{def1}
	
	\subsection{Holonomy forms}
	In the following, we will frequently consider the diffeological space $PM \times I \times I$. Its elements will be written as $(\gamma, s, t)$ and we define the smooth evaluation maps
	\begin{nalign}
		&\text{ev}^s : PM \times I \times I \to M, \text{ev}^s(\gamma, s, t) = \gamma(s),\\ &\text{ev}^t : PM \times I \times I \to M, \text{ev}^t(\gamma, s, t) = \gamma(t).
	\end{nalign}
	Given a vector bundle $(E, \nabla)$ on $M$ we denote by $E^{s/t}$ the pullbacks along $\text{ev}^{s/t}$. Given a form $f \in \Omega^i(M, \text{Hom}(E,F))$, we use the notation \[f^s := (\text{ev}^s)^*f \in \Omega^i(PM \times I \times I, \text{Hom}(E^s,F^s))\] and \[f^t := (\text{ev}^t)^*f \in \Omega^i(PM \times I \times I, \text{Hom}(E^t,F^t)).\] We will frequently consider $\iota_{\partial_s}(f^s)$ and $\iota_{\partial_t}(f^t)$ which we will abbreviate by $\iota_s f$ and $\iota_t f$.
	In this formulation, we view the parallel transport of $(E, \nabla)$ as a smooth section \[{PT \in \Gamma(PM \times I \times I, \text{GL}( E^s, E^t))},\]i.e., $PT(\gamma, s, t)$ is the parallel transport along the path $\gamma|_{[s,t]}$ and therefore satisfies
	\begin{align*} \nabla_t PT(s, t) = 0 = \nabla_s PT(s, t)\end{align*}
	and $PT(\gamma, t,t) = \mathbbm 1$.
	\begin{remark}
		In the following proposition we will define holonomy forms. They will be of crucial importance for the rest of the paper. We already mentioned some aspects of these in the outline. Here, we would like to give a practical motivation. Consider some (ungraded) vector bundle $(E, \nabla)$. We can easily define an object $(\{E_x\}_{x \in M}, PT(0,1))$ of $\text{Loc}(M)^{pre}$ where $E_x$ are the fibers of $E$ and $PT$ is the parallel transport of $E$, defined by $\nabla_t PT(s, t) = 0$ and $PT(s,s) = \mathbbm 1$ for all $s \in I$. Now suppose $(E^*, \mathbb E)$ is an object of $\mathcal P(M)^{\infty}$ and we would like to obtain an object of $\text{Loc}(M)^{\infty}$. This object should have the form $(\{E_x^*\}_{x \in M}, P)$ but it is not clear how to define $P$. In the first paper \cite{BlockRiemann}, Smith and Block define the "holonomy of a $\mathbb Z$-graded connection" which is then integrated over to obtain $P$. This holonomy is defined locally via an infinite series, involving iterated integrals, which were first introduced by Chen \cite{Chen}. More elegantly, Igusa \cite{Igusa} defined the holonomy $\psi$ of $(E^*, \mathbb E)$ by the solution of a differential equation of the form
		\begin{align*}A_{\partial_t} \psi = 0\end{align*}	
		together with $\psi(t,t) = \mathbbm 1$ where $A = \nabla + \mathbb E$.
		Locally, after writing $A =  d + A^0 + A^1 + A^2+...$, this may be expanded as 
		\begin{align*} \partial_t \psi_p(s,t) = \sum_{i = 0}^p (-1)^{i+1} \iota_{\partial_t}(A^{i+1}) \psi_{p-i}(s,t),\end{align*}
		compare Definition 2.8 in \cite{Igusa}.
		In particular, for $p = 0$ we have \[\partial_t \psi_0(s,t) = -\iota_{\partial_t}(A^1) \psi_0(s,t)\] which is the local form of $\nabla_{\partial_t} \psi_0(s,t) = 0$ where $\nabla = d + A^1$. Thus, $\psi_0$ is just the ordinary parallel transport of the connection $\nabla$. The other components $\psi_p$ for $p > 0$ may therefore be thought of as correction terms to the parallel transport, accounting for the higher components of $\mathbb E$. They are small in the sense that $\psi_p(s,s) = 0$ for all $s \in I$ and $p > 0$.
		To define the higher Riemann-Hilbert functor on morphisms, Block and Smith consider iterated cones. In our approach, on the other hand, we only need to define an $A_\infty$-functor $\mathcal P(M)^{pre} \to \text{Loc}(M)^{pre}$ by Definition \ref{DefAssociatedFunctor} and 
		Lemma \ref{inducedFunctorLemma}, which formally take care of the iteration processes. The definition of this functor is straightforward on objects, here we only need the usual parallel transport. Thus, it remains to define this functor on morphisms. To do this, we first define certain holonomy forms as solutions to differential equations and then integrate over them in a suitable manner to obtain our functor. The holonomy forms are rather simple to define. However, in order for them to be useful after integration, we will need quite a few observations, in particular involving symmetries and different kinds of reparametrization invariances. Our holonomy forms $\text{hol}_{-}$ are then related to Igusa's holonomy $\psi$ by the equation
		\begin{align*} \psi = PT + \sum_{k \geq 1} \text{hol}_{\mathbb E^{\otimes k}},\end{align*}
		possibly up to signs.
	\end{remark}
	\begin{proposition}
		\label{holProp}
		Let $E_1,E_2,...$ be a finite or countable family of vector bundles together with connections and let $f_k \in \Omega^*(M, \text{Hom}(E_{k+1}, E_k))$.
		There is a unique family of forms \begin{align*}(\text{hol}_{f_a,f_{a+1},...,f_b})_{a \leq b} \in \Omega^{|f_a|+...+|f_b| -b+a-1}(PM \times I^2, \text{Hom}((E_b)^s, (E_a)^t))\end{align*}
		such that 
		\begin{enumerate}
			\item \label{holPropA0} $\text{hol}_{f_a,...,f_b}(t,t) = 0$ for all $t \in I$.
			\item \label{holPropA1} $\text{hol}_{f_a,...,f_b}$ satisfies the ordinary differential equation
			\begin{nalign} \nabla_{\partial_t} (\text{hol}_{f_a,...,f_b}(s,t)) = (-1)^{|f_a|} \iota_{t}(f_a) \text{hol}_{f_{a+1},...,f_b}(s,t) \end{nalign}
			for $a < b$ and
			\begin{nalign} \nabla_{\partial_t} (\text{hol}_{f_a}(s,t)) = (-1)^{|f_a|} \iota_{t}(f_a) PT(s,t) \end{nalign}
			for $a = b$.
		\end{enumerate}
		This form furthermore satisfies the following properties
		\begin{enumerate}
			\item \label{holPropP1}  \begin{nalign}\iota_{\partial t}(\text{hol}_{f_a,...,f_b}) = \iota_{\partial s}(\text{hol}_{f_a,...,f_b}) = 0.
			\end{nalign}
			\item  \label{holPropP2} $\text{hol}_{f_a,...,f_b}$ satisfies the ordinary differential equation
			\begin{nalign} \nabla_{\partial_s} (\text{hol}_{f_a,...,f_b}(s,t)) = (-1)^{|f_b|} \text{hol}_{f_{a},...,f_{b-1}}(s,t) \iota_{s}(f_b)  \end{nalign}
			for $a < b$
			and
			\begin{nalign} \nabla_{\partial_s} (\text{hol}_{f_a}(s,t)) = (-1)^{|f_a|+1} PT(s,t) \iota_{s}(f_a)
			\end{nalign}
			for $a = b$.
			\item  \label{holPropP3} For every $0 < c < d < 1$ and smooth diffeomorphism $\phi : [c,d] \to I$ fixing the endpoints, we have that the restriction of $\text{hol}_{f_a,...,f_b}$ to $PM \times [c,d]^2$ is equal to the pullback of $\text{hol}_{f_a,...,f_b}$ along the map
			\begin{nalign} p : PM \times [c,d]^2 \to PM \times I^2, (\gamma, s,t) \mapsto (\gamma \circ \phi^{-1}, \phi(s), \phi(t))\end{nalign}
			\item \label{holPropP4}For all $s,t,r \in I$ the following identity holds
			\begin{nalign} \text{hol}_{f_a,...,f_b}(s,t) =
				& \text{hol}_{f_a,...,f_b}(r,t)\text{PT}(s,r) + \text{PT}(r,t)\text{hol}_{f_a,...,f_b}(s,r) \\&+ \sum_{i = a}^{b-1} \text{hol}_{f_a,...,f_i}(r,t)\text{hol}_{f_{i+1},...,f_b}(s,r)\end{nalign} 
			\item  \label{holPropP5}Let $A$ be a diffeological space and let $q : A \to PM$ be smooth. Furthermore, let $\phi :A \times I \to I$ be a map such that there exists a partition $0 = a_0 < a_1 < ... < a_n = 1$ with the property that $\phi|_{A \times [a_i, a_{i+1}]}$ is smooth and such that $A \times I \to M, (a,t) \mapsto q(a)(\phi(a,t))$ is smooth. Consider the maps
			\begin{nalign}
				&f : A \times I^2 \to PM \times I^2, f(a, s,t) := (q(a) \circ \phi(-,a), s,t)\\
				&g : A \times I^2 \to PM \times I^2, g(a, s,t) := (q(a), \phi(s,a),\phi(t,a)).
			\end{nalign}
			(The first one is smooth by assumption and the second is piecewise smooth.) Then we have $f^* \text{hol}_{f_a,...,f_b} = g^* \text{hol}_{f_a,...,f_b}$.
		\end{enumerate}
		
		\begin{proof}
			Using global existence, uniqueness statements, and smooth dependence on parameters for solutions of linear, first order differential equations, we obtain a unique form 
			$ \text{hol}_{f_a,...,f_b}(s,t)$ on $\Omega^{|f_a|+...+|f_b| -b+a-1}(PM \times I \times I, \text{Hom}((E_b)^0, (E_a)^1))$ satisfying the stated conditions by induction over $b - a \in \mathbb N$. \\ \\
			It remains to check the other properties. Property \ref{holPropP1} follows inductively from the definition since $ \iota_{\partial_t}\iota_{\partial_t}((\text{ev}^t)^*f_a) =\iota_{\partial_s}\iota_{\partial_t}((\text{ev}^t)^*f_a)  = 0$.\\ \\
			We continue with the second property. We show this for $a = b$, the general argument can then be done inductively. Since $\iota_t \text{hol}_{f_a} = \iota_s \text{hol}_{f_a} = 0$ and using $d(ev^t)(\partial_s) = 0 = d(ev^s)(\partial_t)$, we obtain that 
			\begin{align*} \nabla_{\partial_t} \nabla_{\partial_s} \text{hol}_{f_a} = \nabla_{\partial_s} \nabla_{\partial_t} \text{hol}_{f_a}.\end{align*}
			Therefore, $\nabla_{\partial_t} \nabla_{\partial_s}  \text{hol}_{f_a} = 0$. On the other hand, we also have that
			\begin{align*} \nabla_{\partial_t} (-1)^{|f_a|+1} PT(s,t) \iota_{s}(f_a) = 0. \end{align*}
			But for $t = s$, we have that both expressions are the same so in fact 
			\begin{align*} \nabla_{\partial_s} \text{hol}_{f_a} = (-1)^{|f_a|+1} PT(s,t) \iota_{s}(f_a).\end{align*} 
			\\ \\
			We have to show property \ref{holPropP3}. First note that the map $p$ is compatible with the evaluations (since $\phi$ perserves endpoints). This means, in particular, that the pullback along this map induces the correct vector bundles and connections. We will again restrict to the case $a = b$, as once again the general case follows inductively. We show that the pullback of $\text{hol}_{f_a}$ along $p$ satisfies the defining properties. The first one is immediate. For the second one we calculate for $s,t \in [c,d]$:
			\begin{nalign}
				\nabla_{\partial_t} p^* \text{hol}_{f_a}(s,t) &= p^* (\nabla_{d\phi(\partial_t) }\text{hol}_{f_a}(\phi(s),\phi(t)))\\
				&= p^* ((-1)^{|f_a|} \iota_{d(\phi)\partial_t}((\text{ev}^t)^* f_a) PT(\phi(s),\phi(t)) )\\
				&= (-1)^{|f_a|} \iota_{t}(f_a) PT(s,t) 
			\end{nalign}
			where we used that $dp(\partial_t) = d\phi(\partial_t)$ (and property 1) so $dp(\partial_t)$, in particular, is a multiple of $\partial_t$ and that $p$ is compatible with the map $\text{ev}^t$.
			By uniqueness, we obtain the claim.\\ \\
			We now proceed with the fourth claim \ref{holPropP4}. We see that the equation holds for $r = t$. Then we proceed inductively over $b-a$. 
			Using the induction hypothesis, we see that both sides agree after applying $\nabla_t$ to them so they must coincide for all $t$.
			\\ \\
			We are left with claim \ref{holPropP5}. We again restrict to the case $a = b$, as the general case follows inductively. Per assumption there is a partition $0 < z_1 < ... < z_n < 1$ such that $\phi$ is smooth on every interval $[z_i, z_{i+1}]$. We consider the restrictions $f'$ and $g'$ of $f$ and $g$ to $PM \times [z_i, z_{i+1}]^2$ for every $i$. Then we have
			\begin{nalign}
				\nabla_{\partial_t} (f')^* \text{hol}_{f_a} &= (f')^* (\nabla_{\partial_t} \text{hol}_{f_a}) \\
				&= (f')^* ((-1)^{|f_a|} \iota_{t}(f_a)PT\\
				&= (-1)^{|f_a|} \iota_{\partial_t}((\text{ev}^t \circ f')^* f_a) PT \circ f'.
			\end{nalign}
			And similarly
			\begin{nalign}
				\nabla_{\partial_t} (g')^* \text{hol}_{f_a} = (-1)^{|f_a|} \iota_{\partial_t}((\text{ev}^t \circ g')^* f_a) PT \circ g'.
			\end{nalign}
			But $PT \circ f'(a, s,t) = PT(q(a) \circ \phi(-,a), s,t) = PT(q(a), \phi(s,a),\phi(t,a))  = PT \circ g'(a,s,t)$  by reparametrization invariance of the parallel transport, and
			\[ev^t \circ g'(a,s,t) = q(a)(\phi(s,a)) = ev^t \circ f'(a,s,t),\] so both terms agree. Since they are also both zero for $t = s$, we obtain that they must agree. Using property \ref{holPropP4}, we see that this equality extends to all of $A \times I^2$.
		\end{proof}
	\end{proposition}
	
	Let $R_E'$ be the curvatures of $(E, \nabla)$, i.e., $\nabla^2 = R_E'$ (not the curvature in $\mathcal P(M)^{pre}[h]$). 
	\begin{lemma}
		\label{holLemma}
		$\text{hol}$ satisfies the following equations 
		\begin{nalign}
			&(-1)^{|f_1|+...+|f_n|+1-n}(\nabla \text{hol}_{f_1,...,f_n}) \\
			&\quad+ (-1)^{|f_{2}| + ... + |f_n|-n} f_1^t \text{hol}_{f_{2},...,f_n} +   \text{hol}_{f_{1},...,f_{n-1}} f_n^s \\
			&= \sum_{k = 1}^n (-1)^{|f_k|+...+|f_n| - n+k}\text{hol}_{f_1,...,\nabla f_k,...,f_n}\\
			&\quad+ \sum_{k = 1}^n (-1)^{|f_{k+1}|+...+|f_n| - n+k-1}\text{hol}_{f_1,...,f_k \circ f_{k+1},...,f_n}\\
			&\quad+ \sum_{k = 1}^{n+1} (-1)^{|f_k|+...+|f_n| - n+k}\text{hol}_{f_1,...,f_{k-1},(R_{E_k}'),f_k,...,f_n}
		\end{nalign}
		where $\text{hol}_{f_{2},...,f_n}$ is understood to be $PT$ when $n = 1$.
		\begin{proof}
			First, we compute using $\iota_{\partial_t}(\text{hol}_{f_1,...,f_n}) = 0$ and $\iota_{\partial_t} \circ (\text{ev}^s)^* = 0$:
			\begin{nalign}
				\iota_{t}(R'_1) \text{hol}_{f_1,...,f_n} &= \iota_{\partial_t} (R'_1 \text{hol}_{f_1,...,f_n} - \text{hol}_{f_1,...,f_n} R'_{n+1}) = \iota_{\partial_t}( \nabla^2(\text{hol}_{f_1,...,f_n})) \\
				&= \nabla_{\partial_t} \nabla(\text{hol}_{f_1,...,f_n}) - \nabla (\iota_{\partial_t}(\nabla(\text{hol}_{f_1,...,f_n}))\\
				&= \nabla_{\partial_t} \nabla(\text{hol}_{f_1,...,f_n}) - \nabla \nabla_{\partial_t}(\text{hol}_{f_1,...,f_n}) 
			\end{nalign}
			and similarly, we obtain
			\begin{nalign}
				\iota_{t}(R')  \text{PT} = \nabla_{\partial_t} \nabla(PT)
			\end{nalign}
			so $\nabla(PT) = \text{hol}_{R'}$ by definition of $\text{hol}_{R'}$.
			We continue to calculate
			\begin{nalign}
				\nabla_{\partial_t} \nabla(\text{hol}_{f_1,...,f_n}) &= \iota_{t}(R'_1) \text{hol}_{f_1,...,f_n} + \nabla \nabla_{\partial_t}(\text{hol}_{f_1,...,f_n})\\
				&= \nabla_{\partial_t} \text{hol}_{R'_1,f_1,...,f_n} + \nabla((-1)^{|f_a|} \iota_{t} (f_a) \text{hol}_{f_{2},...,f_n}).
			\end{nalign}
			This leads us to showing the identity inductively. We start with $n = 1$ and obtain
			\begin{align*}
				\nabla_{\partial_t} (\nabla(\text{hol}_{f_1})) &=
				\nabla_{\partial_t} \text{hol}_{R'_1,f_1} + \nabla((-1)^{|f_1|} \iota_{t} (f_1) \text{PT}) \\
				&= \nabla_{\partial_t} \text{hol}_{R'_1,f_1} + (-1)^{|f_1|} (\nabla_{\partial_t} f_1^t - \iota_{t}(\nabla f_1^t)) PT -\iota_{t} (f_1) \text{hol}_{R_2'}\\
				&= \nabla_{\partial_t} \text{hol}_{R'_1,f_1} + (-1)^{|f_1|} \nabla_{\partial_t} (f_1^t PT) + \nabla_{\partial_t}  \text{hol}_{\nabla f_1}- (-1)^{|f_1|} \nabla_{\partial_t} \text{hol}_{f_1,R_2'}\\
				&= \nabla_{\partial_t} (\text{hol}_{R'_1,f_1} + (-1)^{|f_1|} (f_1^t PT) +  \text{hol}_{\nabla f_1} +(-1)^{|f_1|+1} \text{hol}_{f_1,R_2'}).
			\end{align*}
			The term $\nabla(\text{hol}_{f_1})$ is equal to $(-1)^{|f_1|} dt \iota_{t} (f_1) PT + (-1)^{|f_1|+1} ds  PT\iota_{s} (f_1)$ for $s = t$. However, the other term $(\text{hol}_{R'_1,f_1} + (-1)^{|f_1|} (f_1^t PT) +  \text{hol}_{\nabla f_1} +(-1)^{|f_1|+1} \text{hol}_{f_1,R_2'})$ is not. Specifically, the only non-zero term in this sum is $(-1)^{|f_1|} (f_1^t PT)$ (for s = t). We fix this by adding $(-1)^{|f_1|+1} PT f_1^s$ which fulfills 
			$\nabla_{\partial_t} (PT f_1^s) = 0$ and clearly $(-1)^{|f_1|+1} PT f_1^s + (-1)^{|f_1|} f_1^t PT = (-1)^{|f_1|} dt \iota_{t} (f_1) PT + (-1)^{|f_1|+1} ds  PT\iota_{s} (f_1)$ for $s = t$. So we obtain
			\begin{align*}
				\nabla(\text{hol}_{f_1}) = \text{hol}_{R'_1,f_1} + (-1)^{|f_1|} f_1^t PT + (-1)^{|f_1|+1} PT f_1^s +  \text{hol}_{\nabla f_1} +(-1)^{|f_1|+1} \text{hol}_{f_1,R_2'}
			\end{align*}
			which is the desired equation for $n = 1$. Let us now assume that the claim has been proven for $n-1$. Applying the induction hypothesis we find
			\begin{align*}
				\nabla_{\partial_t}(\nabla \text{hol}_{f_1,...,f_n})
				&= \nabla_{\partial_t}[(-1)^{|f_1|} f_1^t \text{hol}_{f_2,...,f_n} + (-1)^{|f_1|+...+|f_n|-n} \text{hol}_{f_1,...,f_{n-1}} f_n^s \\
				&\quad+ \sum_{k = 1}^n (-1)^{|f_1|+...+|f_{k-1}|+k+1} \text{hol}_{f_1,...,\nabla(f_k),f_{n}}\\
				&\quad+ \sum_{k = 1}^n (-1)^{|f_1|+...+|f_{k}|+k} \text{hol}_{f_1,...,f_k \circ f_{k+1},f_{n}}\\
				&\quad+ \sum_{k = 1}^n (-1)^{|f_1|+...+|f_{k-1}|+k+1} \text{hol}_{f_1,...,f_{k-1},R'_k,...,f_{n}} ].
			\end{align*}
			This equation is the same as the desired equation after applying $\nabla_t$. Therefore, it suffices to show that they are equal for $t = s$. Both sides are equal to $0$ for $t = s$ and $n \geq 2$, and thus the proof is concluded.
		\end{proof}
	\end{lemma}
	\begin{def1}
		Let $x_0, x_1 \in M$. We denote by $PM(x_0,x_1)$ the diffeological subspace of paths that start at $x_0$ and end at $x_1$. We consider the pullback along the map
		\begin{align*} L : PM(x_0,x_1) \to PM \times I^2, \gamma \mapsto (\gamma, 0, 1).\end{align*}
		When pulling back the bundles $E^s$ and $E^t$ along this map, they become trivial, i.e., $E^s \cong E_{x_0}$ and $E^t \cong E_{x_1}$. The pullback of the connections along this map become the trivial connection $d$. We denote the pullback of $\text{hol}$ to $PM(x_0,x_1)$ by $\text{hol}^{x_0,x_1}$.
	\end{def1}
	
	\begin{corollary}
		\label{holCor}
		The pullback $ \text{hol}^{x_0,x_1}_{f_1,...,f_n}$ satisfies
		\begin{nalign}
			&(-1)^{|f_1|+...+|f_n|+1-n}(d \text{hol}^{x_0,x_1}_{f_1,...,f_n}) \\
			&\quad+ (-1)^{|f_{2}| + ... + |f_n|-n}  f_1^0(x_1) \text{hol}^{x_0,x_1}_{f_{2},...,f_n} +   (-1)^{|f_n|} \text{hol}_{f_{1},...,f_{n-1}} f_n^0(x_0) \\
			&= \sum_{k = 1}^n (-1)^{|f_k|+...+|f_n| - n+k}\text{hol}^{x_0,x_1}_{f_1,...,\nabla f_k,...,f_n}\\
			&\quad+ \sum_{k = 1}^n (-1)^{|f_{k+1}|+...+|f_n| - n+k-1}\text{hol}^{x_0,x_1}_{f_1,...,f_k \circ f_{k+1},...,f_n}\\
			&\quad+ \sum_{k = 1}^{n+1} (-1)^{|f_k|+...+|f_n| - n+k}\text{hol}^{x_0,x_1}_{f_1,...,f_{k-1},(R_{E_k}'),f_k,...,f_n}.
		\end{nalign}
		\begin{proof}
			This follows from pulling back the equation of the previous lemma together with the fact that when we pullback $f_1^t$ to $PM(x_0,x_1)$, we obtain $f_1(x_1)$ as the map $\text{ev}^t \circ L$ is constant. We also added the sign
			$(-1)^{|f_n|}$ to $\text{hol}_{f_{1},...,f_{n-1}} f_n^0(x_0)$ which is possible since the only way that this term is non-zero is when $f_n$ has degree $0$.
		\end{proof}
	\end{corollary}
	In the above equation the curvatures $R_k'$ appear rather than $R_k$. We need the following Lemma.
	\begin{lemma}
		\label{holHLemma}
		The following identity holds:
		\begin{nalign}
			\sum_{k = 1}^{n+1} (-1)^{|f_1|+...+|f_{k-1}| + k-1} \text{hol}_{f_1,...,f_{k-1}, h \mathbbm 1_{E_k},f_k, ...,f_n}(s,t) = \text{hol}_h(s,t) \text{hol}_{f_1,...,f_n}(s,t)
		\end{nalign}
		\begin{proof}
			We differentiate the left side with respect to $\nabla_{\partial_t}$ and obtain
			\begin{nalign}
				&\nabla_{\partial_t} (\sum_{k = 1}^{n+1} (-1)^{|f_1|+...+|f_{k-1}| + k-1} \text{hol}_{f_1,...,f_{k-1}, h \mathbbm 1_{E_k},f_k, ...,f_n}(s,t)) \\
				&= \sum_{k = 2}^{n+1} (-1)^{|f_2|+...+|f_{k-1}| + k-1} \iota_{\partial_t}(f_1) \text{hol}_{f_2,...,f_{k-1}, h \mathbbm 1_{E_k},f_k, ...,f_n}(s,t) + \iota_{\partial_t}(h) \text{hol}_{f_1,...,f_n}(s,t)
			\end{nalign}
			Differentiating the right side yields
			\begin{nalign}
				&\nabla_{\partial_t} (\text{hol}_h(s,t) \text{hol}_{f_1,...,f_n}(s,t))\\
				&= \iota_{t}(h) PT_{\mathbb K}(s,t) \text{hol}_{f_1,...,f_n}(s,t) + (-1)^{|f_1|} \text{hol}_{h}(s,t) \iota_{t} (f_1) \text{hol}_{f_2,...,f_n}(s,t).
			\end{nalign}
			Since the parallel transport of $\mathbb K$ is the identity and $\text{hol}_h$ has degree $1$ (and $\text{hol}_h$ is central), we obtain
			\begin{nalign}
				&\nabla_{\partial_t} (\text{hol}_h(s,t) \text{hol}_{f_1,...,f_n}(s,t)) \\
				&= \iota_{t}(h) \text{hol}_{f_1,...,f_n}(s,t) - \iota_{t}(f_1) \text{hol}_{h }(s,t) \text{hol}_{f_2,...,f_n}(s,t).
			\end{nalign}
			We now proceed inductively. For $n = 1$, we see that both terms are equal after applying $\nabla_{\partial_t}$. Moreover, both sides are equal to $0$ for $s = t$. Now assume that the claim has been proven for $n-1$. Then one may check
			\begin{nalign}
				&\nabla_{\partial_t} (\text{hol}_h(s,t) \text{hol}_{f_1,...,f_n}(s,t)) \\
				&= \nabla_{\partial_t} \sum_{k = 1}^{n+1} (-1)^{|f_1|+...+|f_{k-1}| + k-1} \text{hol}_{f_1,...,f_{k-1}, h \mathbbm 1_{E_k},f_k, ...,f_n}(s,t).
			\end{nalign}
			Again both terms also coincide $s = t$, so we obtain the claim.
		\end{proof}
	\end{lemma}

	\subsection{Relating cubes with simplices}
	For each simplex $\Delta^n$, we wish to find a map $I^{n-1} \times I \to \Delta^n$ in order to be able to define the higher parallel transport along simplices. In order for the maps to be useful, they need to be homotopy coherent in some sense, up to reparametrization. This is the content of the next definitions. After this, we will explicitly construct maps satisfying these conditions.
	\begin{def1}
		\label{helpDef}
		\begin{enumerate}
			\item Denote by $\text{pr}^{1,i}$ the projection \begin{align*}\text{pr}^{1,i} : I^n \to I^i, \text{pr}^{0,i}(t_1,...,t_n) = (t_1,...,t_i)\end{align*}
			and by $\text{pr}^{i,n}$ the projection
			\begin{align*}\text{pr}^{i,n} : I^n \to I^{n-i+1}, \text{pr}^{i,n}(t_1,...,t_n) = (t_{i},...,t_n).\end{align*}
			Furthermore, we denote by $\text{pr}_i$ the projection
			\begin{align*}\text{pr}_i : I^n \to I^{n-1}, \text{pr}_i(t_1,...,t_n) = (t_1,...,t_{i-1},t_{i+1},...,t_n).\end{align*}
			\item For numbers $a < b \in \mathbb R$ and $c < d \in \mathbb R$ we define 
			$sc : [a,b] \to [c,d]$ to be the map given by
			\begin{align*} sc(ta + (1-t)b) = tc + (1-t)d\end{align*}
			for all $t \in I$, i.e., $sc$ reparametrizes the interval $[a,b]$ linearily to $[c,d]$.
			\item A smooth map $\phi : A \to B$ between oriented, compact manifolds with corners $A$ and $B$ is called \emph{integral-compatible} if $\phi$  fulfills the equation
			\begin{align*} \int_A \phi^* \omega = \int_B \omega\end{align*}
			for all smooth forms $\omega$ on $B$. In other words, $\phi$ is a degree $1$ map.
			\item Furthermore we define the cubical frontface retraction $F_i : I^{n} \to I^n$ by $F_i(t_1,...,t_n) = (t_1,...,t_{i-1},0,t_{i+1},...,t_n)$ and similarly $R_i : I^n \to  I^n$ by \\$R_i(t_1,...,t_n) = (t_1,...,t_{i-1},1,t_{i+1},...,t_n)$.
		\end{enumerate}
	\end{def1}
	\begin{def1}
		\label{AdmissibleAxioms}
		A collection of maps $\smiley^n : I \times I^{n} \to \Delta^{n+1}$ is called \emph{admissible} it satisfies the following axioms:
		\begin{enumerate}[label=A\arabic*,ref=A\arabic*]
			\item \label{AXA1}Each $\smiley^n$ is smooth.
			\item \label{AXA2}The maps $\smiley^n$ are integral-compatible.
			\item\label{AXA3}For each $n$, we have that $\smiley^n(\{1\} \times I^n   ) = \{(1,...,1)\}$ and $\smiley^n(\{0\} \times I^n   )  = \{(0,...,0)\}$. 
			\item \label{AXA4}For each $n$ and $i = 1,...,n$, there exists a piecewise smooth map fixing the endpoints $\phi : I \to I$ such that the diagram
			\[
			\begin{tikzcd}
				I \times R_i(I^n) \arrow[r, "\smiley^n"] \arrow[d, "\phi \times \text{pr}_i"'] & \Delta^{n+1}               \\
				I \times I^{n-1} \arrow[r, "\smiley^{n-1}"]                                    & \Delta^n \arrow[u, "q_i"']
			\end{tikzcd}
			\]
			commutes.
			\item \label{AXA5}For each $n$ and $i = 1,...,n$ there are piecewise smooth maps fixing the endpoints $\phi, \phi' : I \to I$ such that the diagrams
			
			\[\begin{tikzcd}
				{[0, \frac{i}{n+1}] \times F_i(I^n)} \arrow[d, "{ (\phi \circ sc) \times \text{pr}^{1,i-1}}"'] \arrow[r, "\smiley^n"] & \Delta^{n+1}               \\
				I \times I^{i-1} \arrow[r, "\smiley^{i-1}"]                                                                        & \Delta^{i} \arrow[u, "\mathcal F^{i}"']
			\end{tikzcd}
			\begin{tikzcd}
				{  [\frac{i}{n+1}, 1] \times F_i(I^n)} \arrow[d, "{ (\phi' \circ sc) \times \text{pr}^{i+1,n}}"'] \arrow[r, "\smiley^n"] & \Delta^{n+1}               \\
				I \times I^{n-i} \arrow[r, "\smiley^{n-i}"]                                                                                          & \Delta^{n-i+1} \arrow[u, "\mathcal R^{n-i+1}"']
			\end{tikzcd}
			\]
			commute. Here, $\mathcal R^{n-i+1}$ denotes the $(n-i+1)$-dimensional rear face map and $\mathcal F^{i}$ is the $i$-dimensional front face.
		\end{enumerate}
	\end{def1}
	
	\begin{remark}
		Igusa \cite{Igusa} defined certain families of paths $I \times I^n \to \Delta^{n+1}$, fulfilling similar properties as our axioms. The maps Igusa defines are only piecewise smooth, making some of the following arguments rather tricky when spelled out in full detail. Therefore, we make an extra effort to find smooth maps.
	\end{remark}
	We need to ensure that such admissible collections exist.
	\begin{def1}
		\begin{enumerate}
			\item We first define certain maps $\frownie^n : I \times I^n \to \Delta^{n+1}$. These will satisfy the above axioms up to smoothness. We define $\frownie^n$ on each subset $[\frac{i}{n+1}, \frac{i+1}{n+1}] \times I^n$ inductively over $i$. For $i = 0$, we set
			\begin{align*}\frownie(t,s) := t(n+1)(1,s_1,s_1 s_2,...,s_1... s_n).\end{align*}
			For $i > 0$, we use the map $sc : [\frac{i}{n+1}, \frac{i+1}{n+1}] \to I$ defined in \ref{helpDef} and then set
			\begin{align*} \frownie(t,s) := \frownie(\frac{i}{n+1},s) + sc(t) \cdot (0,...,0,1-s_i,(1-s_i)s_{i+1},...,(1-s_i)s_{i+1}...s_n)\end{align*}
			where the component which contains $1-s_i$ is the $i$-th component.
		\end{enumerate}
	\end{def1}
	We think of $\frownie^n$ as an $n$-dimensional homotopy of paths $\{\frownie^n(-,s_1,...,s_n)\}_{s_i \in I}$.
	Here is an illustration of the paths $\frownie^2(-,s_1,s_2)$ for different values of $s_1$ and $s_2$.\\
	\tdplotsetmaincoords{60}{110}
	\begin{tikzpicture}[scale=1.7,tdplot_main_coords]
		\coordinate (O) at (0,0,0);
		\draw[thick,->] (0,0,0) -- (1.5,0,0) node[below left=-3]{$t_1$};
		\draw[thick,->] (0,0,0) -- (0,1.1,0) node[right=-1]{$t_2$};
		\draw[thick,->] (0,0,0) -- (0,0,0.5) node[above=-1]{$t_3$};
		\draw[] (0,0,0) -- (1,0,0) -- (1,1,0) -- (1,1,1) -- cycle;
		\draw[dashed] (0,0,0) -- (1,1,0);
		\draw[] (1,0,0) -- (1,1,1);
		\draw[color=blue] (0,0,0) -- (1,0.000000,0.000000) -- (1,1,0.000000) -- (1,1,1) node[anchor=south] {0.00, 0.00};
		
	\end{tikzpicture}
	\begin{tikzpicture}[scale=1.7,tdplot_main_coords]
		\coordinate (O) at (0,0,0);
		\draw[thick,->] (0,0,0) -- (1.5,0,0) node[below left=-3]{$t_1$};
		\draw[thick,->] (0,0,0) -- (0,1.1,0) node[right=-1]{$t_2$};
		\draw[thick,->] (0,0,0) -- (0,0,0.5) node[above=-1]{$t_3$};
		\draw[] (0,0,0) -- (1,0,0) -- (1,1,0) -- (1,1,1) -- cycle;
		\draw[dashed] (0,0,0) -- (1,1,0);
		\draw[] (1,0,0) -- (1,1,1);
		\draw[color=blue] (0,0,0) -- (1,0.333333,0.000000) -- (1,1,0.000000) -- (1,1,1) node[anchor=south] {0.33, 0.00};
		
	\end{tikzpicture}
	\begin{tikzpicture}[scale=1.7,tdplot_main_coords]
		\coordinate (O) at (0,0,0);
		\draw[thick,->] (0,0,0) -- (1.5,0,0) node[below left=-3]{$t_1$};
		\draw[thick,->] (0,0,0) -- (0,1.1,0) node[right=-1]{$t_2$};
		\draw[thick,->] (0,0,0) -- (0,0,0.5) node[above=-1]{$t_3$};
		\draw[] (0,0,0) -- (1,0,0) -- (1,1,0) -- (1,1,1) -- cycle;
		\draw[dashed] (0,0,0) -- (1,1,0);
		\draw[] (1,0,0) -- (1,1,1);
		\draw[color=blue] (0,0,0) -- (1,0.666667,0.000000) -- (1,1,0.000000) -- (1,1,1) node[anchor=south] {0.67, 0.00};
		
	\end{tikzpicture}
	\begin{tikzpicture}[scale=1.7,tdplot_main_coords]
		\coordinate (O) at (0,0,0);
		\draw[thick,->] (0,0,0) -- (1.5,0,0) node[below left=-3]{$t_1$};
		\draw[thick,->] (0,0,0) -- (0,1.1,0) node[right=-1]{$t_2$};
		\draw[thick,->] (0,0,0) -- (0,0,0.5) node[above=-1]{$t_3$};
		\draw[] (0,0,0) -- (1,0,0) -- (1,1,0) -- (1,1,1) -- cycle;
		\draw[dashed] (0,0,0) -- (1,1,0);
		\draw[] (1,0,0) -- (1,1,1);
		\draw[color=blue] (0,0,0) -- (1,1.000000,0.000000) -- (1,1,0.000000) -- (1,1,1) node[anchor=south] {1.00, 0.00};
		
	\end{tikzpicture}
	\\
	\begin{tikzpicture}[scale=1.7,tdplot_main_coords]
		\coordinate (O) at (0,0,0);
		\draw[thick,->] (0,0,0) -- (1.5,0,0) node[below left=-3]{$t_1$};
		\draw[thick,->] (0,0,0) -- (0,1.1,0) node[right=-1]{$t_2$};
		\draw[thick,->] (0,0,0) -- (0,0,0.5) node[above=-1]{$t_3$};
		\draw[] (0,0,0) -- (1,0,0) -- (1,1,0) -- (1,1,1) -- cycle;
		\draw[dashed] (0,0,0) -- (1,1,0);
		\draw[] (1,0,0) -- (1,1,1);
		\draw[color=blue] (0,0,0) -- (1,0.000000,0.000000) -- (1,1,0.333333) -- (1,1,1) node[anchor=south] {0.00, 0.33};
		
	\end{tikzpicture}
	\begin{tikzpicture}[scale=1.7,tdplot_main_coords]
		\coordinate (O) at (0,0,0);
		\draw[thick,->] (0,0,0) -- (1.5,0,0) node[below left=-3]{$t_1$};
		\draw[thick,->] (0,0,0) -- (0,1.1,0) node[right=-1]{$t_2$};
		\draw[thick,->] (0,0,0) -- (0,0,0.5) node[above=-1]{$t_3$};
		\draw[] (0,0,0) -- (1,0,0) -- (1,1,0) -- (1,1,1) -- cycle;
		\draw[dashed] (0,0,0) -- (1,1,0);
		\draw[] (1,0,0) -- (1,1,1);
		\draw[color=blue] (0,0,0) -- (1,0.333333,0.111111) -- (1,1,0.333333) -- (1,1,1) node[anchor=south] {0.33, 0.33};
		
	\end{tikzpicture}
	\begin{tikzpicture}[scale=1.7,tdplot_main_coords]
		\coordinate (O) at (0,0,0);
		\draw[thick,->] (0,0,0) -- (1.5,0,0) node[below left=-3]{$t_1$};
		\draw[thick,->] (0,0,0) -- (0,1.1,0) node[right=-1]{$t_2$};
		\draw[thick,->] (0,0,0) -- (0,0,0.5) node[above=-1]{$t_3$};
		\draw[] (0,0,0) -- (1,0,0) -- (1,1,0) -- (1,1,1) -- cycle;
		\draw[dashed] (0,0,0) -- (1,1,0);
		\draw[] (1,0,0) -- (1,1,1);
		\draw[color=blue] (0,0,0) -- (1,0.666667,0.222222) -- (1,1,0.333333) -- (1,1,1) node[anchor=south] {0.67, 0.33};
		
	\end{tikzpicture}
	\begin{tikzpicture}[scale=1.7,tdplot_main_coords]
		\coordinate (O) at (0,0,0);
		\draw[thick,->] (0,0,0) -- (1.5,0,0) node[below left=-3]{$t_1$};
		\draw[thick,->] (0,0,0) -- (0,1.1,0) node[right=-1]{$t_2$};
		\draw[thick,->] (0,0,0) -- (0,0,0.5) node[above=-1]{$t_3$};
		\draw[] (0,0,0) -- (1,0,0) -- (1,1,0) -- (1,1,1) -- cycle;
		\draw[dashed] (0,0,0) -- (1,1,0);
		\draw[] (1,0,0) -- (1,1,1);
		\draw[color=blue] (0,0,0) -- (1,1.000000,0.333333) -- (1,1,0.333333) -- (1,1,1) node[anchor=south] {1.00, 0.33};
		
	\end{tikzpicture}
	\\
	\begin{tikzpicture}[scale=1.7,tdplot_main_coords]
		\coordinate (O) at (0,0,0);
		\draw[thick,->] (0,0,0) -- (1.5,0,0) node[below left=-3]{$t_1$};
		\draw[thick,->] (0,0,0) -- (0,1.1,0) node[right=-1]{$t_2$};
		\draw[thick,->] (0,0,0) -- (0,0,0.5) node[above=-1]{$t_3$};
		\draw[] (0,0,0) -- (1,0,0) -- (1,1,0) -- (1,1,1) -- cycle;
		\draw[dashed] (0,0,0) -- (1,1,0);
		\draw[] (1,0,0) -- (1,1,1);
		\draw[color=blue] (0,0,0) -- (1,0.000000,0.000000) -- (1,1,0.666667) -- (1,1,1) node[anchor=south] {0.00, 0.67};
		
	\end{tikzpicture}
	\begin{tikzpicture}[scale=1.7,tdplot_main_coords]
		\coordinate (O) at (0,0,0);
		\draw[thick,->] (0,0,0) -- (1.5,0,0) node[below left=-3]{$t_1$};
		\draw[thick,->] (0,0,0) -- (0,1.1,0) node[right=-1]{$t_2$};
		\draw[thick,->] (0,0,0) -- (0,0,0.5) node[above=-1]{$t_3$};
		\draw[] (0,0,0) -- (1,0,0) -- (1,1,0) -- (1,1,1) -- cycle;
		\draw[dashed] (0,0,0) -- (1,1,0);
		\draw[] (1,0,0) -- (1,1,1);
		\draw[color=blue] (0,0,0) -- (1,0.333333,0.222222) -- (1,1,0.666667) -- (1,1,1) node[anchor=south] {0.33, 0.67};
		
	\end{tikzpicture}
	\begin{tikzpicture}[scale=1.7,tdplot_main_coords]
		\coordinate (O) at (0,0,0);
		\draw[thick,->] (0,0,0) -- (1.5,0,0) node[below left=-3]{$t_1$};
		\draw[thick,->] (0,0,0) -- (0,1.1,0) node[right=-1]{$t_2$};
		\draw[thick,->] (0,0,0) -- (0,0,0.5) node[above=-1]{$t_3$};
		\draw[] (0,0,0) -- (1,0,0) -- (1,1,0) -- (1,1,1) -- cycle;
		\draw[dashed] (0,0,0) -- (1,1,0);
		\draw[] (1,0,0) -- (1,1,1);
		\draw[color=blue] (0,0,0) -- (1,0.666667,0.444444) -- (1,1,0.666667) -- (1,1,1) node[anchor=south] {0.67, 0.67};
		
	\end{tikzpicture}
	\begin{tikzpicture}[scale=1.7,tdplot_main_coords]
		\coordinate (O) at (0,0,0);
		\draw[thick,->] (0,0,0) -- (1.5,0,0) node[below left=-3]{$t_1$};
		\draw[thick,->] (0,0,0) -- (0,1.1,0) node[right=-1]{$t_2$};
		\draw[thick,->] (0,0,0) -- (0,0,0.5) node[above=-1]{$t_3$};
		\draw[] (0,0,0) -- (1,0,0) -- (1,1,0) -- (1,1,1) -- cycle;
		\draw[dashed] (0,0,0) -- (1,1,0);
		\draw[] (1,0,0) -- (1,1,1);
		\draw[color=blue] (0,0,0) -- (1,1.000000,0.666667) -- (1,1,0.666667) -- (1,1,1) node[anchor=south] {1.00, 0.67};
		
	\end{tikzpicture}
	\\
	\begin{tikzpicture}[scale=1.7,tdplot_main_coords]
		\coordinate (O) at (0,0,0);
		\draw[thick,->] (0,0,0) -- (1.5,0,0) node[below left=-3]{$t_1$};
		\draw[thick,->] (0,0,0) -- (0,1.1,0) node[right=-1]{$t_2$};
		\draw[thick,->] (0,0,0) -- (0,0,0.5) node[above=-1]{$t_3$};
		\draw[] (0,0,0) -- (1,0,0) -- (1,1,0) -- (1,1,1) -- cycle;
		\draw[dashed] (0,0,0) -- (1,1,0);
		\draw[] (1,0,0) -- (1,1,1);
		\draw[color=blue] (0,0,0) -- (1,0.000000,0.000000) -- (1,1,1.000000) -- (1,1,1) node[anchor=south] {0.00, 1.00};
		
	\end{tikzpicture}
	\begin{tikzpicture}[scale=1.7,tdplot_main_coords]
		\coordinate (O) at (0,0,0);
		\draw[thick,->] (0,0,0) -- (1.5,0,0) node[below left=-3]{$t_1$};
		\draw[thick,->] (0,0,0) -- (0,1.1,0) node[right=-1]{$t_2$};
		\draw[thick,->] (0,0,0) -- (0,0,0.5) node[above=-1]{$t_3$};
		\draw[] (0,0,0) -- (1,0,0) -- (1,1,0) -- (1,1,1) -- cycle;
		\draw[dashed] (0,0,0) -- (1,1,0);
		\draw[] (1,0,0) -- (1,1,1);
		\draw[color=blue] (0,0,0) -- (1,0.333333,0.333333) -- (1,1,1.000000) -- (1,1,1) node[anchor=south] {0.33, 1.00};
		
	\end{tikzpicture}
	\begin{tikzpicture}[scale=1.7,tdplot_main_coords]
		\coordinate (O) at (0,0,0);
		\draw[thick,->] (0,0,0) -- (1.5,0,0) node[below left=-3]{$t_1$};
		\draw[thick,->] (0,0,0) -- (0,1.1,0) node[right=-1]{$t_2$};
		\draw[thick,->] (0,0,0) -- (0,0,0.5) node[above=-1]{$t_3$};
		\draw[] (0,0,0) -- (1,0,0) -- (1,1,0) -- (1,1,1) -- cycle;
		\draw[dashed] (0,0,0) -- (1,1,0);
		\draw[] (1,0,0) -- (1,1,1);
		\draw[color=blue] (0,0,0) -- (1,0.666667,0.666667) -- (1,1,1.000000) -- (1,1,1) node[anchor=south] {0.67, 1.00};
		
	\end{tikzpicture}
	\begin{tikzpicture}[scale=1.7,tdplot_main_coords]
		\coordinate (O) at (0,0,0);
		\draw[thick,->] (0,0,0) -- (1.5,0,0) node[below left=-3]{$t_1$};
		\draw[thick,->] (0,0,0) -- (0,1.1,0) node[right=-1]{$t_2$};
		\draw[thick,->] (0,0,0) -- (0,0,0.5) node[above=-1]{$t_3$};
		\draw[] (0,0,0) -- (1,0,0) -- (1,1,0) -- (1,1,1) -- cycle;
		\draw[dashed] (0,0,0) -- (1,1,0);
		\draw[] (1,0,0) -- (1,1,1);
		\draw[color=blue] (0,0,0) -- (1,1.000000,1.000000) -- (1,1,1.000000) -- (1,1,1) node[anchor=south] {1.00, 1.00};
		
	\end{tikzpicture}

	\begin{lemma}
		Each $\frownie^n$ is smooth on the subsets $[\frac{i}{n+1}, \frac{i+1}{n+1}] \times I^n$ for $i = 0,...,n$
		and they satisfy Axioms \ref{AXA2}, \ref{AXA3}, \ref{AXA4}, and \ref{AXA5}.  
		\begin{proof}
			The statement about the piecewise smoothness and Axiom \ref{AXA3} immediately follow from the construction. Let us now consider Axiom \ref{AXA2}. Of particular importance is the first piece $I^n \times [0,\frac{1}{n+1}]$. It is straightforward to check that the restriction
			to the interior of $[0,\frac{1}{n+1}] \times I^n$ is an orientation preserving diffeomorphism whose image is the interior of $\Delta^{n+1}$. This implies that for each form $\omega$, we get
			\begin{align*}\int_{[0,\frac{1}{n+1}] \times I^n} (\frownie^n)^* \omega = \int_{\Delta^{n+1}} \omega. \end{align*}
			Furthermore, we see that the map restricted to $[\frac{1}{n+1},n] \times I^n$ factors through $\Delta^n$ since the first component remains unchanged for $t \geq \frac{1}{n+1}$. Therefore $\int_{[\frac{1}{n+1},1] \times I^n} (\frownie^n)^* \omega = 0$ and hence
			\begin{align*} \int_{I \times I^n} (\frownie^n)^* \omega = \int_{\Delta^{n+1}} \omega,\end{align*} as required. We continue with Axiom \ref{AXA4}. For given $n$ and $i = 1,...,n$ we define the map $\phi : I \to I$ by
			\begin{nalign}
				\phi(t) := \begin{cases}
					\frac{n+1}{n} t, \text{ for } t \leq \frac{i}{n+1}\\
					\frac{i}{n} , \text{ for } \frac{i}{n+1} \leq t \leq \frac{i+1}{n+1}\\
					\frac{n+1}{n}t - \frac{1}{n} , \text{ for } \frac{i+1}{n+1} \leq t\\
				\end{cases}
			\end{nalign}
			As before we see that $\phi$ is integral-compatible. A straightforward calculation shows that with this definition of $\phi$ the diagram of Axiom \ref{AXA4} is commutative. The proof of Axiom \ref{AXA5} is simpler. Here, it suffices to choose $\phi, \phi'$ equal to the identity.
		\end{proof}
	\end{lemma}
	
	\begin{theorem}
		There exists an admissible collection of maps $\smiley^n : I \times I^n \to \Delta^{n+1}$.
		\begin{proof}
			According to the previous lemma we need to reparametrize our maps \\$\frownie^n(-,s_1,...,s_n)$.  
			Choose some smooth map $\eta : I \to I$ such that $\eta|_{(0,1)} : (0,1) \to (0,1)$ is a diffeomorphism, $(\partial_t)^k \eta(0) = (\partial_t)^k \eta(1) = 0$ for all $k > 0$ and $\eta(0) = 0, \eta(1) = 1$. Then we define for each $n$ $\eta^n : I \to I$ on each interval $[\frac{i}{n+1}, \frac{i+1}{n+1}]$ using the maps $sc_i : [\frac{i}{n+1}, \frac{i+1}{n+1}] \to I$ by
			\begin{align*} \eta^n(t) = sc_i^{-1} \circ \eta \circ sc_i(t), t \in [\frac{i}{n+1}, \frac{i+1}{n+1}],\end{align*}
			see the following graphic.\\
			\begin{tikzpicture}[scale = 0.6]
				
				\begin{axis}[samples=500,domain=-6*pi:6*pi,restrict y to domain =-20:100]
					\addplot[black, domain = 0.0001:0.9999]plot (\x, {1/(1+ exp(1/(\x-1)+1/\x))}) node[right] {$\eta$};
				\end{axis}
			\end{tikzpicture}
			\begin{tikzpicture}[scale = 0.6]
				
				\begin{axis}[samples=500,domain=-6*pi:6*pi,restrict y to domain =-20:100]
					\addplot[black, domain = 0.0001:0.4999]plot (\x, {0.5/(1+ exp(1/(2*\x-1)+1/(2*\x)))});
					\addplot[black, domain = 0.5001:0.9999]plot (\x, {0.5+0.5/(1+ exp(1/(2*\x-2)+1/(2*\x-1)))}) node[right] {$\eta^1$};
				\end{axis}
			\end{tikzpicture}
			\begin{tikzpicture}[scale = 0.6]
				
				\begin{axis}[samples=500,domain=-6*pi:6*pi,restrict y to domain =-20:100]
					\addplot[black, domain = 0.0001:0.3333]plot (\x, {1/3/(1+ exp(1/(3*\x-1)+1/(3*\x)))});
					\addplot[black, domain = 0.3334:0.6666]plot (\x, {1/3+ 1/3/(1+ exp(1/(3*\x-2)+1/(3*\x-1)))});
					\addplot[black, domain = 0.6667:0.9999]plot (\x, {2/3+ 1/3/(1+ exp(1/(3*\x-3)+1/(3*\x-2)))}) node[right] {$\eta^2$};
				\end{axis}
			\end{tikzpicture}\\
			By the properties of $\eta$, this is again smooth and it is clearly integral-compatible. We then define our desired maps $\smiley^n$ by
			\begin{align*} \smiley^n := \frownie^n \circ (\eta^n \times id).\end{align*}
			Per construction, $\smiley^n$ is now smooth (because the Taylor expansions of $\eta^n$ at the points $i/(n+1)$ are trivial and the directional derivatives of $\frownie^n$ exist). Since it just a reparametrization of the $I_t$ factor, leaving the points $\frac{i}{n+1}$ invariant, one may verify that Axioms A2 to A5 remain valid.
			
		\end{proof}
	\end{theorem}
	Note that in the construction of the maps $\smiley^n$ as in the previous theorem the picture as before remains the same. The only difference is the parametrisation of the paths $t \mapsto \smiley^n(t,s_1,...,s_n)$.
	
	\subsection{The higher Riemann-Hilbert functor}
	From now on, we fix an admissible collection of maps $\smiley^n : I \times I^n  \to \Delta^{n+1}$.
	\begin{def1}
		In the following, we will consider smooth maps $\sigma : \Delta^{n+1} \to M$. We then obtain the map $\sigma \circ \smiley^n : I \times I^n \to M$ and thus an adjoint map $\tilde \sigma : I^n \to PM(\sigma_{(0)},\sigma_{(n+1)})$. By $\text{hol}'^{\sigma}_{f_1,...,f_m}$, we will then denote the pullback of $\text{hol}^{x_0,x_1}_{f_1,...,f_m}$ along $\tilde \sigma$. Moreover, note that $\text{hol}^{x_0,x_1}_h$ is closed by Corollary $\ref{holCor}$ since $h$ is closed. Therefore, also $\text{hol}_h'^{\sigma}$ is closed and since the cohomology of $I^n$ is trivial, we obtain a unique smooth map $s^{\sigma} : I^n \to \mathbb K$ such that $ds^{\sigma} = -\text{hol}_h'^{\sigma}$ and $s(1,...,1) = 0$. We now set 
		\begin{align*} \text{hol}^{\sigma}_{f_1,...,f_m} := (-1)^{\frac{n(n+1)}{2}} e^{s^{\sigma}} \text{hol}'^{\sigma}_{f_1,...,f_m}.\end{align*}
		In the case $m = 0$, we let $\text{hol}_{f_1,...,f_m}$ be $PT$. Similarly, we define $PT'^{\sigma}$ and $PT^{\sigma}$.
	\end{def1}
	\begin{lemma}
		\label{coherentHol}
		Let $\sigma : \Delta^{n+1} \to M$ be smooth. The following identity holds
		\begin{nalign}
			\label{holIsCoherent}
			\int_{I^n} d(\text{hol}^{\sigma}_{f_1,...,f_m}) &= \sum_{i = 1}^n \int_{I^{n-1}} (-1)^{n-i+1} \text{hol}^{\partial^i \sigma}_{f_1,...,f_m}\\
			&\quad+ \sum_{k = 1}^{m-1} \sum_{i = 1}^n e^{\int_{\Delta^2} \sigma_{(0,i,n+1)}^* h }(-1)^{i+n} \int_{I^{n-i}}  \text{hol}^{\sigma_{(i,...,n+1)}}_{f_1,...,f_k} \int_{I^{i-1}} \text{hol}^{\sigma_{(0,...,i)}}_{f_{k+1},...,f_{m}}\\
			&\quad+ e^{\int_{\Delta^2} \sigma_{(0,n,n+1)}^* h } PT(\sigma_{(n,n+1)}) \int_{I^{n-1}}   \text{hol}^{\sigma_{(0,...,n)}}_{f_1,...,f_m} \\
			&\quad+   e^{\int_{\Delta^2} \sigma_{(0,1,n+1)}^* h }(-1)^{n+1} \int_{I^{n-1}}  \text{hol}^{\sigma_{(1,...,n+1)}}_{f_1,...,f_m} PT(\sigma_{(0,1)}).
		\end{nalign}
		\begin{proof}
			By Stokes Theorem, we obtain
			\begin{nalign}
				\int_{I^n} d(\text{hol}^{\sigma}_{f_1,...,f_m})
				&= \sum_{i = 1}^{n}  \int_{R_i(I^n)} \text{hol}^{\sigma}_{f_1,...,f_m}
				+ \sum_{i = 1}^{n}  \int_{F_i(I^n)} \text{hol}^{\sigma}_{f_1,...,f_m}.\\
			\end{nalign}
			We consider the first summand for each $i$. Denote by $\phi$ the map as in Axiom \ref{AXA4} and let $PM_\phi$ be the diffeological space of smooth paths $\gamma$ such that $\gamma \circ \phi$ is also smooth. Then we obtain the commutative diagram 
			\[
			\begin{tikzcd}
				R_i(I^n) \arrow[d, "\text{pr}_i"'] \arrow[r, "\tilde \sigma|_{R_i(I^n)}"] & PM \\
				I^{n-1} \arrow[r, "\widetilde {\partial^i \sigma}"]                            & PM_{\phi} \arrow[u, "\gamma \mapsto \gamma \circ \phi"]     .                                         
			\end{tikzcd}
			\]
			By Proposition  \ref{holProp} part \ref{holPropP5}, we obtain
			\begin{align*} \text{hol}'^{\sigma}_{f_1,...,f_m} = \text{pr}_i^* \text{hol}'^{\partial^i \sigma}_{f_1,...,f_m} \text{ on } R_i(I^n) \end{align*}
			since $\phi : I \to I$ is piecewise smooth and fixes the endpoints. Moreover, we have (on $R_i(I^n)$)
			\begin{align*} d(s^{\partial^i \sigma} \circ \text{pr}_i) = -\text{pr}_i^* \text{hol}'^{\partial^i \sigma}_h = -\text{hol}'^{\sigma}_{h}.\end{align*}
			But we also have $s^{\partial^i \sigma} \circ \text{pr}_i(1,...,1) = s^{\partial^i \sigma}(1,...,1) = 0 = s^{ \sigma}(1,...1)$, so in fact 
			\begin{align*} s^{\partial^i \sigma} \circ \text{pr}_i = s^{ \sigma} \text{ on } R_i(I^n).\end{align*}
			Combining these results, we obtain
			\begin{align*} \text{hol}^{\sigma}_{f_1,...,f_m} = (-1)^{n} \text{pr}_i^* \text{hol}^{\partial^i \sigma}_{f_1,...,f_m} \text{ on } R_i(I^n). \end{align*}
			Now since the map $\text{pr}_i : R(I^n) \to I^{n-1}$ changes the orientation by the sign $(-1)^{i+1}$, we combine
			\begin{nalign}
				\sum_{i = 1}^{n}  \int_{R_i(I^n)} \text{hol}^{\sigma}_{f_1,...,f_m} &= \sum_{i = 1}^{n}  \int_{R_i(I^n)}  (-1)^{n} \text{pr}_i^* \text{hol}^{\partial^i \sigma}_{f_1,...,f_m}\\
				&= \sum_{i = 1}^n \int_{I^{n-1}} (-1)^{n-i+1} \text{hol}^{\partial^i \sigma}_{f_1,...,f_m}
			\end{nalign}
			which is the first summand as in the statement of the lemma. We now consider the other terms
			\begin{align*} \int_{F_i(I^n)} \text{hol}^{\sigma}_{f_1,...,f_m}.\end{align*}
			By Proposition \ref{holProp} part \ref{holPropP4} we have
			\begin{nalign}\label{eq43}\text{hol}'^{\sigma}_{f_1,...,f_m} &= \tilde \sigma^*[\text{hol}_{f_1,...,f_m}(\frac{i}{n+1},1)\text{PT}(0,\frac{i}{n+1}) + \text{PT}(\frac{i}{n+1},1)\text{hol}_{f_1,...,f_m}(0,\frac{i}{n+1})\\
				& + \sum_{k = 1}^{m-1} \text{hol}_{f_1,...,f_k}(\frac{i}{n+1},1)\text{hol}_{f_{k+1},...,f_m}(0,\frac{i}{n+1})].\end{nalign}
			By Axiom \ref{AXA5}, we obtain commutative diagrams
			\[\begin{tikzcd}
				{
					{[0,\frac{i}{n+1}]^2 \times F_i(I^n)}} \arrow[d, "{sc^{\times 2} \times \ \text{pr}^{1,i-i}}"] \arrow[rr, "{sc^{\times 2} \times ([\gamma \mapsto \gamma \circ sc^{-1}] \circ \tilde \sigma)}"] &  & I^2 \times PM                                                                                   &  & {[\frac{i}{n+1},1]^2 \times F_i(I^n)} \arrow[d, "{sc^{\times 2} \times \ \text{pr}^{i+1,n}}"] \arrow[rr, "{sc^{\times 2} \times ([\gamma \mapsto \gamma \circ sc^{-1}] \circ \tilde \sigma)}"] &  & I^2 \times PM                                                                                     \\
				I^2 \times I^{i-1} \arrow[rr, "{\text{id} \times \widetilde{\sigma_{(0,...,i)}}}"']                                                                                                                                     &  & I^2 \times PM_{\phi} \arrow[u, "\text{id} \times \mathcal (\gamma \mapsto \gamma \circ \phi)"'] &  & I^2 \times I^{n-i} \arrow[rr, "{\text{id} \times \widetilde{\sigma_{(i,...,n+1)}}}"']                                                                                                          &  & I^2 \times PM_{\phi'} \arrow[u, "\text{id} \times \mathcal (\gamma \mapsto \gamma \circ \phi')"'].
			\end{tikzcd}\]
			Proposition \ref{holProp} part \ref{holPropP3} and part \ref{holPropP5} yield
			\begin{nalign}
				&\tilde \sigma^* \text{hol}_{(-)}(0,\frac{i}{n+1}) = (\text{pr}^{1,i-1})^* \text{hol}_{(-)}'^{\sigma_{(0,...,i)}} ,\\
				&\tilde \sigma^* \text{hol}_{(-)}(\frac{i}{n+1},1) = (\text{pr}^{i+1,n})^* \text{hol}_{(-)}'^{\sigma_{(i,...,n+1)}} \text{ on } F_i(I^n).
			\end{nalign}
			Combining with equation \ref{eq43} we see that
			\begin{nalign}\text{hol}'^{\sigma}_{f_1,...,f_m} &= (\text{pr}^{i+1,n})^* \text{hol}_{(f_1,...,f_m)}'^{\sigma_{(i,...,n+1)}}(\text{pr}^{1,i-1})^*\text{PT}'^{\sigma_{(0,...,i)}}\\
				&\quad+ (\text{pr}^{i+1,n})^*\text{PT}'^{\sigma_{(i,...,n+1)}}(\text{pr}^{1,i-1})^* \text{hol}_{(f_1,...,f_m)}'^{\sigma_{(0,...,i)}}\\
				&\quad+ \sum_{k = 1}^{m-1} (\text{pr}^{i+1,n})^* \text{hol}_{(f_1,...,f_k)}'^{\sigma_{(i,...,n+1)}}(\text{pr}^{1,i-1})^* \text{hol}_{(f_{k+1},...,f_m)}'^{\sigma_{(0,...,i)}} \text{ on } F_i(I^n).
			\end{nalign}
			Next, we obtain for each $i$
			\begin{align*} s^{\sigma_{(i,...,n+1)}} \circ \text{pr}^{i+1,n} + s^{\sigma_{(0,...,i)}} \circ \text{pr}^{1,i-1} + s^{\sigma}(1,...,1,0,1,...,1) = s^{\sigma} \end{align*}
			by comparing both sides after applying the differential, using proposition \ref{holProp} part \ref{holPropP4} and evaluating both sides at $(1,...,1,0,1,...,1)$. Then we combine to 
			\begin{nalign}\text{hol}^{\sigma}_{f_1,...,f_m} &= (-1)^{in} e^{s^{\sigma}(1,...,1,0,1,...,1)}[(\text{pr}^{i+1,n})^* \text{hol}_{(f_1,...,f_m)}^{\sigma_{(i,...,n+1)}}(\text{pr}^{1,i-1})^*\text{PT}^{\sigma_{(0,...,i)}}\\ &\quad+ (\text{pr}^{i+1,n})^*\text{PT}^{\sigma_{(i,...,n+1)}}(\text{pr}^{1,i-1})^* \text{hol}_{(f_1,...,f_m)}^{\sigma_{(0,...,i)}}\\
				&\quad+ \sum_{k = 1}^{m-1} (\text{pr}^{i+1,n})^* \text{hol}_{(f_1,...,f_k)}^{\sigma_{(i,...,n+1)}}(\text{pr}^{1,i-1})^* \text{hol}_{(f_{k+1},...,f_m)}^{\sigma_{(0,...,i)}}] \text{ on } F_i(I^n)
			\end{nalign}
			using the identity
			\begin{align*} n(n+1) = (n+1-i)(n-i) + 2(n+1-i)i + i(i-1).\end{align*}
			Therefore, we conclude by tracking orientations and using Fubinis theorem
			\begin{align*}
				& \sum_{i = 1}^{n}  \int_{F_i(I^n)} \text{hol}^{\sigma}_{f_1,...,f_m} \\
				&= \sum_{k = 1}^{m-1} \sum_{i = 1}^n (-1)^{i+n} e^{s^{\sigma}(1,...,1,0,1,...,1)} \int_{I^{n-i}}  \text{hol}^{\sigma_{(i,...,n+1)}}_{f_1,...,f_k} \int_{I^{i-1}} \text{hol}^{\sigma_{(0,...,i)}}_{f_{k+1},...,f_{m}}\\
				&\quad+ e^{s^{\sigma}(1,...,1,0)} PT(\sigma_{(n,n+1)}) \int_{I^{n-1}}   \text{hol}^{\sigma_{(0,...,n)}}_{f_1,...,f_m} \\
				&\quad+ e^{s^{\sigma}(0,1,...,1)} (-1)^{n+1} \int_{I^{n-1}}  \text{hol}^{\sigma_{(1,...,n+1)}}_{f_1,...,f_m} PT(\sigma_{(0,1)}).
			\end{align*}
			Comparing this with equation with \ref{holIsCoherent}, we see that it remains to show that \begin{align*}s^{\sigma}(1,...,1,0,1,...,1) = \int_{\Delta^2} \sigma_{(0,i,n+1)}^* h.\end{align*}
			By writing $\sigma_{(0,i,n+1)}$ as an iterated face of $\sigma$ and using the reasoning as before repeatedly, we obtain that 
			\begin{align*} s^{\sigma}(1,...,1,0,1,...,1) = s^{\sigma_{(0,i,n+1)}}(0).\end{align*}
			We calculate using Stokes Theorem 
			\begin{nalign}
				s^{\sigma_{(0,i,n+1)}}(0)
				&= - \int_I ds^{\sigma_{(0,i,n+1)}} = \int_I \text{hol}_h'^{\sigma_{(0,i,n+1)}}.
			\end{nalign}
			Per definition of $\text{hol}_h$, we have that 
			\begin{align*} \partial_t \text{hol}_h(0,t) = \iota_{\partial_t}(h)\end{align*} and $\text{hol}_h(0,0) = 0$. Explicitly, the solution of this initial value problem is
			\begin{align*} \text{hol}_h(0,t) = \int_{0}^{t} dt' \text{ } \iota_{\partial_{t'}}(h)\end{align*}
			and we obtain using the above equation
			\begin{nalign}
				\label{eqHDef}
				s^{\sigma_{(0,i,n+1)}}(0) &= \int_I (\int_{0}^{1} dt \text{ } \iota_{\partial_{t}}(h))\\
				&= \int_{I^2} (\smiley^1)^* \sigma_{(0,i,n+1)}^* h.
			\end{nalign}
			Now by Axiom $A2$, this integral is equal to $\int_{\Delta^2} \sigma_{(0,i,n+1)}^* h$, hence the proof is concluded.
		\end{proof}
	\end{lemma}
	\begin{remark}
		\label{remPT}
		We can do the same calculation as above for $PT^{\sigma}$ instead of $\text{hol}^{\sigma}_{f_1,...,f_m}$ to obtain the equation 
		\begin{nalign}
			\int_I d(PT^{\sigma}) = -PT(\sigma_{(0,2)}) + e^{\int_{\Delta^2} \sigma^* h}PT(\sigma_{(1,2)}) PT(\sigma_{(0,1)})
		\end{nalign}
		for $2$-simplices $\sigma : \Delta^2 \to M$.
	\end{remark}
	
	We can now define our $A_\infty$-functor.
	\begin{def1}
		We define an $A_\infty$-functor $\mathcal{RH}(M)^{pre}[h] : \mathcal P(M)^{pre}[h] \to \text{Loc}(M)^{pre}[H]$, where $H = e^{DR(h)}$. On objects, we set
		\begin{align*}\mathcal{RH}(M)^{pre}[h](E, \nabla) := ((E_x)_{x \in M}, PT)\end{align*}
		where $PT$ denotes the parallel transport of $(E, \nabla)$. On composable morphisms $f_1,...,f_m$, we set
		\begin{align*}\mathcal{RH}(M)^{pre}[h](f_1,...,f_m)(\sigma) := \int_{I^n} \text{hol}^{\sigma}_{f_1,...,f_m} \end{align*}
		for each smooth simplex $\sigma : \Delta^{n+1} \to M$. On 0-simplices $\sigma : \Delta^0 \to M$, we set
		\begin{nalign}
			\mathcal{RH}(M)^{pre}[h](f_1,...,f_m)(\sigma) = \begin{cases}
				f_1(\sigma(0)) \text{ if } m = 1 \text{ and } |f_1| = 0,\\
				0 \text{ else }\\
			\end{cases}
		\end{nalign}
	\end{def1}
	\begin{theorem}
		\label{Atheorem}
		$\mathcal{RH}(M)^{pre}[h]$ is indeed a proper $A_\infty$-functor of cdg-categories.
		\begin{proof}
			It is clear that $((E_x)_{x \in M}, PT)$ indeed defines an object of $\text{Loc}(M)^{pre}[H]$. We next check that the functor respects the curvatures. Let $\sigma : \Delta^2 \to M$ be smooth. By the proof of \ref{holLemma}, we see that
			\begin{align*}d(PT'^{\sigma}) = \text{hol}'^{\sigma}_{R'}.\end{align*}
			We calculate 
			\begin{nalign}
				d(PT^{\sigma}) &= -ds^{\sigma} e^{s^\sigma} PT'^{\sigma} - e^{s^\sigma} d(PT'^{\sigma})\\
				&= -\text{hol}^{\sigma}_h PT'^{\sigma} - e^{s^\sigma}  \text{hol}'^{\sigma}_{R'}\\
				&= -\text{hol}^{\sigma}_{h \mathbbm 1_{E}} + \text{hol}^{\sigma}_{R'}\\
				&= \text{hol}^{\sigma}_{R}.
			\end{nalign}
			By Remark \ref{remPT}, we then obtain
			\begin{align*} -PT(\sigma_{(0,2)}) + e^{\int_{\Delta^2} \sigma^* h}PT(\sigma_{(1,2)})PT(\sigma_{(0,1)})  = \int_I \text{hol}^{\sigma}_{R} = \mathcal{RH}(M)^{pre}[h](R)(\sigma).\end{align*}
			The $A_\infty$-relations follow from Lemma \ref{coherentHol}, Corollary \ref{holCor}, and Lemma \ref{holHLemma}. It is proper since by the definition of $\text{hol}$, we have that $\text{hol}_{f_1,...,f_m} = 0$ if there is an $i$ such that $|f_i| = 0$.
		\end{proof}
	\end{theorem}
	\subsection{ $\mathcal {RH}(M)^{\infty}[h]$ is an $A_\infty$-quasi equivalence }
	We want to show that the induced $A_\infty$-functor of dg-categories \begin{align*}\mathcal{RH}(M)^{\infty}[h] := \text{Tw}(\mathcal{RH}(M)^{pre}[h]): \mathcal P(M)^{\infty}[h] \to \text{Loc}(M)^{\infty}[H]\end{align*} is a quasi-equivalence. The idea is to apply Theorem  \ref{EquivalenceTheorem}. Thus, we start with the following observation.
	\begin{lemma}
		\label{isSpectral}
		$\mathcal P(M)^{pre}[h]$ and $\text{Loc}(M)^{pre}[H]$ are split.
		\begin{proof}
			We can construct the kernel $k : (K, P^K) \to (E, P^E)$ of a closed degree $0$ morphism $f$ in $\text{Loc}(M)^{pre}[H]$ from $(E, P^E)$ to $(F, P^F)$ by setting $(k_x, K_x)$ to be the kernel of $f_x : E_x \to F_x$ for every $x$. Since $f$ is closed, we see that $P^E$ restricts to a map $P^K$ on $K$, making $k$ the kernel of $f$. It is straightforward to check the universal property. Similarly, we see that $\text{Loc}(M)^{pre}[H]$ has cokernels. Next, we observe that a short exact sequence yields a short exact sequence for every $x \in M$. The choice of a splitting is then the same thing as a splitting for every $x \in M$ (as the splittings are not required to be closed) which exists as we are working over fields. Therefore $\text{Loc}(M)^{pre}[H]$ is split.\\ \\
			Next, we consider $\mathcal P(M)^{pre}[h]$. Let $f : (E, \nabla^E) \to (F, \nabla^F)$ be closed and of degree $0$. Then $f$ is just a map of vector bundles that respects the connections. Therefore, $f$ must have constant rank which implies that $f$ has a kernel and a cokernel on the level of vector bundles. Now as $\nabla(f) = 0$, we see that the connection on $E$ restricts to a connection on $\text{ker}(f)$ and similarly for the cokernel. For any short exact sequence of vector bundles, we obtain splittings by the Serre-Swan correspondence. So $\mathcal P(M)^{pre}[h]$ is also split.
		\end{proof}
	\end{lemma}
	\begin{corollary}
		Let $[h] \in H^2(M,\mathbb C)$ which is not in the image of $2\pi i : H^2(M,\mathbb Q) \to H^2(M,\mathbb C)$. Then $\mathcal P(M)^{\infty}[h]$ is quasi-equivalent to $0$.
		\begin{proof}
			As $\mathcal P(M)^{\infty}[h]$ is split and twist-closed it suffices to show (by Theorem \ref{EquivalenceTheorem}) that there are no non-trivial $h$-curved local systems when $h$ is not in the image of $2\pi i : H^2(M,\mathbb Q) \to H^2(M,\mathbb C)$. So let $(E, \nabla)$ be a vector bundle with curvature $\nabla^2 = h$. The determinant line bundle $\text{det}(E) = \Lambda^n E$ has curvature $[n \cdot h] \in 2\pi i \cdot H^2(M, \mathbb Z)$. Thus, $[\frac{h}{2\pi i}]$ is rational.
		\end{proof}
	\end{corollary}
	\begin{lemma}
		Let $(E, P)$ be an object of $\text{Loc}(M)^{0}[H]$ for $H(\sigma) = e^{\int_{\Delta^2} \sigma^*h}$. Let $x$ and $y$ be points of $M$. Then $P$ viewed as a map between diffeological spaces
		\[P : PM(x,y) \to \text{Hom}(E_x,E_y)\]
		is smooth and we have 
		\[ d(P) = \text{hol}_h P.\]
		\begin{proof}
			Let $\phi : U \to PM(x,y)$ be a plot of $PM(x,y)$. We compute the partial derivatives of $P \circ \phi$. Let $p \in U \subset \mathbb R^n$, $1 \leq i \leq n$ and $t \in \mathbb R$ sufficiently small. As $(E, P)$ has curvature $H$ we obtain that $P \circ \phi(p + te_i) = H(\sigma) P \circ \phi(p)$ where $\sigma$ is a 2-simplex with $\partial_0 \sigma = c_{y}$, $\partial_1 \sigma = P \circ \phi(p + te_i)$, and $\partial_2 \sigma = P \circ \phi(p)$. As seen in equation \ref{eqHDef}, we have $\int_{\Delta^2} \sigma^* h = \int_I \gamma^* \text{hol}_h$ where $\gamma : I \to PM(x,y)$, $\gamma(s) = \phi(p + st e_i)$. Equivalently, we can write this as  $\int_0^t \mu^* \text{hol}_h(\partial_s) ds$ where $\mu(s) = \phi(p + s e_i)$. But $t \mapsto \int_0^t \mu^* \text{hol}_h(\partial_s) ds$ is clearly smooth and we thus obtain
			\begin{nalign} 
				\partial_i (P \circ \phi) = \phi^* \text{hol}_h(\partial_i) P \circ \phi
			\end{nalign}
			This equation can also be solved as follows. $\text{hol}_{h}$ is closed since $h$ is closed (for example by Corollary \ref{holCor}). Fix some $x_0 \in U$ and a neighbourhood such that there exists a solution $p$ of $dp = \text{hol}_h$ with $p(x_0) = 0$. Then $e^p P \circ \phi(x_0)$ solves the same equation as the one above and by uniqueness we obtain $e^p P \circ \phi(x_0) = P \circ \phi$ (on the chosen neighbourhood). Thus, $P \circ \phi$ is smooth and
			\begin{nalign} 
				d(P) = \text{hol}_h P.
			\end{nalign}
		\end{proof}
	\end{lemma}
	\begin{lemma}
		\label{lemmaQes}
		$\mathcal{RH}(M)^{0}[h] : \mathcal P(M)^{0}[h] \to \text{Loc}(M)^{0}[H]$ is essentially surjective.
		\begin{proof}
			Let $(E, P)$ be an object of $\text{Loc}(M)^{0}[H]$ of dimension $m$. We let $F := \sqcup_{x \in M} E_x$ and we need to find trivializations that have smooth transition functions. For each chart $\varphi_\alpha : \mathbb R^n \to M$ of $M$, we define a trivialization $F|_{\mathbb R^n} \to \mathbb K^m$ as follows. To simplify notation, we will identify $\mathbb R^n$ with the image of $\varphi_\alpha$. First we choose any isomorphism $\phi(0) : E_0 \to \mathbb K^m$. Then we set $\phi(x) : E_x \to E_0 \to \mathbb K^m$, $\phi(x) = P(t \mapsto (1-t)x) \circ \phi(0)$. By the proof of the Poincaré-Lemma, there is a 1-form $\omega$ on $\mathbb R^n$ given by 
			\begin{align*} \omega_{p}(v) := \int_0^1 t h_{tp}(p, v) dt\end{align*} 
			which satisfies $d\omega = h$. We note that $\omega_x(x) = 0$ for every $x \in \mathbb R^n \cong T_{x} \mathbb R^n$. This means that the parallel transport of the connection $d + \omega$ is equal to the identity along the paths $t \mapsto t \cdot p$ for every $p \in \mathbb R^n$. Furthermore, the connection $d + \omega$ has curvature $\omega \wedge \omega + d\omega = d\omega = h$. This implies that the parallel transport of the connection $d + \omega$ is precisely equal to $P|_{\mathbb R^n}$. It remains to show that the transitions functions are smooth. It then automatically follows that the connections agree on intersections of charts, as they give rise to the same parallel transport. It is also clear, that in this case $\mathcal{RH}(M)^{0}[h]$ maps this object to $(E, P)$. \\ \\
			To show that the transition functions are smooth, consider charts $\varphi_\alpha$ with image $U_\alpha$ and $\varphi_\beta$ with image $U_\beta$. Denote the corresponding trivializations by $\phi_\alpha$ and $\phi_\beta$. Let $\eta : I \to I$ be a smooth non-decreasing map, fixing the endpoints such that $\eta|_{[1/2 - \epsilon, 1/2 + \epsilon]} = 1/2$ for some $\epsilon > 0$. Furthermore set $x = \phi_\alpha(0), y = \phi_\alpha(1)$. Then we define a smooth map 
			\[ U_\alpha \cap U_\beta \to PM(x,y), p \mapsto (\gamma_\beta^p)^{-1} * \gamma_\alpha^p \circ \eta \]
			where $\gamma_{\alpha / \beta}^p(t) = \phi_{\alpha / \beta}(t \phi_{\alpha / \beta}^{-1}(p))$ for $t \in I$. As $P$ is reparametrisation invariant (since $\int_{\Delta^2} \sigma^* h = 0$ if $\sigma$ factors through $I$), we see that the transition functions are smooth by the previous Lemma using the definition of the trivialisations.
		\end{proof}
	\end{lemma}
	We have to recall the de Rham isomorphism for flat vector bundles. 
	Let $(Q, \nabla)$ be a flat vector bundle. Then note that for any space with trivial fundamental group $A$ and smooth map $A \to M$ the pullback of $(Q, \nabla)$ to $A$ canonically trivializes to the fiber $Q_x$ for any point $x \in A$. We define the singular cochain complex $C^*(M, (Q, \nabla))$ with values in $(Q, \nabla)$, in degree $n$, to consist of assignments $f$ that map any smooth simplex $\sigma : \Delta^n \to M$ to an element $f(\sigma) \in Q_{\sigma_{(0)}}$. The differential is given by 
	\begin{align*} (df)(\sigma) := \sum_{i = 1}^{n+1} (-1)^{i+n+1} f(\partial^i \sigma) +  (-1)^{n+1}PT(\sigma_{(0,1)}^{-1}) [f(\partial^0 \sigma)]\end{align*}
	which squares to $0$, using that $Q$ is flat. Moreover, there is a map
	\begin{align*}\text{DR} : \Omega^*(M, Q) \to C^*(M, (Q, \nabla))\end{align*}
	given by
	\begin{align*} DR(q)(\sigma) = (-1)^{\frac{(n+1)n}{2}}\int_{\Delta^n} \sigma^* q \in Q_{\sigma_{(0)}}\end{align*}
	where the integral makes sense by the above considerations. This map is compatible with the differentials and is known to be a quasi-isomorphism (the usual sheaf-theoretic proof works; one has to take some extra care when defining barycentric subdivision). Now let $(E, \nabla^E)$ and $(F, \nabla^F)$ be objects of $\mathcal P(M)^{0}[h]$. Then $(Q, \nabla) := (\text{Hom}(E,F), \nabla)$ is a flat vector bundle. The following holds:
	\begin{lemma}
		\label{relateLemma}
		There is an isomorphism of complexes \begin{align*}C^*(M, (Q, \nabla)) \cong \text{Hom}^*_{\text{Loc}(M)^{0}[H]}(\mathcal {RH}(E), \mathcal {RH}(F))\end{align*} making the diagram
		\[
		\begin{tikzcd}
			& {\Omega^*(M, \text{Hom}(E,F))} \arrow[rd, "{\mathcal{RH}(M)^{0}[h]}"] \arrow[ld, "\text{DR}"'] &                                                                                                            \\
			{C^*(M, (Q, \nabla))} &                                                                                                       & {\text{Hom}_{\text{Loc}(M)^{0}[H]}(\mathcal {RH}(E), \mathcal {RH}(F))} \arrow[ll] \arrow[ll, "\cong"']
		\end{tikzcd}\]
		commute.
		\begin{proof}
			We define \begin{align*} P : \text{Hom}^*_{\text{Loc}(M)^{0}[H]}(\mathcal {RH}(E), \mathcal {RH}(F)) \to C^*(M, (Q, \nabla))\end{align*} by
			\begin{align*} P(f)(\sigma) := PT(\sigma_{(0,n)}^{-1}) \circ f(\sigma) \in \text{Hom}(E_{\sigma_{(0)}},F_{\sigma_{(0)}})\end{align*}
			for $f$ of degree $n$ and $\sigma : \Delta^n \to M$. We leave it to the reader to verify that this map indeed fulfills the properties as stated in the Lemma. The main observation one has to use is that $e^{s^\sigma} PT(\sigma_{(0,n+1)}^{-1}) PT'^{\sigma}(0,1) = \mathbbm 1$ which may be be checked differentiating both sides and comparing both sides at $(1,..,1)$. The integration as in the the definition of $\mathcal {RH}(M)^{0}[h]$ is related to the integration in the definition of $DR$ by means of Axiom \ref{AXA2}. 
		\end{proof}
	\end{lemma}
	\begin{corollary}
		\label{qffCor}
		The functor
		\begin{align*}\mathcal{RH}^0(M)[h] : \mathcal P^0(M)[h] \to \text{Loc}^0(M)[H]\end{align*}
		is quasi-fully faithful.
		\begin{proof}
			Since $DR$ is a quasi-isomorphism we obtain that $\mathcal{RH}(M)^0[h]$ must also be a quasi-isomorphism on morphisms by the previous Lemma.
		\end{proof}
	\end{corollary}
	\begin{theorem}
		\label{hCurvedDGEquivalence}
		The induced $A_\infty$-functor 
		\begin{align*}\mathcal{RH}(M)^{\infty}[h] := \text{Tw}(\mathcal{RH}(M)^{pre}[h]) : \mathcal P(M)^{\infty}[h] \to \text{Loc}(M)^{\infty}[H]\end{align*}
		is a quasi-equivalence of dg-categories. 
	\end{theorem}
	\begin{proof}
		We have seen that $\mathcal{RH}^0(M)[h]$ is a quasi-equivalence in Lemma \ref{lemmaQes} and Corollary \ref{qffCor}. By Lemmas \ref{lemmaP}, \ref{lemmaLoc} and \ref{isSpectral} we are in the setting of Theorem \ref{EquivalenceTheorem} which tells us that $\mathcal{RH}(M)^{\infty}[h]$ is a quasi-equivalence.
	\end{proof}
	
	\subsection{Homotopy invariance}
	\label{HomotopyInvarianceSection}
	So far, we have only considered a single manifold $M$. Given a map of manifolds $f : M \to N$, we would like to have an induced functor $\mathcal P(N)[h] \to \mathcal P(M)[f^* h]$. If we are given two such maps $M \to N$ which are homotopic, we want to compare the induced functors. In particular, we desire a dg-equivalence $\mathcal P(N)[h] \to \mathcal P(M)[f^* h]$ if $f$ is a homotopy equivalence. In the following, we will sometimes consider the category $\mathcal P(M)$ when $M$ is a manifold with boundary, of particular importance will be $M \times I$. We note that the definition of holonomy forms and the categories $\mathcal P(A)$ can be done for arbitrary diffeological spaces $A$ (but of course we then may not obtain the equivalences as above.) 
	
	\begin{def1}
		Let $M$ and $N$ be manifolds (possibly with boundary, or even diffeological spaces). Let $h$ be a closed 2-form on $N$ and let $f : M \to N$ be smooth. 
		We define dg-functor of cdg-categories
		\begin{align*} \mathcal P(f)^{pre} : \mathcal P(N)^{pre}[h] \to \mathcal P(M)^{pre}[f^*h].\end{align*}
		On objects, we set $\mathcal P^{pre}(f)(E, \nabla) := (f^*E, f^*\nabla)$ and for a morphism $g$, we set
		\begin{align*} \mathcal P^{pre}(f) = f^*(g).\end{align*}
		It is straightforward to see that this indeed defines a dg-functor.
	\end{def1}
	\begin{def1}
		Let $h$ be a closed 2-form on $M$ and $\omega$ a 1-form. We define a dg-functor of cdg-categories
		\begin{align*}\mathcal P^{pre}[\omega] : \mathcal P(M)^{pre}[h] \to \mathcal P(M)^{pre}[h + d\omega].\end{align*}
		On objects, we set $\mathcal P^{pre}[\omega](E, \nabla) := (E, \nabla + \omega)$ and it is the identity on morphisms. Note that 
		$\mathcal P^{pre}[-\omega]$ is clearly inverse to $\mathcal P^{pre}[\omega]$. 
	\end{def1}
	
	\begin{def1}
		Consider the inclusions $\iota_1 : M \to M \times I, x \mapsto (x,1)$ and $\iota_0 : M \to M \times I, x \mapsto (x,0)$ and let $h$ be a closed 2-form on $M \times I$. Furthermore, we define the map
		\begin{align*} \beta : M \to P(M \times I), \beta(x) = (t \mapsto (x,t))\end{align*}
		and note that $ev^1 \circ \beta = \iota_1$ and $ev^0 \circ \beta = \iota_0$. We can pullback the holonomy form $\text{hol}_h$ along $\beta$ to obtain the 1-form $\omega_h := \beta^* \text{hol}_h$ which satisfies $d\omega_h = \iota_1^* h - \iota_0^* h$ using that $h$ is closed and Lemma \ref{holLemma}. Then we want to compare the composition 
		\begin{align*} \mathcal P^{pre}[\omega_h] \circ \mathcal P(\iota_0)^{pre} : \mathcal P(M \times I)^{pre}[h] \to \mathcal P(M)^{pre}[\iota_0^* h] \to \mathcal P(M)^{pre}[\iota_1^* h]\end{align*}
		with 
		\begin{align*} \mathcal P(\iota_1)^{pre} : \mathcal P(M \times I)^{pre}[h] \to \mathcal P(M)^{pre}[\iota_1^* h].\end{align*}
		We define an $A_\infty$-natural transformation
		\begin{align*} \eta^{pre} : \mathcal P^{pre}[\omega_h] \circ \mathcal P(\iota_0)^{pre} \Rightarrow \mathcal P(\iota_1)^{pre}.\end{align*}
		Let $(E, \nabla)$ be an object of $\mathcal P(M \times I)^{pre}[h]$. We denote $(E_0, \nabla_0) := \mathcal (\iota_0)^{pre}(E, \nabla)$ and $(E_1, \nabla_1) := \mathcal P^{pre}(\iota_1)(E, \nabla)$. Then, $\mathcal P^{pre}[\omega_h] \circ \mathcal P(\iota_0)^{pre}(E, \nabla) = (E_0, \nabla_0 + \omega_h)$. Let $PT$ be the parallel transport of $(E, \nabla)$. On objects, we set 
		\begin{align*} \eta^{pre}_{(E, \nabla)} := \beta^* PT \in \text{Hom}(E_0, E_1).\end{align*}
		For composable morphisms $f_1,...,f_n$ in $\mathcal P(M \times I)^{pre}[h]$, we set 
		\begin{align*} \eta^{pre}(f_1,...,f_n) := \beta^* \text{hol}_{f_1,...,f_n}.\end{align*}
		The $A_{\infty}$-relations precisely follow by pulling back the equation of Lemma \ref{holLemma} along $\beta$ and using Lemma \ref{holHLemma}. \\ \\
		We define another $A_\infty$-natural transformation 
		\begin{align*} (\eta^{pre})^{-1} : \mathcal P(\iota_1)^{pre} \Rightarrow \mathcal P^{pre}[\omega_h] \circ \mathcal P(\iota_0)^{pre}\end{align*} 
		which we shall show to constitute an inverse to $\eta^{pre}$. We now consider 
		\begin{align*} \beta^{-1} : M \to P(M \times I), \beta(x) = (t \mapsto (x,1-t)).\end{align*} For objects, we set
		\begin{align*} (\eta^{pre})^{-1}_{(E, \nabla)} := (\beta^{-1})^* PT \in \text{Hom}(E_0, E_1).\end{align*}
		For composable morphisms $f_1,...,f_n$ in $\mathcal P(M \times I)^{pre}[h]$ we set 
		\begin{align*} (\eta^{pre})^{-1} (f_1,...,f_n) := (\beta^{-1})^* \text{hol}_{f_1,...,f_n}.\end{align*}
	\end{def1}
	\begin{lemma}
		$\eta^{pre}$ and $(\eta^{pre})^{-1}$ are inverse to one another.
		\begin{proof}
			We generalize the above maps $\beta$ and $\beta^{-1}$ to
			\begin{align*} \tilde \beta : M \times I \to P(M \times I), (x,s) \mapsto ( t \mapsto (x, st))\end{align*}
			and 
			\begin{align*} \tilde \beta^{-1} : M \times I \to P(M \times I), (x,s) \mapsto ( t \mapsto (x, (1-t)s )).\end{align*}
			Then, $\tilde \beta_1 := \tilde \beta((-),1) = \beta$ and $\tilde \beta_1^{-1} = \beta^{-1}$. We consider the expression  
			\begin{align*} \sum_{k = 0}^n (\tilde \beta^{-1})^* \text{hol}_{f_1,...,f_k} \tilde \beta^* \text{hol}_{f_{k+1},...,f_n}\end{align*}
			whose vanishing for $s = 1$ precisely implies the claim (for $n \geq 1$). By the definition of $\text{hol}$, all summands in this equation are $0$ for $s = 0$. We show that this expression becomes $0$ after differentiating using $\nabla_{s}$. By Definition \ref{holPropA1} in Proposition \ref{holProp} and property \ref{holPropP4}, we deduce the following cancellation for each $k$:
			\begin{align*} \nabla_{s} [(\tilde \beta^{-1})^* \text{hol}_{f_1,...,f_k}] \tilde \beta^* \text{hol}_{f_{k+1},...,f_n} + (\tilde \beta^{-1})^* \text{hol}_{f_1,...,f_{k-1}} \nabla_{s} \tilde \beta^* \text{hol}_{f_{k},...,f_n} = 0.\end{align*}
			Together with $\nabla_s (\tilde \beta^{-1})^* PT = 0 = \nabla_s (\tilde \beta)^* PT$ this proves the claim.
		\end{proof}
	\end{lemma}
	\begin{corollary}
		Let $f,g$ be smooth homotopic maps of manifolds $f,g : M \to N$ via $H : M \times I \to N$. Then the induced dg-functors
		\begin{align*}\mathcal P(f)^{\infty}  : \mathcal P(N)^{\infty}[h] \to \mathcal P(M)^{\infty}[f^* h]\end{align*} 
		and
		\begin{align*}\mathcal P^{\infty}[\omega_{H^* h}] \circ \mathcal P(g)^{\infty}  : \mathcal P(N)^{\infty}[h] \to \mathcal P(M)^{\infty}[f^* h]\end{align*} 
		are $A_\infty$-naturally isomorphic dg-functors.
	\end{corollary}
	\begin{corollary}
		Let $f$ be a smooth homotopy equivalence of manifolds $f : M \to N$. Then 
		\begin{align*}\mathcal P(f)^{\infty} : \mathcal P(N)^{\infty}[h] \to \mathcal P(M)^{\infty}[f^* h]\end{align*} 
		is a dg-quasi equivalence.
	\end{corollary}

	\section{Projectively flat graded vector bundles and projective representations}
	
	\subsection{Projective representations of simplicial monoids}
	\label{sectionProjRepMonoids}
	In the following, we let $S$ be a simplicial monoid.
	Elements of $S_1$ will be referred to as homotopies. The product on each $S_i$ will be written as $*$.
	
	\begin{def1}
		\begin{enumerate}
			\label{barAugmentedDefinition}
			\item Let $C_*$ be some complex. We denote by $sC_*$ the shifted complex with entries \begin{align*} sC_n = C_{n-1} \end{align*} and whose differential is induced by the differential on $C_*$.
			\item  Let $C_*$ be a unital (homological) dg-algebra (over $\mathbb Z$) and let $\epsilon : C_* \to \mathbb Z$ be an augmentation, i.e., $\epsilon$ is a map of unital dg-algebras. Denote the kernel ideal of $\epsilon$ by $\bar C_*$. The bar-construction of $(C_*, \epsilon)$ is the dg-coalgebra $B_{\epsilon}(C_*)$ given by 
			\begin{align*} B_{\epsilon}(C_*) := \bigoplus_{n \geq 0} (s \bar C_*)^{\otimes_{\mathbb Z} n}\end{align*}
			endowed with the derivation
			\begin{nalign}
				b(a_1 \otimes ... \otimes a_n) &= \sum_{i = 1}^n (-1)^{|a_1|+...+|a_{i-1}|+1}a_1 \otimes ... \otimes d(a_i) \otimes ... \otimes a_n\\
				&+\sum_{i = 1}^{n-1} (-1)^{|a_1|+...+|a_{i}|+1}a_1 \otimes ... \otimes a_i a_{i+1} \otimes ... \otimes a_n
			\end{nalign}
			The coproduct is
			\begin{align*} B_{\epsilon}(C_*) \to B_{\epsilon}(C_*)  \otimes B_{\epsilon}(C_*), a_1 \otimes ... \otimes a_n \mapsto \sum_{i = 0}^{n} (a_1 \otimes ... \otimes a_i) \otimes (a_{i+1} \otimes ... \otimes a_n). \end{align*}
			The degree $|a|$ in the above formula refers to the degree of $a$ as an element in $s \bar C_*$ (in contrast to the conventions often used in the literature). 
			\item Here, we define the Eilenberg Zilber map for arbitrary simplicial sets $X$ and $Y$. Denote by $C_*(X)$ the associated chain complex of non-degenerate chains of $X$. Then we define the Eilenberg-Zilber map of chain complexes by the linear extension of
			\begin{align*} EZ : C_*(X) \otimes C_*(Y) \to C_*(X \times Y), x \otimes y \mapsto \sum_{(\mu, \nu) \in Sh(p,q)} \text{sgn}(\mu, \nu) s_\nu(x) \times s_\mu(y)\end{align*}
			where $x \in X_p$ and $y \in Y_q$.
			$Sh(p,q)$ is the set of $(p,q)$-shuffles and $s_\mu$ is the iterated face $s_{\mu_p -1} \circ ... \circ s_{\mu_1 -1}$. Details can be found in \cite[\href{https://kerodon.net/tag/00RF}{Tag 00RF}]{Kerodon} where $EZ$ is called the shuffle product. 
			\item We will consider the quotient chain-complex $C_*(S)$ of non-degenerate chains with values in $S$. This comes equipped with a multiplication
			\begin{align*} \cdot := C_*(*) \circ EZ :  C_*(S) \otimes C_*(S) \to C_*(S \times S) \to  C_*(S)\end{align*}
			and an augmentation $\epsilon : C_0(S) \to \mathbb Z$, $\epsilon(\sum_s \lambda_s s) = \sum_s \lambda_s$.
			We thus obtain a dg-coalgebra $B_{\mathbb Z}(S,1) := B_{\epsilon}(C_*(S))$. We will also use the notation $B^*(S, A) := \text{Hom}_{\mathbb Z}(B_{\mathbb Z}(S,1), A)$ for an abelian group $A$.
		\end{enumerate}
	\end{def1}
	A basis of $\text{ker}(\epsilon : C_0(S) \to \mathbb Z)$ is given by $a - e$ for $a \in S_0 - \{e\}$. To simplify notation, we will write $a$ instead of $a - e$ in $\text{ker}(\epsilon : C_0(S) \to \mathbb Z)$.
	
	\begin{def1}
		\begin{enumerate}
			\item A curvature element of $S$ is 2-cocycle 
			\begin{align*}H \in \text{Hom}_{\mathbb Z}^2(B_{\mathbb Z}(S,1), \mathbb K^*) = B^2(S, \mathbb K^*).\end{align*}
			\item A logarithmic curvature element of $S$ is 2-cocycle 
			\begin{align*}H \in \text{Hom}_{\mathbb Z}^2(B_{\mathbb Z}(S,1), \mathbb K) = B^2(S, \mathbb K).\end{align*}
			\item Let $H$ and $H'$ be curvature elements. A twist from $H$ to $H'$ is an element
			\begin{align*}a \in \text{Hom}_{\mathbb Z}^1(B_{\mathbb Z}(S,1), \mathbb K^*) = B^1(S, \mathbb K^*)\end{align*}
			satisfying $b(a) = H^{-1}H'$ (in particular, the existence of such a twist means that $H$ is cohomologous to $H'$).
		\end{enumerate}
	\end{def1}
	
	\begin{remark}
		For any logarithmic curvature element $H$, we have that $e^{-H}$ is a curvature element.
		For $\mathbb K = \mathbb R$, we see that a curvature element comes from a logarithmic curvature if and only if it has values in $\mathbb R_{> 0}$.
	\end{remark}
	
	\begin{def1}
		Let $H$ be a curvature element of $S$. We define a new simplicial monoid $S^H$. In degree $k$, we define $S^H$ to be 
		\begin{align*} (S^H)_k = \mathbb K^* \times S_k.\end{align*}
		We write its elements as $a.s := (a,s)$ for $a \in \mathbb K^*$ and $s \in S_k$. The multiplication is given by
		\begin{align*} (a.s) *^H (b.t) := (ab H(s_n \otimes t_n )).(s * t).\end{align*}
		We continue by defining the face maps:
		\begin{nalign}
			\partial_i^H (a.s) = a.\partial_i(s), i = 0,...,n-1\\
			\partial_n^H (a.s) = (H(s_{(n-1, n)})a).\partial_n (s)
		\end{nalign}
		for $a \in \mathbb K^*$ and $s \in S_n$. Using that $H$ is closed one verifies the simplicial identities and the compatibility with the monoid structure. Using the Eilenberg Zilber map and the monoid structure $*$ on $S^H$ we obtain a multiplication 
		\begin{align*} \cdot := C_*(*) \circ EZ : \tilde C_*(S^H) \otimes \tilde C_*(S^H) \to \tilde C_*(S^H \times S^H) \to \tilde C_*(S^H).\end{align*}
		Next, we define the ideal-complex $D_*(S^H)$ in degree $i$ to be generated by expressions $(y.s) - y(1.s)$, $y \in \mathbb K^*, s \in S_i$. It is easy to see that the differential on $C_*(S^H)$ indeed restricts to a differential on $D_*(S^H)$. Then we define 
		\begin{align*} C_*^H(S) := \quot{\tilde C_*(S^H)}{D_*(S^H)}.\end{align*}
		With its induced multiplication, this turns $C_*^H(S)$ into a unital dg-algebra and we note that the underlying vector spaces of $C_*^H(S)$ and $C_*(S) \otimes \mathbb C$ agree.
		We simply write its elements in degree $k$ as $\sum_s \lambda_s s $ for $\lambda_s \in \mathbb K$ and $s \in S_k$. 
	\end{def1}
	\begin{def1}
		\label{twistMap}
		Let $H$ and $H'$ be curvature elements and $a$ a twist from $H$ to $H'$. We define a map of simplicial monoids
		\begin{align*} S[a] : S^{H'} \to S^{H}\end{align*} by the linear extension of
		\begin{align*} S[a](s) = a(s_{(n)}) s, s \in S_n.\end{align*}
		It is invertible with inverse $S[a^{-1}]$.
	\end{def1}
	\begin{def1}
		Let $H$ be a curvature element.
		The \emph{unnormalized bar construction} of $(S,H)$ is the dg-coalgebra
		\begin{align*} B^{u}(S,H) := \bigoplus_{n \geq 0} (sC^H_*(S))^{\otimes n}.\end{align*} The formulas for the differential and coproduct are the same as for the bar construction of an augmented dg-algebra (compare \ref{barAugmentedDefinition}).
	\end{def1}
	\begin{remark}
		The bar construction is usually defined for an \emph{augmented} dg-algebra (as in Definition \ref{barAugmentedDefinition}). The problem in our case is that there is no natural choice of augmentation for $C_*^H(S)$. An augmentation of $C_*^H(S)$ is the same as a 1-dimensional object of $\text{Rep}(S)^0[H]$ (which will be defined in the following), which then corresponds to a line bundle with curvature $H$. If for a given $H$, there is a line bundle with curvature $H$, we already obtain $\text{Rep}(S)^{pre}[H] \cong \text{Rep}(S)^{pre}[0]$. It is therefore crucial to not assume the existence of an augmentation for $C_*^H(S)$.
	\end{remark}
	\begin{def1}
		\label{RepDef}
		We define the cdg-category $\text{Rep}(S)^{pre}[H]$ of $H$-projective representations of $S$.
		An object of $\text{Rep}(S)^{pre}[H]$ is a finite-dimensional $\mathbb K$-vector space $E$ together with a map
		\begin{align*} \mu_E : S_0 \to \text{Hom}_{\mathbb K}(E,E)\end{align*}
		satisfying $\mu_E(e) = 1$. 
		(This is not yet required to be multiplicative). For two objects $E$ and $F$, we define the morphism spaces to be
		\begin{align*} \text{Hom}^*(E,F) := \text{Hom}_{\mathbb K}(B^{u}(S,H), \text{Hom}_{\mathbb K}(E,F)).\end{align*}
		The connection is determined by the formula
		\begin{nalign}
			(\nabla f)(a_1 \otimes ... \otimes a_k) &= (-1)^{|f|+1} f(b(a_1 \otimes ... \otimes a_k))\\
			&\quad- f(a_1 \otimes ... \otimes a_{k-1}) a_{k}\\
			&\quad+  (-1)^{|f|}a_1 f(a_2 \otimes ... \otimes a_{k})
		\end{nalign}
		where we set
		\begin{align*}f(a_1 \otimes ... \otimes a_{k-1}) a_{k} := \begin{cases} f(a_1 \otimes ... \otimes a_{k-1}) \circ (\mu_E(a_k)) \text{ if } a_k \in C^H_0(S), \\ 0 \text{ else } \end{cases}\end{align*}
		and we define $a_1 f(a_2 \otimes ... \otimes a_{k})$ similarly.
		The composition is induced by the coproduct on $B(S,H)$ and the composition of linear maps of vector spaces. The curvature of $E$ is given by 
		\begin{nalign} 
			\label{EqCurvature}&R_E(a)(v) = b(a).v \\
			&R_E(a_1 \otimes a_2) = \mu_E(a_1 a_2)v - \mu_E(a_1)\mu_E(a_2)v = (a_1 a_2).v - a_1.(a_2.v)\end{nalign} 
		for $a \in s C^H_2(S) =  C^H_1(S), v \in E$ and $a_1 \otimes a_2 \in s C^H_1(S) \otimes s C^H_1(S)$.
	\end{def1}
	\begin{remark}
		The curvature $R_E$ measures the failure of $E$ to define an $A_\infty$-module for the dg-algebra $C^H_*(S)$.
	\end{remark}
	\begin{remark}
		Let $H$ be a curvature element. If $H = 1$, we can define, for any finite dimensional vector space $E$, the flat object 
		$(E, \mu_E)$
		where $\mu_E(s) = 1$ for all $s \in S_0$. In the case $H \neq 1$, this object has non-vanishing curvature.
	\end{remark}
	
	\begin{lemma}
		\label{RepIsCdg}
		$\text{Rep}(S)^{pre}[H]$ is a sufficiently twist-closed cdg-category.
		\begin{proof}
			As before one can check that it is a cdg-category using that $H$ is closed.
			It remains to check that it is sufficiently twist-closed. An element $a \in \text{Hom}^1(E,E)$ such that $a(e) = 0$ is just a equivalently map $S_0 \to  \text{Hom}_{\mathbb K}(E,E)$ sending $e$ to $0$. Any element $a$ of the form $a = f \circ \nabla(g)$ has the property that $a(e) = 0$. The object $(E, \mu_E + a)$ satisfies the desired property.
		\end{proof}
	\end{lemma}
	As before we denote the associated dg-category of twisted complexes as $\text{Rep}(S)[H] := \text{Tw}(\text{Rep}(S)^{pre}[H])$.
	\begin{remark}
		By Definition \ref{twistMap}, we see that a twist $a$ induces a cdg-functor
		\begin{align*} \text{Rep}^{pre}(S[a]) : \text{Rep}^{pre}(S)[H] \to  \text{Rep}^{pre}(S)[H']\end{align*} and thus an equivalence 
		\begin{align*} \text{Rep}(S[a]) : \text{Rep}(S)[H] \to  \text{Rep}(S)[H']\end{align*}
		with inverse $\text{Rep}(S[a^{-1}])$.
	\end{remark}
	\begin{def1}
		\label{projectiveRepresentation}
		\begin{enumerate}
			\item Let $H$ be a curvature element. An \emph{$H$-projective representation of $S$} is an object of the associated dg-category $\text{Rep}(S)[H]$.
			\item A \emph{projective representation of $S$} is an $H$-projective representation for some curvature element $H$.
			\item A \emph{logarithmic projective representation of $S$} is an $e^{-H}$-projective representation for some logarithmic curvature element $H$.
		\end{enumerate}
	\end{def1}
	\begin{lemma}
		\label{repSpectral}
		$\text{Rep}^{pre}(S)[H]$ is split for every curvature element $H$.
		\begin{proof}
			Similar to Lemma \ref{isSpectral}.
		\end{proof}
	\end{lemma}
	We now want to define the category of all projective representations of $S$. In this category, the morphism sets between projective representations with different curvature elements can be non-empty.
	
	\begin{remark}
		A twist $a$ from $H$ to $H'$ is equivalently just a flat object of the form $(\mathbb K, a)$ of $\text{Rep}(S)^{pre}[H^{-1}H']$, i.e., an object of $\text{Rep}(S)^{0}[H^{-1}H']$ of the form $(\mathbb K, a)$.
	\end{remark}
	\begin{def1}
		\begin{enumerate}
			\item Let $H, H', H''$ be curvature elements, $a$ a twist from $H$ to $H'$, and $\tilde a$ a twist from $H'$ to $H''$. Then $\tilde a a$ is a twist from $H$ to $H''$.		
			\item Let $H, H'$ and $a$ be as above. Let $E$ be an object of $\text{Rep}(S)^{\infty}[H]$ and $F$ an object of $\text{Rep}(S)^{\infty}[H']$. We set
			\begin{align*} \text{PHom}^{H, H'}_a(E,F) := \quot{H^0(\text{Hom}_{ \text{Rep}(S)[H']}(\text{Rep}(S[a])(E),F))}{\mathbb K^*}.\end{align*}
			Furthermore, we can identify
			\begin{align*} \text{Hom}_{\text{Rep}(S)[H'']}(\text{Rep}(S[\tilde a a])(E),\text{Rep}(S[\tilde a])F) = \text{Hom}_{\text{Rep}(S)[H']}(\text{Rep}(S[a])E,F)\end{align*}
			which yields a composition
			\begin{align*} \text{PHom}^{H', H''}_{\tilde a}(F,G) \times \text{PHom}^{H, H'}_a(E,F) \to \text{PHom}^{H, H''}_{a \tilde a}(E,G).\end{align*}
			Using the above property, we see that this is associative.
		\end{enumerate}
	\end{def1}
	\begin{def1}
		\begin{enumerate}
			\item We define the category $\text{PRep}(S)$ of projective representations of $S$. Its objects are non-zero projective representations of $S$ (see Definition \ref{projectiveRepresentation}). For two objects $E \in  \text{Rep}(S)[H]$ and $F \in \text{Rep}(S)[H']$, the morphism set is defined to be
			\begin{align*} \text{Hom}_{\text{PRep}(S)}(E,F) := \bigsqcup_a \text{PHom}^{H, H'}_a(E,F) \end{align*}
			where the disjoint union is indexed by all twists $a$ from $H$ to $H'$. The composition is induced by the composition defined above.
			\item The category $\text{LPRep}(S)$ of logarithmic projective representations of $S$ is the full subcategory of $\text{PRep}(S)$ consisting of those objects which are logarithmic projective representations of $S$.
		\end{enumerate}
	\end{def1}
	
	\subsection{Curved representations of the loop space}
	\label{sectionCurvedRepLoop}
	We fix $x_0 \in M$.
	\begin{def1}
		\label{hochDefs}
		\begin{enumerate}
			\item Let $P^{[0,\infty)}(M)$ be the diffeological space of smooth paths $\gamma : [0, \infty) \to M$. We define the \emph{diffeological space of Moore loops} by
			\begin{nalign} \Omega_{x_0}(M) := &\{(\gamma, l) | \gamma([0,\epsilon)) = \{x_0\}, \gamma([l-{\epsilon}, \infty)) = \{x_0\}, \text{ for some } \epsilon > 0\} \\ &\subset P^{[0,\infty)}(M) \times [0,\infty).
			\end{nalign}
			We denote the components by $(\gamma, l)_1 := \gamma$ and $(\gamma, l)_2 := l$.
			The smooth structure is \textbf{not} inherited from $P^{[0,\infty)}(M) \times [0,\infty)$. The plots
			are maps $\phi : U \to \Omega_{x_0}(M)$ such that
			\begin{enumerate}
				\item the composition $U \to \Omega_{x_0}(M) \to P^{[0,\infty)}(M) \times [0,\infty)$ is smooth, and
				\item for every $x \in U$ there exists an open neighbourhood $U_x \subset U$ of $x$ and an $\epsilon > 0$ such that $\phi_1(y)|_{[0, \epsilon)} = x_0 = \phi_1(y)|_{(\phi_2(y) - \epsilon, \infty)}$ for all $y \in U_x$.
			\end{enumerate}
			With this smooth structure the concatenation product
			\begin{align*}  * : \Omega_{x_0}(M) \times  \Omega_{x_0}(M) \to  \Omega_{x_0}(M), (\gamma, l) \times (\gamma',l') \mapsto (t \mapsto \begin{cases} \gamma'(t) , t \leq l'\\ \gamma(t-l'), t > l' \end{cases}, l+l')\end{align*}
			is easily seen to be smooth. The main advantage of $\Omega_{x_0}(M)$ is that this product is associative.
			\item We consider the singular simplicial set of smooth simplices $\text{Sing}(\Omega_{x_0}(M))$ (the definition can be repeated verbatim as for smooth manifolds). The concatenation product on $\Omega_{x_0}(M)$ turns $\text{Sing}(\Omega_{x_0}(M))$ into a simplicial monoid via
			\begin{align*} \Delta^n \overset{\text{diag}}\to \Delta^n \times \Delta^n \overset{\alpha \times \beta} \to \Omega_{x_0}(M) \times \Omega_{x_0}(M) \overset{*}\to \Omega_{x_0}(M).\end{align*} The unit in degree $k$ is given by the simplex $\Delta^n \to \Omega_{x_0}(M)$ sending each point to the constant path $(c_{x_0}, 0) \in \Omega_{x_0}(M)$.
			\item We consider the smooth map 
			\begin{align*}\alpha : \Omega_{x_0}(M) \to PM, (\gamma, l) \mapsto \gamma \circ sc_l\end{align*}
			where $sc_l : [0,1] \to [0,l]$. In particular, when working over $\Omega_{x_0}(M)$, we will simply write $\text{hol}_{f_1,...,f_n}$ instead of $\alpha^* \text{hol}_{f_1,...,f_n}$. We also define a smooth map
			\begin{align*}\tilde \alpha : \Omega_{x_0}(M) \to \Omega_{x_0}(M), (\gamma, l) \mapsto (\gamma \circ sc_l,1).\end{align*}
		\end{enumerate}
	\end{def1}
	
	In the following, we will often simply write $\Omega_{x_0}(M)$ instead of $\text{Sing}(\Omega_{x_0}(M))$.
	In particular, we can now consider the cdg-category $\text{PRep}^{pre}(\Omega_{x_0}(M))[H]$ for any curvature $H$. Let $h \in Z\Omega^2_B(M)$ and define the logarithmic projective curvature element 
	\begin{align*} H(\tau) := - \int_I \tau^* \alpha^* \text{hol}_h \end{align*} for $\tau : \Delta^1 = I \to \Omega_{x_0}(M)$. We want to define a proper $A_\infty$-functor of cdg-categories
	\begin{align*} \mathcal P(M)^{pre}[h] \to \text{PRep}^{pre}(\Omega_{x_0}(M))[e^{-H}]\end{align*}
	which yields a quasi-equivalence of associated dg-categories. We will do this in a similar way as before and adopt the notation.
	\begin{def1}
		Let $\tau : \Delta^n \to \Omega_{x_0}(M)$ be smooth. We postcompose $\tau$ with $\alpha$ to obtain a map $\alpha \circ \tau : \Delta^n \to PM(x_0,x_0)$ and thus a map \begin{align*}\tilde \tau : I \times \Delta^n \to M\end{align*} with $\tilde \tau(1,-) = x_0 = \tilde \tau(0,-)$. We denote by $\widetilde{\text{hol}}'^{\tau}_{f_1,...,f_m}$ the pullback of $\text{hol}^{x_0, x_0}_{f_1,...,f_m}$ along $\alpha \circ \tau$. Then we set
		\begin{align*}\widetilde{\text{hol}}^{\tau}_{f_1,...,f_m} := (-1)^{\frac{n(n+1)}{2}} e^{s^{\tau}}  \widetilde{\text{hol}}'^{\tau}_{f_1,...,f_m}\end{align*}
		where $s^\tau$ is the unique smooth function on $\Delta^n$ with $d(s^\tau) = -\widetilde{\text{hol}}'^{\tau}_{h}$ and $s^{\tau}(1,...,1) = 0$.
	\end{def1}
	\begin{lemma}
		\label{tauLemma}
		The following identity holds
		\begin{nalign}
			\int_{\Delta^n} d(\widetilde{\text{hol}}^{\tau}_{f_1,...,f_m}) = \sum_{i = 0}^{n-1} (-1)^{i+n} \int_{\Delta^{n-1}} \widetilde{\text{hol}}^{\partial^i \tau}_{f_1,...,f_m} + e^{\int_I \tau_{(n-1, n)}^* \text{hol}_h} \int_{\Delta^{n-1}} \widetilde{\text{hol}}^{\partial^n \tau}_{f_1,...,f_m}
		\end{nalign}
	\end{lemma}
	\begin{proof}
		The proof is analogous to the proof of \ref{coherentHol} (but much simpler): Apply Stokes Theorem and use the definition of $\widetilde{\text{hol}}^{\tau}_{f_1,...,f_m}$. The last face carries the factor $e^{\int_I \tau_{(n-1, n)}^*\text{hol}_h}$ since (the right side fulfills the defining properties of $s^{\partial^n \tau}$)
		\begin{align*} s^{\partial^n \tau} = s^\tau \circ q^n - s^\tau(1,...,1,0)\end{align*}
		and 
		\begin{align*} s^\tau(1,...,1,0) = - \int_I d(\tau_{(n-1,n)}^* s^\tau) = \int_I \tau_{(n-1,n)}^*\text{hol}_h.\end{align*}
	\end{proof}
	Recall the definition of the maps $\alpha, \tilde \alpha$ from Definition \ref{hochDefs} for the following Lemma.
	\begin{lemma}
		\label{lemmaAlpha}
		Let $A$ be a diffeological space equipped with smooth maps \begin{align*}p,q : A \to \Omega_{x_0}(M).\end{align*}
		Consider the (non-commutative) diagram
		\[
		\begin{tikzcd}
			A \arrow[rr, "{(p,q)}"] &  & \Omega_{x_0}(M) \times \Omega_{x_0}(M) \arrow[rr, "*"] \arrow[rd, "\tilde \alpha \times \tilde \alpha"'] &                                                             & \Omega_{x_0}(M) \arrow[r, "\alpha"] & PM \\
			&  &                                                                                                              & \Omega_{x_0}(M) \times \Omega_{x_0}(M) \arrow[ru, "*"'] &                                       &   .
		\end{tikzcd}\]
		Then we have the identity
		\begin{align*}  (\alpha \circ * \circ (p,q))^{*}  \text{hol}_{f_1,...,f_n} = (\alpha \circ * \circ (\tilde \alpha \times \tilde \alpha)\circ (p,q))^{*}  \text{hol}_{f_1,...,f_n}\end{align*} on $A$.
	\end{lemma}
	\begin{proof}
		We write out both compositions explicitly interpreted as maps $A \times I \to M$. We have
		\begin{nalign}
			(\alpha \circ * \circ (p,q))(t,a) = \begin{cases}
				q(t (p_2(a)+q_2(a)),a) \text{ for } t \leq \frac{q_2(a)}{p_2(a)+q_2(a)}\\
				p(t(p_2(a)+q_2(a))-q_2(a),a) \text{ else }\\
			\end{cases}
		\end{nalign}
		and 
		\begin{nalign}
			(\alpha \circ * \circ (\tilde \alpha \times \tilde \alpha) \circ (p,q))(t,a) = \begin{cases}
				q(q_2(a) 2t,a) \text{ for } t \leq \frac{1}{2}\\
				p(p_2(a) 2(t-\frac{1}{2}),a) \text{ else }.\\
			\end{cases}
		\end{nalign}
		It suffices to prove the claim in the case that $A \to \Omega_{x_0}(M)$ is a plot of $\Omega_{x_0}(M)$ and furthermore, we show the identity locally. Let $a \in A$. There are two cases.
		\begin{enumerate}
			\item Case $p_2(a)+q_2(a) \neq 0$: Then there exists an open neighbourhood $U_a \subset A$ of $a$ such that  $p_2(x)+q_2(x) \neq 0$ for all $x \in U_a$. Then we can define the piecewise smooth reparametrization
			\begin{align*} \phi : U_\alpha \times I \to I, \phi(x,t) := \begin{cases}
					2t \frac{q_2(x)}{p_2(x) + q_2(x)} \text{ if } t \leq \frac{1}{2}\\
					\frac{2p_2(x)(t - \frac{1}{2})+ q_2(x)}{p_2(x)+q_2(x)} \text{ else}.
			\end{cases}\end{align*}
			Then we have 
			\begin{align*}(\alpha \circ * \circ (p,q))(\phi(x,t),x)  = (\alpha \circ * \circ (\tilde \alpha \times \tilde \alpha) \circ (p,q))(x,t)\end{align*}
			for all $(x,t) \in U_a \times I$. Therefore, part 5 of Proposition \ref{holProp} yields the desired identity on $U_a$.
			\item Case $p_2(a)+q_2(a) = 0$. By the definition of $\Omega_{x_0}(M)$, there is an open subset $a \in V_a \subset A$ and an $\epsilon > 0$ such that $p_1(x)|_{[0, \epsilon)} = q_1(x)|_{[0, \epsilon)} = x_0 = p_1(x)|_{(p_2(x) - \epsilon, \infty)} = q_1(x)|_{(q_2(x) - \epsilon, \infty)}$ for all $x \in V_a$. Now, choose an open subset $a \in U_a \subset V_a$ such that $p_2(x) + q_2(x) < \epsilon$ for all $x \in U_a$. This implies that $p_1(x) = q_1(x) = x_0$ for all $x \in U_a$. Thus, both compositions are given by $U_a \ni x \mapsto c_{x_0} \in PM$ so the identity holds on $U_a$. 
		\end{enumerate}
	\end{proof}
	\begin{lemma}
		\label{diffeoLemma}
		Let $p + q = n \in \mathbb N$. For each $(\mu, \nu) \in Sh(p,q)$, we consider the smooth map
		\begin{align*} \phi_{(\mu, \nu)} : \Delta^n \to \Delta^p \times \Delta^q, (t_1,...,t_n) \mapsto ((t_{\mu_1},...,t_{\mu_p}),(t_{\nu_1},...,t_{\nu_q})).\end{align*}
		Then the map 
		\begin{align*} \phi := \bigsqcup_{(\mu, \nu) \in Sh(p,q)}  \phi_{(\mu, \nu)} : \bigsqcup_{(\mu, \nu)} \Delta^n \to \Delta^p \times \Delta^q\end{align*}
		is a diffeomorphism restricted to open subsets of $\bigsqcup_{(\mu, \nu)} \Delta^n$ and $\Delta^p \times \Delta^q$ whose complements have measure zero. Furthermore each $\phi_{(\mu, \nu)}$ changes the orientation by the sign $sgn(\mu, \nu)$.
		\begin{proof}
			On the left side, we simply restrict to the interior $U := \bigsqcup_{(\mu, \nu)} \mathring \Delta^n$. On the right side, we restrict even further and set
			\begin{align*} V := \{((t_1,...,t_p),(s_1,...,s_q)) \in \mathring \Delta^p \times \mathring  \Delta^q | t_i \neq s_j \text{ for all } i,j)\}.\end{align*}
			It straightforward to see that the restriction of $\phi$ to $U$ factors through $V$ and we claim that this restriction is a diffeomorphism. To do this, one checks by hand that the map is injective, surjective and its differential is invertible.
		\end{proof}
	\end{lemma}
	\begin{corollary}
		\label{tauCor}
		Let $\tau : \Delta^p \to \Omega_{x_0}(M)$ and $\tau' : \Delta^q \to \Omega_{x_0}(M)$ be smooth and set $n = p+q$. Then the following identity holds
		\begin{align*} \sum_{(\mu, \nu) \in Sh(p,q)} \text{sgn}(\mu, \nu) \int_{\Delta^n} \widetilde{\text{hol}}_{f_1,...,f_n}^{s^\nu \tau *  s^\mu \tau'} = (-1)^{pq} \sum_{k = 0}^n \int_{\Delta^p} \widetilde{\text{hol}}_{f_1,...,f_k}^{\tau} \int_{\Delta^q} \widetilde{\text{hol}}_{f_{k+1},...,f_n}^{\tau'}.\end{align*}
		\begin{proof}
			We have
			\begin{nalign}
				\widetilde{\text{hol}}_{f_1,...,f_n}^{s^\nu \tau *   s^\mu \tau'} &= (-1)^{\frac{n(n+1)}{2}} e^{s^{s^\nu \tau *   s^\mu \tau'}} (s^\nu \tau * s^\mu \tau')^* \alpha^* \text{hol}_{f_1,...,f_n}\\
				&\stackrel{1.}{=} (-1)^{\frac{n(n+1)}{2}} e^{s^{s^\nu \tau *   s^\mu \tau'}} (\tilde \alpha \circ s^\nu \tau *   \tilde \alpha \circ s^\mu \tau')^* \alpha^* \text{hol}_{f_1,...,f_n}\\
				&\stackrel{2.}{=} (-1)^{\frac{n(n+1)}{2}} e^{s^{s^\nu \tau *   s^\mu \tau'}} (\tilde \alpha \circ s^\nu \tau *   \tilde \alpha \circ s^\mu \tau')^* \alpha^* (\sum_{k = 0}^{n} \text{hol}_{f_1,...,f_k}(\frac{1}{2},1)\text{hol}_{f_{k+1},...,f_n}(0,\frac{1}{2}))\\
				&\stackrel{3.}{=} (-1)^{\frac{n(n+1)}{2}} e^{s^{s^\nu \tau *   s^\mu \tau'}} \sum_{k = 0}^{n} (s^\nu \tau)^* \alpha^* \text{hol}_{f_1,...,f_k} (s^\mu \tau')^* \alpha^* \text{hol}_{f_{k+1},...,f_n}.
			\end{nalign}
			where we used  
			\begin{enumerate}
				\item Lemma \ref{lemmaAlpha},
				\item Proposition \ref{holProp} part 4 and
				\item Proposition \ref{holProp} part 3.
			\end{enumerate}
			By the same calculation, we obtain
			\begin{align*} (s^\nu \tau *   s^\mu \tau')^* \alpha^* \text{hol}_{h} = s^\nu \tau^* \alpha^* \text{hol}_h + s^\mu \tau'^* \alpha^* \text{hol}_h\end{align*}
			and since all degeneracy maps $\Delta^* \to \Delta^{* -1}$ send $(1,...,1)$ to $(1,...,1)$ this means that 
			\begin{align*} s^{s^\nu \tau *   s^\mu \tau'} = (p^\nu)^* s^{\tau} + (p^\mu)^* s^{\tau'}.\end{align*}
			Combining these two results, we obtain 
			\begin{nalign}
				\widetilde{\text{hol}}_{f_1,...,f_n}^{s^\nu \tau *   s^\mu \tau'} = (-1)^{pq} \sum_{k = 0}^n (p^{\nu})^*\widetilde{\text{hol}}^{\tau}_{f_1,...,f_k} (p^{\mu})^*\widetilde{\text{hol}}^{\tau'}_{f_{k+1},...,f_n}
			\end{nalign}
			using that 
			\begin{align*} n(n+1) = p(p+1) + q(q+1) + 2pq.\end{align*}
			By summing over all $(\mu, \nu)$ and integrating over $\Delta^n$, we obtain the claim by Lemma \ref{diffeoLemma}.
		\end{proof}
	\end{corollary}
	
	\begin{def1}
		We now define the $A_\infty$-functor 
		\begin{align*} \text{Rep}(M)^{pre}[h] : \mathcal P(M)^{pre}[h] \to \text{Rep}(\Omega_{x_0}(M))^{pre}[e^{-H}]\end{align*}
		of cdg-categories. On objects, we set 
		\begin{align*} \text{Rep}(M)^{pre}[h](E, \nabla) := (E_{x_0}, PT)\end{align*}
		where $PT$ is the parallel transport of $\nabla$. For composable morphisms $f_1,...,f_n$ in $\mathcal P(M)^{pre}[h]$ and $\tau_i : \Delta^{n_i} \to \Omega_{x_0}(M)$, $i = 1,...,m$, we set
		\begin{nalign}
			\text{Rep}(M)^{pre}[h](f_1,...,f_n)(\tau_1 \otimes ... \otimes \tau_m) = \begin{cases}
				\int_{\Delta^{n_1}} \widetilde{\text{hol}}^{\tau_1}_{f_1,...,f_n} \text{ if } m = 1,\\
				f_{1}(x_0) \text{ if } m = 0, n = 1 \text{ and } |f_1| = 0,\\
				0  \text{ else}.
			\end{cases}
		\end{nalign}
	\end{def1}
	\begin{theorem}
		$\text{Rep}(M)^{pre}[h] : \mathcal P(M)^{pre}[h] \to \text{Rep}(\Omega_{x_0}(M))^{pre}[e^{-H}]$ is a proper $A_\infty$-functor of cdg-categories.
		\begin{proof}
			The proof is similar to that of \ref{Atheorem}. The $A_\infty$-relations follow by combining Corollary \ref{holCor}, Lemma \ref{holHLemma}, Lemma \ref{tauLemma}. and Corollary \ref{tauCor}.
		\end{proof}
	\end{theorem}
	\begin{lemma}
		\label{qEss}
		$\text{Rep}(M)^{0}[h] : \mathcal P(M)^{0}[h] \to \text{Rep}(\Omega_{x_0}(M))^{0}[e^{-H}]$ is essentially surjective.
		\begin{proof}
			This is analogous to the proof of Lemma \ref{lemmaQes}.
		\end{proof}
	\end{lemma}
	To show quasi-fully faithfulness, we proceed as before, i.e., we reduce to the case of flat vector bundles. Let $(E, \nabla)$ and $(F, \nabla)$ be objects of $\mathcal P(M)^{0}[h]$. Then, $\text{Hom}(E,F)$ together with its induced connection is a flat vector bundle and we obtain an object $\text{Hom}(E,F)_{x_0}$ together with its induced parallel transport of $\text{PRep}^{0}(\Omega_{x_0}(M))[e^{0} = 1]$. Furthermore, we write $\mathbb K$ for the bundle $\mathbb K$ on $M$ with connection $\nabla = d$.
	\begin{lemma}
		\label{relateLemma2}
		There is an isomorphism of complexes
		\begin{nalign}
			\text{Hom}(B^u(S,e^{-H}), \text{Hom}_{\mathbb K}(E_x,F_y)) \cong \text{Hom}_{\mathbb K}(B^u(S,1), \text{Hom}(\mathbb K_x, [\text{Hom}(E,F)]_x))
		\end{nalign}
		making the diagram \\ \\
		\adjustbox{scale=0.9,left}{%
			\begin{tikzcd}[column sep=small]
				{\Omega^*(M, \text{Hom}(\mathbb K, \text{Hom}(E,F))) = \Omega^*(M, \text{Hom}(E,F))} \arrow[rrd, "{\text{Rep}^{0}(M)[0]}"] \arrow[dd, "{\text{Rep}^{0}(M)[h]}"'] &  &                                                                                                                 \\
				&  & \text{Hom}(B(S,1), \text{Hom}(\mathbb K_x, [\text{Hom}(E,F)]_x)) \arrow[lld, "\cong"] \\
				{\text{Hom}(B(S,e^{-H}), \text{Hom}_{\mathbb K}(E_x,F_x))}                                                   &  &                                                                                                                
			\end{tikzcd}
		}
		commute.
		\begin{proof}
			The proof is analogous to that of Lemma \ref{relateLemma}. The isomorphism is given by the linear extension of 
			\begin{align*} f \mapsto (\tau_1 \otimes ... \otimes \tau_n \mapsto PT^{-1}((\tau_k)_{(n_k)})...PT^{-1}((\tau_1)_{(n_1)}) f(\tau_1 \otimes ... \otimes \tau_k))\end{align*}
			where $\tau_i : \Delta^{n_i} \to \Omega^{M}_{x_0}(M)$ for $i = 1,...,k$ and $PT$ denotes the parallel transport of $F$.
		\end{proof}
	\end{lemma}
	Set $\text{Rep}(\Omega_{x_0}(M))^{\infty}[e^{-H}] := \text{Tw}(\text{Rep}(\Omega_{x_0}(M))^{pre}[e^{-H}])$.
	Using Proposition 4.17 in \cite{AbadSchaetz}, we obtain that $\text{Rep}^{0}(M)[0]$ is a quasi-fully faithful. Combing this with Lemma \ref{relateLemma2}, we see that $\text{Rep}^{0}(M)[h]$ is quasi-fully faithful for every $h$. We thus obtain: 
	\begin{theorem}
		\label{hTheorem2}
		The induced $A_\infty$-functor
		\begin{align*}\text{Rep}(M)^{\infty}[h] := \text{Tw}(\text{Rep}(M)^{pre}[h]) : \mathcal P(M)^\infty[h] \to \text{Rep}(\Omega_{x_0}(M))^\infty[e^{-H}] \end{align*}
		of dg-categories is a dg-quasi equivalence.
		\begin{proof}
			We have shown that $\text{Rep}(M)^{0}[h]$ is quasi-fully faithful. By Lemma \ref{qEss} we have that $\text{Rep}(M)^{0}[h]$ is quasi-essentially surjective. By Lemma \ref{isSpectral} and Lemma \ref{repSpectral}, we may apply Theorem \ref{EquivalenceTheorem} to obtain that $\text{Rep}(M)^{\infty}[h]$ is an $A_\infty$-quasi-equivalence. 
		\end{proof}
	\end{theorem}
	\subsection{Remarks on extensions and curvature}
	\begin{remark}
		Given a simplicial monoid and curvature element $H$ we defined the cdg-category of $H$-curved representations $\text{Rep}(S)^{pre}[H]$. There is an alternative description of this category. Consider the cdg-category $\text{Rep}(S^H)^{pre}[0]$. We consider the $\mathbb K^*$-equivariant subcategory $\text{Rep}(S^H)^{pre}_{\mathbb K^*-\text{eq}}$ of $\text{Rep}(S^H)^{pre}[0]$. Its objects are those pairs $(E, \mu)$ for which we have $\mu(x,e) = x \mu(1,e)$ and similarly we may define $\mathbb K^*$-equivariant morphisms (they are $\mathbb K^*$-equivariant in each tensor factor). It is easy to verify that $\text{Rep}(S)^{pre}[H] \cong \text{Rep}(S^H)^{pre}_{\mathbb K^*-\text{eq}}$ as cdg-categories (and thus the associated categories of twisted complexes are dg-equivalent).\\ \\
		For projectively flat vector bundles the situation is slightly more subtle. Curvature elements are parametrised by $H^2(M; \mathbb K)$. As $\mathbb K$ is a field we have \[H^2(M; \mathbb K) = \mathbb K \otimes_{\mathbb Z} H^2(M; \mathbb Z).\]
		We now restrict to $\mathbb K = \mathbb C$. Given a closed 2-form $h$ we see that there are line bundles $(L_1,\nabla_1)$,...,$(L_n,\nabla_n)$ with curvatures $h_i = (\nabla_i)^2$ such that $h$ is cohomologous to $\sum_i \lambda_i h_i$ for some $\lambda_i \in \mathbb C$. As cohomologous curvatures give rise to equivalent cdg-categories, we assume $h = \sum_i \lambda_i h_i$. The line bundles give rise to a principle $(\mathbb C^*)^n$-bundle over $M$ which we denote by $\pi : M^h \to M$ and the corresponding right action is denoted by $\rho : M^h \times (\mathbb C^*)^n \to M^h$. Each $h_i$ becomes trivial on $M^h$, i.e., we fix $\omega_i$ satisfying $d\omega^i = \pi^* h_i$. We consider the cdg-category $\mathcal P(M^h)^{pre}[0]$. We may use the $\omega_i$ and $\rho$ to construct an equivariant subcategory of $\mathcal P(M^h)^{pre}[0]$ which is equivalent to $\mathcal P(M)^{pre}[h]$. 
		
	\end{remark}
	\subsection{Projective representations of the loop space}
	We already defined the category of projective representations $\text{Rep}(\Omega_{x_0}(M))$. We need to define an analogue for vector bundles.
	
	\begin{def1}
		\label{defLineDG}
		Let $h, h' \in Z\Omega^2(M)$ and $(L, \nabla_L)$ be a line bundle with curvature $h'-h$. We define a functor 
		\begin{align*} PS^{pre}[L] : \mathcal P(M)^{pre}[h] \to \mathcal P(M)^{pre}[h']\end{align*} of cdg-categories.
		On objects, we set
		\begin{align*}PS^{pre}[L](E, \nabla_E) := (E \otimes L, \nabla_E \otimes \mathbbm 1 + \mathbbm 1 \otimes \nabla_L)\end{align*} and on morphisms, we set 
		\begin{align*}\Omega^k(M, \text{Hom}(E,F)) \ni f \mapsto f \otimes \mathbbm 1_L \in \Omega^k(M, \text{Hom}(E \otimes L,F \otimes L)).\end{align*}
		It is straightforward to verify that this indeed defines a functor of cdg-categories. Let $(L', \nabla')$ be another such line bundle, and let $\phi : (L, \nabla) \to (L', \nabla')$ be an isomorphism of line bundles compatible with the connections  (i.e., a non-zero closed morphism $L \to L'$ in the category $\mathcal P(M)^{0}[h'-h]$). Then $\phi$ induces a natural transformation
		\begin{align*} PS^{pre}[\phi] : PS^{pre}[L] \Rightarrow PS^{pre}[L']\end{align*}
		by $PS^{pre}[\phi](E) := \mathbbm 1_{E} \otimes \phi : E \otimes L \to E \otimes L'$
		with inverse $PS^{pre}[\phi^{-1}]$.
		For $E^*$ an object of $\mathcal P(M)[h]$ and $F^*$ an object of $\mathcal P(M)^\infty[h']$, we set 
		\begin{align*} \text{Hom}^{h,h'}_L(E^*, F^*) := H^0(\text{Hom}_{\mathcal P(M)^\infty[h']}(PS[L](E^*), F^*)).\end{align*}
		We define the composition
		\begin{align*} \text{Hom}^{h,h''}_{L'}(F^*, G^*) \times \text{Hom}^{h,h'}_L(E^*, F^*) \to  \text{Hom}^{h,h''}_{L \otimes L'}(E^*, G^*)\end{align*}
		by identifying 
		\begin{align*}\text{Hom}^{h,h'}_L(E^*, F^*) = \text{Hom}^{h-h'+h'',h''}_L(PS[L']E^*, PS[L']F^*).\end{align*}
		Using that $PS[L] \circ PS[L'] = PS[L \otimes L']$ one sees that this composition is associative.
		Next, consider the disjoint union 
		\begin{align*}  \bigsqcup_L \text{Hom}^{h,h'}_L(E^*, F^*)\end{align*}
		indexed by all line bundles $(L, \nabla)$ with curvature $h'-h$. We define the following equivalence relation $\sim$ on this set. $f \in \text{Hom}^{h,h'}_L(E^*, F^*)$ and $g \in \text{Hom}^{h,h'}_{L'}(E^*, F^*)$ are identified if there exists an isomorphism of line bundles $\phi : (L, \nabla) \to (L', \nabla')$ 
		such that \begin{align*} g \circ H^0(PS[\phi](E^*)) = f.\end{align*}
		In particular, note that for every $f \in \text{Hom}^{h,h'}_L(E^*, F^*)$ we have $\lambda f \sim f$ for every $\lambda \in \mathbb K^*$.
	\end{def1}
	\begin{def1}
		We define the category $\text{PF}_{\infty}(M)$ of projectively flat graded bundles. The class of objects of $\text{PF}_{\infty}(M)$ is given by the disjoint union of the non-zero objects of $\mathcal P(M)[h]$ for all $h \in Z\Omega^2(M)$ . Let $E^* $ be an object of $\mathcal P(M)[h]$ and $F^*$ an object of $\mathcal P(M)[h']$. Then the morphism set between $E^*$ and $F^*$ is given by
		\begin{align*}  \text{Hom}_{\text{PF}_{\infty}(M)}(E^*, F^*) := \quot{\bigsqcup_L \text{Hom}^{h,h'}_L(E^*, F^*)}{\sim}.\end{align*}
		It is straightforward to verify that the above composition induces a composition in this category that is associative.
	\end{def1}
	\begin{lemma}
		Let $h, h' \in Z\Omega^2(M)$ and $(L, \nabla)$ a line bundle with curvature $h'-h$. Denote the parallel transport of $(L, \nabla)$ by $PT$. Note that we can canonically identify $\text{Aut}(L_{x_0}) = \mathbb K^*$. We define a twist $a$ from $e^{-H}$ to $e^{-H'}$ by setting
		\begin{align*}a(\gamma) := PT(\gamma) \in \mathbb K^*\end{align*}
		for $\gamma \in \Omega_{x_0}(M)$.  
		Then any isomorphism $\phi : L_{x_0} \to \mathbb K$ induces a natural isomorphism of functors
		\begin{align*} \eta_\phi : \text{PRep}(M)[h'] \circ PS^{pre}[L] \Rightarrow  \text{PRep}^{pre}(\Omega_{x_0}(M)[a])\circ \text{PRep}^{pre}(M)[h]\end{align*}
		between the categories $ \mathcal P^{pre}(M)[h]$ and $\text{PRep}^{pre}(\Omega_{x_0}(M))[e^{-H'}]$. 
		\begin{proof}
			As $L$ has curvature $h'-h$, we obtain that $a$ is indeed a twist from $e^{-H}$ to $e^{-H'}$. Let $(E, \nabla)$ be an object of $P^{pre}_\infty(M)[h]$. On the left side, we obtain the object $(E_{x_0} \otimes L_{x_0}, PT_E \cdot PT_L)$ and on the right side, we obtain $(E_{x_0}, PT_E \cdot a)$. $\mathbbm 1_{E} \otimes \phi$ now provides an isomorphism between them. It is clear that this assembles to give a natural isomorphism.
		\end{proof}
	\end{lemma}
	\begin{def1}
		We define a functor 
		\begin{align*}\text{PRep}(M) : \text{PF}_{\infty}(M) \to \text{LPRep}(\Omega_{x_0}(M)).\end{align*}
		On objects, it is induced by the functors $\text{PRep}(M)[h]$ for $h \in Z\Omega^2(M)$. Let $E^* $ be an object of $\mathcal P(M)[h]$ and $F^*$ an object of $\mathcal P(M)[h']$. Furthermore, let 
		\begin{align*} [f] \in \text{Hom}_{\text{PF}_{\infty}(M)}(E^*, F^*) \end{align*}
		represented by some $f \in \text{PHom}^{h,h'}_L(E^*, F^*)$. Applying $H^0(\mathcal P(M)[h'])$ to \begin{align*}f \in H^0(\text{Hom}_{\mathcal P(M)[h']}(PS[L](E^*), F^*))\end{align*} yields an element of 
		\begin{align*} H^0(\text{Hom}_{\text{PRep}^{pre}(\Omega_{x_0}(M))[e^{-H'}]}(\text{PRep}(M)[h']PS[L](E^*), \text{PRep}(M)[h']F^*)) \end{align*}
		which can be identified with 
		\begin{align*}H^0(\text{Hom}_{\text{PRep}^{pre}(\Omega_{x_0}(M))[e^{-H'}]}(\text{PRep}^{pre}(\Omega_{x_0}(M)[a]) \text{PRep}^{pre}(M)[h]E^*,\text{PRep}(M)[h']F^*))\end{align*}
		by the previous Lemma via any isomorphism $\phi : L_{x_0} \to \mathbb K$. This provides an element of 
		\begin{align*} \text{PHom}^{H, H'}_a(\text{PRep}^{pre}(M)[h]E,\text{PRep}^{pre}(M)[h']F) \end{align*} which we now project to 
		\begin{align*} \text{Hom}_{\text{LPRep}(\Omega_{x_0}(M))}(\text{PRep}^{pre}(M)[h]E,\text{PRep}(M)[h']F).\end{align*} As we have identified $\mathbb K^*$ to $1$ this element does not depend on the choice of $\phi$. It is easy to verify that this construction also does not depend on the representative $f$, as isomorphic line bundles give rise to the same twist $a$. One can also check that it is compatible with composition. 
	\end{def1}
	We can now state and prove our main result.
	\begin{theorem}
		\label{PFPrep}
		\begin{align*}\text{PRep}(M) : \text{PF}_{\infty}(M) \to \text{LPRep}(\Omega_{x_0}(M))\end{align*}
		is an equivalence of categories.
		\begin{proof}
			\begin{enumerate}
				\item $\text{PRep}(M)$ is essentially surjective: Let $E^*$ be an object of $\text{Rep}(\Omega_{x_0}(M))[e^{-H}]$ for some logarithmic curvature element $H$. If $H$ is of the form \begin{align*}H(\tau) := H_h(\tau) := - \int_I \tau^* \alpha^* \text{hol}_h\end{align*} for some $h \in Z\Omega^2(M)$, the claim follows from Theorem \ref{hTheorem2}. Therefore, it suffices to show that each object $E^*$ with curvature $e^{-H}$ is isomorphic to an object with curvature $e^{-H_h}$ for some $h \in Z\Omega^2(M)$. By definition of the morphism sets, we have to show that $e^{-H}$ is cohomologous to some $e^{-H_h}$. The quasi-fully faithfulness of $\text{PRep}^0(M)[h = 0]$ implies that $H = H_h + \nabla a$ for some $h \in Z\Omega^2(M)$ and $a \in \text{Hom}^1(B(\Omega_{x_0}(M),1), \mathbb K)$. Thus, $H$ and $H_h$ are cohomologous and therefore so are $e^{-H}$ and $e^{-H_h}$.
				\item Fullness of $\text{PRep}(M)$ follows from the fullness $H^0(\text{PRep}(M)[h])$, combined with the following observation: Let $a$ be any twist from $e^{-H_{h}}$ to $e^{-H_{h'}}$. Then $(\mathbb K, a)$ defines an object of $\text{PRep}^0(\Omega_{x_0}(M))[e^{-H_{h'-h}}]$ and it is isomorphic to an object
				\begin{align*} \text{Rep}^0(M)[h'-h](L, \nabla) \end{align*}
				where $(L, \nabla)$ is some line bundle since $\text{Rep}^0(M)[h'-h]$ is essentially surjective.
				\item Faithfulness of $\text{Rep}(M)$ follows similarly: $H^0(\text{Rep}(M)[h'-h])$ being an equivalence implies that two line bundles with curvature $h'-h$ are isomorphic if and only if they give rise to the same twist. Furthermore, the automorphisms of any line bundle with connection are precisely the non-zero constant maps (as $M$ is connected).
			\end{enumerate}
		\end{proof}
	\end{theorem}
	\subsection{Relation with Koszul duality}
	In the uncurved case, there is another way of proving Theorem \ref{hTheorem2} via \emph{Koszul duality}. Namely, it is known by \cite{Rivera}, that there exists a natural zig-zag of quasi-isomorphisms between the dg-algebras $\text{Cobar}(C_*(M))$ and $C_*(\Omega_{x_0}(M))$ where $\text{Cobar}(-)$ denotes the cobar construction of a dg-coalgebra. In particular, this yields a Morita equivalence  $D(\text{Cobar}(C_*(M))) \cong D(C_*(\Omega_{x_0}(M)))$. Now by Koszul duality \cite[Theorem 6.4]{Posi} we futhermore have an equivalence
	$D(\text{Cobar}(C_*(M))) \cong D^{co}(C_*(M))$. The category $\text{Loc}^\infty(M)$ of $\infty$-local systems is a full subcategory of $D^{co}(C_*(M))$ and under the above equivalence, $\infty$-local systems correspond to cohomologically finite-dimensional dg-modules over $\text{Cobar}(C_*(M))$ and hence also to cohomologically finite-dimensional dg-modules over $C_*(\Omega_{x_0}(M))$. The category of cohomologically finite-dimensional dg-modules has a natural dg-enhancement given by finite-dimensional $A_\infty$-modules and this dg-category is equivalent to $\text{Rep}^\infty(\Omega_{x_0}(M))$ (by definition of $\text{Rep}^\infty(\Omega_{x_0}(M))$). Combining these observations with Theorem \ref{hCurvedDGEquivalence} yields another way of proving \ref{hTheorem2} in the uncurved case. Of course, it is expected that the same reasoning works in the curved case but one would need to prove the quasi-isomorphism between $\text{Cobar}(C^H_*(M))$ and $C^H_*(\Omega_{x_0}(M))$ (the curved dg-coalgebra $C^H_*(M)$ may be defined as in the formulas appearing in Definition \ref{DefInfintyLocal}). In any case, it is still easy to see that:
	\begin{proposition}
		The category of $H$-curved $\infty$-local systems $\text{Loc}^\infty(M)[H]$ is a full subcategory of $D^{co}(C^H_*(M))$ and under Koszul duality
		\[D^{co}(C^H_*(M)) \cong D(\text{Cobar}(C^H_*(M))) \]
		$H$-curved $\infty$-local system correspond to cohomologically finite-dimensional dg-modules over $\text{Cobar}(C^H_*(M))$.
	\end{proposition}

	\section{1-categorical results}
	In this chapter we specialize the obtained results to ungraded objects. In particular, we want to understand the relation between projectively flat vector bundles and projective representations of the fundamental group of $M$. We start with a more general construction for a simplicial monoid $S$ as before. 
	\subsection{Representations of simplicial monoids and of monoids}
	\begin{def1}
		Let $\quot{S_0}{\sim}$ be the set of path components of $S$ which is again a monoid. In the case of interest, this recovers the fundamental group of $M$. Essentially, we want to relate representations of $\quot{S_0}{\sim}$ with representations of $S$. We could in principle now define categories of curved representations of arbitrary monoids from scratch, but it in our framework it is formally more straightforward to consider $\quot{S_0}{\sim}$ itself as a simplicial monoid (set $(\quot{S_0}{\sim})_n := \quot{S_0}{\sim}$ for every $n$ and all face and degeneracy maps are the identity). Then there is a map of simplicial monoids
		\begin{nalign}
			\pi_0 : S_* \to (\quot{S_0}{\sim})_*, \sigma \mapsto 
			\begin{cases}
				[\sigma] \text{ if } \sigma \in S_0\\
				[\sigma_{(0)}] \text{ else}
			\end{cases}
		\end{nalign}
		This works because $\sigma_{(0)}$ is homotopic to $\sigma_{(i)}$ for all $i$, witnessed by $\sigma_{(0,i)}$.
		All the constructions in the previous chapter are functorial. Given a curvature element $H$ of $\quot{S_0}{\sim}$ (which is just a 2-cocycle in the group complex of  $\quot{S_0}{\sim}$) we obtain a cdg-functor
		\[ \pi_0(S)^{pre}[H] : \text{Rep}^{pre}(\quot{S_0}{\sim})[H] \to \text{Rep}^{pre}(S)[\pi_0^*H].\]
	\end{def1}
	\begin{proposition}
		\label{propRep0}
		The functor
		\[ H^0(\pi_0(S)^0[H]) :H^0(\text{Rep}^{0}(\quot{S_0}{\sim})[H]) \to H^0(\text{Rep}^{0}(S)[\pi_0^*H])\]
		is an equivalence.
		\begin{proof}
			Let $(E, \mu_E)$ be an object of $\text{Rep}^{0}(S)[\pi_0^*H]$. We define an object $(E, \mu_E')$ of  $\text{Rep}^{0}(\quot{S_0}{\sim})[H]$. For $\gamma \in S$ we set
			\begin{nalign}
				\mu_E'([\gamma]) = \mu_E(\gamma).
			\end{nalign}
			This is well-defined because the curvature element $\pi_0^*H$ is trivial on $S_1$. It is straightforward to see that this indeed defines an object of $\text{Rep}^{0}(\quot{S_0}{\sim})[H]$ which maps to $(E, \mu_E)$.
			Moreover one easily checks that the functor is fully faithful.
		\end{proof}
	\end{proposition}
	We have shown that $H^0(\pi_0(S)^0[H])$ is an equivalence. Note that we do not claimJ. Block and C. Daenzer, Mukai duality for gerbes with connection
	that $\pi_0(S)^0[H]$ is a dg-quasi equivalence. It is essentially surjective but it does not induce isomorphisms on higher cohomology groups.
	\begin{corollary}
		Let $H$ be a 2-cochain in the group complex of $\pi_1(M)$. Then there is an equivalence of categories
		\[ H^0(\text{Rep}^{0}(\pi_1(M))[H]) \to H^0(\text{Rep}^{0}(\Omega_{x_0} M)[\pi_0^*H])\]
		where the first category can be identified with the category of $H$-projective representations of $\pi_1(M)$.
	\end{corollary}
	We would further like to relate the category $H^0(\text{Rep}^{0}(\pi_1(M))[H])$ with the category of projectively flat vector bundles for some fixed curvature $h$. For this to work, we need to find a curvature element $H$ of $\pi_1(M)$ such that $\pi_0^*H$ and $e^{-H_h}$ are cohomologous as curvature elements of $\text{Sing}(\Omega_{x_0}(M))$. Moreover, we need to understand which kind of curvature elements of $\pi_1(M)$ correspond to logarithmic curvature elements of $\text{Sing}(\Omega_{x_0}(M))$. From now on let $\mathbb K = \mathbb C$ and we again consider an arbitrary simplicial monoid $S$.
	\begin{proposition}
		\label{propLongExact}
		Let $B_B^*(\quot{S_0}{\sim}, \mathbb C)$ be the sub-cochain complex of $B^*(\quot{S_0}{\sim}, C)$ consisting of those cochains that map into the image of $B^*(\quot{S_0}{\sim}, \mathbb C^*) \to B^*(S, \mathbb C^*)$ under $B^*(\exp) : B^*(\quot{S_0}{\sim}, \mathbb C) \to B^*(\quot{S_0}{\sim}, \mathbb C^*)$. Then there is a short exact sequence of cochain-complexes
		\[ 0 \to B^*(S, \mathbb Z) \to B_B^*(S, \mathbb C) \to B^*(\quot{S_0}{\sim}, \mathbb C^*) \to 0.\]
		Consequently, there is a long exact sequence
		\[ ... \to H^i(S, \mathbb Z) \to H_B^i(S, \mathbb C) \to H^i(\quot{S_0}{\sim}, \mathbb C^*) \to H^{i+1}(S, \mathbb Z) \to ... \]
		\begin{proof}
			The map of simplicial monoids
			\[ S \to \quot{S_0}{\sim}\]
			is levelwise surjective. Consider the commutative diagram
			\[
			\begin{tikzcd}
				& 0 \arrow[d]                                           & 0 \arrow[d]                                           & 0 \arrow[d]                                           &   \\
				0 \arrow[r] & {B^*(\quot{S_0}{\sim}, \mathbb Z)} \arrow[d] \arrow[r] & {B^*(\quot{S_0}{\sim}, \mathbb C)} \arrow[d] \arrow[r] & {B^*(\quot{S_0}{\sim}, \mathbb C^*)} \arrow[d] \arrow[r] & 0 \\
				0 \arrow[r] & {B^*(S, \mathbb Z)} \arrow[r]                        & {B^*(S, \mathbb C)} \arrow[r]                        & {B^*(S, \mathbb C^*)} \arrow[r]                        & 0
			\end{tikzcd}.\]
			All rows and columns are exact. A simple diagram chase shows the existence of a short exact sequence
			\[ 0 \to B^*(S, \mathbb Z) \to B_B^*(S, \mathbb C) \to B^*(\quot{S_0}{\sim}, \mathbb C^*) \to 0.\]
		\end{proof}
	\end{proposition}
	\begin{remark}
		The method of proof in the above theorem can be applied to any short exact sequence of abelian groups $0 \to M \to N \to L \to 0$ and a surjective map $P \to Q$ of simplicial monoids. 
	\end{remark}
	\begin{remark}
		The previous Proposition also shows the following: Let $H$ be a logarithmic curvature element such that $[H] \notin H_B^2(S, \mathbb C)$, then the category $\text{Rep}(S)^{\infty}[e^{-H}]$ is dg-quasi equivalent to $0$ by Lemma \ref{trivialDiff} and by inspection of the proof of Proposition \ref{PropPiSame}.
	\end{remark}
	\begin{def1}
		Denote by $H^2_B(\quot{S_0}{\sim}, \mathbb C^*)$ the image of the above defined map $H_B^2(S, \mathbb C) \to H^2(\quot{S_0}{\sim}, \mathbb C^*)$ (note that this does depend on $S$, not just on $\quot{S_0}{\sim}$). Let $\text{LPRep}^0(S)$ be the full subcategory of $\text{LPRep}(S)$ consisting of objects concentrated in degree $0$ and let $\text{PRep}^0_{H^2_B(\quot{S_0}{\sim}, \mathbb C^*)}(\quot{S_0}{\sim})$ be the full subcategory of $PRep(\quot{S_0}{\sim})$ consisting of objects concentrated in degree $0$ and whose curvatures belong to $H^2_B(\quot{S_0}{\sim}, \mathbb C^*)$ (this is equivalently just the category of projective representations of the monoid $\quot{S_0}{\sim}$ whose curvatures belong to $H^2_B(\quot{S_0}{\sim}, \mathbb C^*)$).
	\end{def1}
	We obtain:
	\begin{proposition}
		\label{PropPiSame}
		There is an equivalence of categories
		\[\text{PRep}^0_{H^2_B(\quot{S_0}{\sim}, \mathbb C^*)}(\quot{S_0}{\sim}) \to \text{LPRep}^0(S) \]
		\begin{proof}
			The statement mainly follows from Proposition \ref{propRep0} which implies that the notion of twists is equivalent on both sides (the arguments are similar to those used in \ref{PFPrep}), we thus deduce fully faithfulness. To show essential surjectivity we make the following, slightly non-trivial observation: Let $H$ be the curvature of a non-zero object $E$ in $\text{LPRep}^0(S)$.  Inspecting equation  \ref{EqCurvature}, one can construct a curvature element $H'$ of $\quot{S_0}{\sim}$ such that it is mapped to a curvature element of $S$ that is cohomologous to $H$ (this is not true for arbitrary curvature elements). By the long exact sequence of the previous proposition and since $H$ is logarithmic, this implies that $[H] \in H^2_B(S, \mathbb C)$. This means that $E$ is isomorphic to an object with curvature contained in $H^2_B(S, \mathbb C)$ which proves the claim using Proposition \ref{propRep0}. 
		\end{proof}
	\end{proposition}
	We apply this result again to $S = \text{Sing}(\Omega_{x_0}(M))$ and obtain:
	\begin{theorem}
		\label{TheoremUngraded}
		The category of projectively flat vector bundles $\text{PF}_{0}(M)$ is equivalent to the category $\text{PRep}^0_{H^2_B(\pi_1(M), \mathbb C^*)}(\pi_1(M))$ of projective representations of $\pi_1(M)$ whose curvatures belong to $H^2_B(\pi_1(M), \mathbb C^*)$. These are precisely the classes in $H^2(\pi_1(M), \mathbb C^*)$ that get mapped to $0$ under the composition $H^2(\pi_1(M), \mathbb C^*) \to H^2(M, \mathbb C^*) \to H^3(M, \mathbb Z)$ where the first map is induced by the natural map $M \to K(\pi_1(M),1)$ to the Eilenberg-MacLane space of $\pi_1(M)$ and the latter is the connecting morphism induced by the short exact sequence $0 \to \mathbb Z \xrightarrow{2 \pi i} \mathbb C \xrightarrow{\exp} \mathbb C^* \to 0$.
		\begin{proof}
			By the previous proposition, we have that $\text{PRep}^0_{H^2_B(\pi_1(M), \mathbb C^*)}(\pi_1(M))$ is equivalent to $\text{LPRep}^0(\Omega_{x_0}(M))$. Specializing Theorem \ref{PFPrep} to objects concentrated in degree $0$ we have that $\text{PF}_{0}(M)$ is equivalent to $\text{LPRep}^0(\Omega_{x_0}(M))$.
		\end{proof}
	\end{theorem}
	\begin{remark}
		Some of the constructions to obtain the categories $\text{PF}_{0}(M)$ and\\ $\text{PRep}^0_{H^2_B(\pi_1(M), \mathbb C^*)}(\pi_1(M))$ were rather technical. Let us spell out equivalent definitions of these categories. 
		\begin{enumerate}
			\item The objects of the category $\text{PF}_{0}(M)$ are pairs $(E, \nabla)$ consisting of a finite dimensional vector bundle $E$ over $M$ together with a $\mathbb C$-linear connection $\nabla$ which is projectively flat. A morphism $f : (E, \nabla) \to (F, \nabla')$ is represented by a non-zero map $f : E \otimes L \to F$ of vector bundles where $(L, \nabla'')$ is a line bundle together with a connection and $f$ is compatible with the connections. A representative $f : E \otimes L \to F$ is identified with a representative $f ': E \otimes L' \to F$ if there exists an isomorphism of line bundles $l : L \to L'$, such that $f \circ (\mathbbm 1 \otimes l) = f$.
			\item An object of $\text{PRep}^0_{H^2_B(\pi_1(M), \mathbb C^*)}(\pi_1(M))$ is a map $\mu : \pi_1(M) \to GL_n$ such that $\pi \circ \mu:  \pi_1(M) \to GL_n \to PGL_n$ is a group homomorphism and the corresponding element of $H^2(\pi_1(M), \mathbb C^*)$ is contained in the subgroup $H_B^2(\pi_1(M), \mathbb C^*)$. A morphism $f : \mu \to \mu'$ is represented by a non-zero linear map $f : \mathbb C^n \to \mathbb C^m$ such that $[f \circ \mu(\gamma)] =  [\mu'(\gamma) \circ f]$ (in $\quot{\text{Hom}(\mathbb C^n , \mathbb C^m)}{\mathbb C^*}$) for all $\gamma \in \pi_1(M)$. Two such representatives $f$ and $g$ are identified if there exists a $\lambda \in \mathbb C^*$ such that $f = \lambda g$.
		\end{enumerate}
	\end{remark}
	\subsection{Examples}
	\begin{example}
		\label{TorusExample}
		We consider the torus $\mathbb T^2$ and want to construct some projectively flat vector bundles associated to projective representations of the fundamental group. We view the torus as the quotient of $\mathbb R^2$ by identifying $(x+p,y+q)$ with $(x,y)$ for all $p,q \in \mathbb Z^2$. The associated quotient map $\mathbb R^2 \to \mathbb T^2$ is the universal cover of $\mathbb T^2$. We denote by $[x,y]$ points of $\mathbb T^2$ under this covering map. The fundamental group of $\mathbb T^2$ is $\pi_1(\mathbb T^2, [0,0]) \cong \mathbb Z \times \mathbb Z$ where the generators are represented by the paths $a(t) = [t,0]$ and $b(t) = [0,t]$. We fix two elements $A,B \in GL_n(\mathbb C)$ such that $AB = -BA$ and consider the projective representation 
		\begin{align*} \mathbb Z \times \mathbb Z \ni (p,q) \mapsto [A^p B^q] \in PGL_n(\mathbb C).\end{align*} In particular, this does not lift to a linear representation. We now construct a vector bundle associated to this representation as follows. First, we consider the trivial bundle $\mathbb R^2 \times \mathbb C^n$ on the universal covering $\mathbb R^2$ of $\mathbb T^2$. We will define the desired bundle as a quotient of this trivial bundle. The group $\mathbb Z \times \mathbb Z$ acts on $\mathbb R^2 \times \mathbb C^n$ by 
		\begin{align*}(p,q).(x,y,v) = (x+p, y+q, e^{i \pi qx} A^pB^q v).\end{align*} One verifies that this indeed defines a group action using the commutation relations of $A$ and $B$.
		Let us call the quotient space under this group action $E$. There is a natural covering map $E \to \mathbb T^2$ which exhibits $E$ as a complex $n$-dimensional vector bundle over $\mathbb T^2$.. Next, we define a connection on the trivial bundle $\mathbb R^2 \times \mathbb C^n$ by setting $\nabla = d - i \pi y dx$. We want to argue that this connection descends to a connection on $E$. We therefore have to show that for each $(p,q) \in \mathbb Z^2$ and section $s$ of $\mathbb R^2 \times \mathbb C^n$ we have the equality $(p,q). \nabla|_{(x,y)}  e = \nabla|_{(x+p,y+q)} (p,q).e$. So we check: 
		\begin{nalign}
			\nabla|_{(x+p,y+q)} (p,q).e &= \nabla|_{(x+p,y+q)} e^{i\pi qx} A^p B^q e\\
			&= i\pi q e^{i\pi q x} A^pB^q dx \otimes e + e^{i\pi q x} A^pB^q (de - i\pi (y+q) dx \otimes e)\\
			&= e^{i\pi q x} A^pB^q (de - i\pi y dx \otimes e)\\
			&= (p,q). \nabla|_{(x,y)}  e. 
		\end{nalign}
		The connection is projectively flat with curvature $h = i \pi dx \wedge dy$ which is a generator of $H^2_{\text{dR}}(\mathbb T^2, \mathbb C) \cong \mathbb C$ (as a $\mathbb C$-vector space). The parallel transports along the paths $a$ and $b$ are precisely equal to $A$ and $B$ so this parallel transport indeed induces the projective representation from above.	
	\end{example}
	
	\begin{example}
		\label{counterexampleProj}
		Here, we will provide a counterexample to the statement that every projective representation of the fundamental group of a manifold arises from the parallel transport of a projectively flat vector bundle. We consider the manifold $M =   \mathbb RP^2 \times \mathbb S^1$. Here is a list of some relevant homology and cohomology groups:
		
		\begin{center}
			\begin{tabular}{ |c|c|c|c|c| } 
				\hline
				i & $H_i(M; \mathbb Z)$ & $H^i(M; \mathbb Z)$ & $H^i(M; \mathbb C)$ & $H^i(M; \mathbb C^*)$ \\
				\hline
				0 & $\mathbb Z$ & $\mathbb Z$ & $\mathbb C$ & $\mathbb C^*$\\ 
				1 & $\quot{\mathbb Z}{2\mathbb Z} \times \mathbb Z$  & $\mathbb Z$ & $\mathbb C$ & $\quot{\mathbb Z}{2\mathbb Z} \times \mathbb C^*$\\ 
				2 & $\quot{\mathbb Z}{2\mathbb Z} \times \mathbb Z$ & $\quot{\mathbb Z}{2\mathbb Z} \times \mathbb Z$ &  $\mathbb C$ & $\quot{\mathbb Z}{2\mathbb Z} \times \mathbb C^*$\\
				3 & $\mathbb Z$ & $\quot{\mathbb Z}{2\mathbb Z} \times \mathbb Z$ & $\mathbb C$ & $\mathbb C^*$\\ 
				\hline
			\end{tabular}
		\end{center}
		The fundamental group is $\pi_1(M; x_0) = \quot{\mathbb Z}{2\mathbb Z} \times \mathbb Z$. We also want to determine $H^2(\pi_1(M); \mathbb C^*)$. The group homology of $\quot{\mathbb Z}{2\mathbb Z}$ is
		\begin{nalign}
			H_i(\quot{\mathbb Z}{2\mathbb Z}; \mathbb Z) = \begin{cases}
				\mathbb Z, i = 0\\
				\quot{\mathbb Z}{2\mathbb Z}, i = 1,3,...\\
				0, i = 2,4,...
			\end{cases}
		\end{nalign}
		whose calculation can be found in \cite{Weibel}.  The group homology of $\mathbb Z$ is 
		\begin{nalign}
			H_i(\mathbb Z; \mathbb Z) = \begin{cases}
				\mathbb Z, i = 0,1\\
				0, \text{ else }
			\end{cases}
		\end{nalign}
		which can also be found in \cite{Weibel}. There is a Künneth theorem for group homology from which we deduce
		\begin{align*} H_1( \quot{\mathbb Z}{2\mathbb Z} \times \mathbb Z; \mathbb Z) =  \quot{\mathbb Z}{2\mathbb Z}, H_2( \quot{\mathbb Z}{2\mathbb Z} \times \mathbb Z; \mathbb Z) =  \mathbb Z \oplus \quot{\mathbb Z}{2\mathbb Z}.\end{align*}
		By the Universal Coefficient Theorem, we then obtain
		\[ H^2(\quot{\mathbb Z}{2\mathbb Z} \times \mathbb Z; \mathbb C^*) = \quot{\mathbb Z}{2\mathbb Z}. \]
		\begin{align*} . \end{align*}
		We define the matrices 
		$A = \begin{pmatrix}
			1 & 0\\
			0 & -1
		\end{pmatrix}$
		and $B = \begin{pmatrix}
			0 & 1\\
			1 & 0
		\end{pmatrix}$.
		The relevant properties of these matrices are $A^2 = \mathbbm 1$ and $AB = -BA$. We define the projective representation $f : \pi_1(M; x_0) \to GL_2$ by $f(\bar a,b) = A^a B^b$. By the above relations, this indeed defines a projective representation. We determine the corresponding element $[df]$ in $H^2(\pi_1(M; x_0); \mathbb C^*)$:
		\begin{nalign} 
			df((\bar a, b), (\bar a', b')) &= f(\bar a, b)f(\bar a', b')(f(\bar {a}+\bar{a'}, b+b'))^{-1}\\
			&= A^aB^bA^{a'}B^{b'} B^{-b-b'} A^{-a-a'} = (-1)^{a' b},
		\end{nalign}
		which represents the generator of $H^2(\quot{\mathbb Z}{2\mathbb Z} \times \mathbb Z; \mathbb C^*) = \quot{\mathbb Z}{2\mathbb Z}$. The map $H^2(M; \mathbb C) \to H^2(M; \mathbb C^*)$ is given by $(0, \exp) : \mathbb C \to \quot{\mathbb Z}{2\mathbb Z} \times \mathbb C^*$ and $H^2(\pi_1(M), \mathbb C^*) \to H^2(M; \mathbb C)$ is $(\text{id},1) :  \quot{\mathbb Z}{2\mathbb Z} \to  \quot{\mathbb Z}{2\mathbb Z} \times \mathbb C^*$. Combining these observations together with Proposition \ref{propLongExact}, we see that $[df]$ is not in the image of $H^2_B(M, \mathbb C) \to H^2(\pi_1(M), \mathbb C^*)$, so $f$ does not correspond to a projectively flat vector bundle.
		
	\end{example}
	
	\begin{example}
		We will consider the 2-sphere $\mathbb S^2$. $\mathbb S^2$ is simply connected and by the exact sequence in Proposition \ref{propLongExact} we obtain $\mathbb Z = H^2(\mathbb S^2; \mathbb Z) = H_B^2(\mathbb S^2; \mathbb C)$. Since line bundles induce dg-equivalences (see Definition \ref{defLineDG}), we have for any $h \in Z\Omega_B^2(\mathbb S^2; \mathbb C)$
		\begin{align*} \mathcal P^{\infty}(M)[h] \cong \mathcal P^{\infty}(M)[0].\end{align*}
		For $h \notin Z\Omega_B^2(\mathbb S^2; \mathbb C)$, Lemma \ref{trivialDiff} implies that $\mathcal P^{\infty}(\mathbb S^2)[h]$ is dg-quasi-equivalent to the trivial dg-category $0$. Therefore, it remains to compute $\mathcal P^{\infty}(\mathbb S^2)[0]$.\\ \\ $\mathbb S^2$ is formal in the sense that there is a quasi-isomorphism of dg-coalgebras $H_* \to C_*(\mathbb S^2)$ where $H_*$ is the cohomology of $\mathbb S^2$. We dualize this map and thus obtain a quasi-isomorphism $Q: C^*(\mathbb S^2; \mathbb C) \to H^*$ of dg-algebras.\\ \\
		We define the cdg-category $S$. We shall show that its associated dg-category $\text{Tw}(S)$ is dg-quasi-equivalent to $\text{Loc}(M)^{\infty}[0]$. The objects of $S$ are finite dimensional vector spaces over $\mathbb C$. For two such objects $E$ and $F$, we set $\text{Hom}_{S}^*(E,F) := H^* \otimes_{\mathbb C} \text{Hom}_{\mathbb C}(E,F)$. The differentials are induced by the differential on $H^*$, i.e., they are zero (the category is therefore even a dg-category; nevertheless, we will regard it as a cdg-category). The composition is induced by the product on $H^*$ and composition of linear maps. We define a dg-functor $T : \text{Loc}(\mathbb S^2)^{0}[0] \to S$ (we cannot define $T$ on $\text{Loc}(\mathbb S^2)^{pre}[0]$ because of the curvatures). We set $T(\{E_x\}_{x \in M}, P) := E_{x_0}$. As all flat vector bundles over $\mathbb S^2$ are trivial, we can naturally identify $\text{Hom}_{\text{Loc}(\mathbb S^2)^{0}[0]}(E,F) = C^*(\mathbb S^2; \mathbb C) \otimes \text{Hom}_{\mathbb K}(E_{x_0}, F_{x_0})$ for any two objects $E$ and $F$ of $\text{Loc}(\mathbb S^2)^{0}[0]$. Then we can simply set $T(f) = Q \otimes \mathbbm 1$ on morphisms and we see that this indeed defines a	
		dg-functor. Using the methods of section \ref{quasiEqSection}, we see that $T$ induces an equivalence 
		\begin{align*} \text{Tw}(T) : \text{Tw}(\text{Loc}(\mathbb S^2)^{0}[0]) \to \text{Tw}(S).\end{align*}
		Furthermore, the embedding $\text{Tw}(\text{Loc}(\mathbb S^2)^{0}[0]) \to \text{Loc}(\mathbb S^2)^{\infty}[0]$ is a dg-quasi equivalence (this is true for any sufficiently twist-closed, split cdg-category).
		Therefore we obtain:
		\begin{nalign}
			\mathcal P(\mathbb S^2)[h] \simeq
			\begin{cases}
				\text{Tw}(S) \text{ if } h \in Z\Omega^2_B(\mathbb S^2)\\
				0 \text{ else. }
			\end{cases}
		\end{nalign}
	\end{example}

	\subsection{Real projectively flat vector bundles}
	We conclude by a proposition showing that projectively flat vector bundles over the real numbers can always be made into flat vector bundles. 
	\begin{proposition}
		\label{realBoring}
		Let $\mathbb K = \mathbb R$. Let $(E, \nabla)$ be a non-zero projectively flat vector bundle. Then there is a 1-form $\omega \in \Omega^1(M)$ such that
		\begin{align*} \nabla' := \nabla - \omega\end{align*}
		is a flat connection on $E$. In particular, the curvature $h$ of $\nabla$ is equal to $d \omega$, i.e., $h$ is an exact 2-form.
		\begin{proof}
			We choose a good open cover $(U_\alpha)$ of $M$ and trivializations $\phi_\alpha$ such that $\nabla$ locally takes the form
			\begin{align*} \nabla = d_{\alpha} + \omega_{\alpha}\end{align*}
			where $\omega_{\alpha} \in \Omega^1(U_{\alpha})$ and $d_\alpha = \phi_\alpha \circ d \circ (\phi_{\alpha})^{-1}$ (this can always be done, see for example the proof of \ref{lemmaQes}). Then on an intersection $U_{\alpha} \cap U_{\beta}$, we have $d_\alpha + \omega_\alpha = d_\beta + \omega_\beta$. Writing $T_{\alpha \beta} = \phi_{\alpha}^{-1} \circ \phi_{\beta}$, we obtain
			\begin{nalign}
				\label{TdiffEq}
				&d_\alpha + \omega_\alpha = d_\beta + \omega_\beta \\
				&\Rightarrow d - T_{\alpha \beta} \circ d \circ T_{\alpha \beta}^{-1} = \omega_{\beta} - \omega_{\alpha} \\
				&\Rightarrow d - (T_{\alpha \beta} \circ d(T_{\alpha \beta}^{-1}) + T_{\alpha \beta}T_{\alpha \beta}^{-1} \circ d) = \omega_{\beta} - \omega_{\alpha}\\
				&\Rightarrow -T_{\alpha \beta} \circ d(T_{\alpha \beta}^{-1}) = \omega_\beta - \omega_{\alpha}.
			\end{nalign}
			We note that this implies that the derivative of $T_{\alpha \beta}$ in any direction is a multiple of itself. Therefore, $T_{\alpha \beta}$ is a a constant linear matrix times a smooth function.  Thus, write $T_{\alpha \beta} = f_{\alpha \beta} T_{\alpha \beta}^0$ where $T_{\alpha \beta}^0$ is a constant matrix. As $T_{\alpha \beta}$ is invertible, since $U_\alpha \cap U_\beta$ is connected by assumption and because we work over the real numbers, we can explicitly choose $f_{\alpha \beta}$ to be $f_{\alpha \beta} := | \text{det} T_{\alpha \beta}|^{1/n} > 0$. Since $f_{\alpha \beta}$ is always positive, we can consider the logarithm $\log f_{\alpha \beta}$ with values in the real numbers. From \ref{TdiffEq}, we can then deduce
			\begin{nalign}
				\label{EqdlogDiff}
				d(\log f_{\alpha \beta}) = -f_{\alpha \beta} d(f_{\alpha \beta}^{-1}) = \omega_\beta - \omega_{\alpha}.
			\end{nalign}
			As $T_{\alpha \beta}T_{\beta \gamma} = T_{\alpha \gamma}$, we obtain 
			\begin{align*} \log f_{\alpha \beta} + \log f_{\beta \gamma} = \log f_{\alpha \gamma}.\end{align*}
			This means that the assignment $(\alpha, \beta) \mapsto \log f_{\alpha \beta}$ defines a 1-cycle in the \u Cech complex with values in the sheaf $\underline \Omega^0$. However $\underline \Omega^0$ is a fine sheaf and therefore the cohomology group $\hat H^1((U_\alpha); \underline \Omega^0)$ is $0$. Hence there exists a map $\alpha \mapsto q_{\alpha} \in \Omega^0(U_{\alpha})$ such that 
			\begin{nalign}
				\label{QLogEq}
				q_\beta - q_\alpha = (\check d q)(\alpha, \beta) = \log f_{\alpha \beta}.
			\end{nalign} 
			Locally, we now define the form $\omega$ by
			\begin{align*} \omega := \omega_{\alpha} - dq_{\alpha}, \end{align*}
			which does not depend on the choice of $\alpha$ since
			\begin{align*} d(q_\beta - q_\alpha) = d(\log f_{\alpha \beta}) = \omega_{\beta} - \omega_\alpha \end{align*}
			by \ref{EqdlogDiff} and \ref{QLogEq}.
			Then the connection $\nabla' := \nabla - \omega$ is locally of the form
			\begin{align*} \nabla' = d_{\alpha} + dq_{\alpha}\end{align*}
			which has $0$ curvature.	
		\end{proof}
	\end{proposition}
	
	\begin{theorem}
		\label{ThmRealBoring}
		Let $\mathbb K = \mathbb R$. Let $h$ be a closed real 2-form on $M$, then we have the following quasi-equivalences:
		\begin{align*}
			\mathcal P(M)^{\infty}[h] \simeq \begin{cases}
				\mathcal P(M)^{\infty}[0] \text{ if } [h] = 0 \text{ in } H^2_{DR}(M, \mathbb R)\\
				0 \text{ else }
			\end{cases}
		\end{align*}
		\begin{proof}
			Let $(E^*, \mathbb E)$ be an object of $\mathcal P(M)^{\infty}[h]$ with $[h] \neq 0$.
			By Lemma  \ref{trivialDiff}, we see that $(E^*, \mathbb E)$ is homotopy equivalent to an object $(H^*, \mathbb H)$ such that $\mathbb H^0 = 0$. Thus, each $H^i$ is a flat object in $\mathcal P(M)^{0}[h]$. By the previous proposition, this implies that each $H^i$ is equal to $0$ since the curvature of any non-zero projectively flat vector bundle is a coboundary. Thus, $(E^*, \mathbb E)$ is homotopy equivalent to $0$.
		\end{proof}
	\end{theorem}
	
	\section{Riemann-Hilbert correspondence for twisted sheaves}
	In the flat case, the category $\mathcal P(M)^{\infty}$ of graded vector bundles together with a superconnection over $M$ (or equivalently the category of cohesive modules over the de Rham algebra of $M$) has been shown to be equivalent to the derived category $D_{fclc}(\underline{\mathbb K})$ consisting of complexes whose cohomology sheaves are bounded and locally constant of finite rank over the constant sheaf of rings $\underline{\mathbb K}$, see \cite{RiemannHilbert}. In the following, the goal is to show that there is a variation of this statement for a fixed curvature $h \in \Omega^2(M; \mathbb K)$. Sheaves will be replaced by sheaves that are twisted by a fixed gerbe. We will show a more general equivalence which also recovers the result of \cite{Block}, concerning an equivalence between cohesive modules over the curved Dolbeault algebra of a complex manifold $X$ and complexes of twisted sheaves whose cohomology sheaves are coherent and bounded.
	
	\subsection{Twisted sheaves}
	We will begin by defining the notion of a twisted sheaf. This basic definitions are largely analogous to \cite{Toda} with minor modifications allowing for sheaves of cdg-algebras. Let $k$ be a commutative, unital ring and let $X$ be a topological space. Let $(\mathscr A^*, h)$ be a sheaf of cdg-algebras on $X$.  Furthermore let $[\check h] \in \check H^2(M, Z(\mathscr A_{cl}^0)^\times)$ where $Z(\mathscr A_{cl}^0)^\times$ is the sheaf of closed units in the center of $\mathscr A^0$ (to cause no confusion, we intend this to be $Z(\mathscr A_{cl}^0)^\times = Z(\mathscr A^*) \cap (A^0)^\times \cap A^0_{cl}$). This class is represented by some cocycle $\check h \in \check C^2(\mathscr U,Z(\mathscr A_{cl}^0)^\times)$ for some cover $\mathscr U$, both of which will be fixed from now on. The elements of this cover will be written as $U_\alpha$ for $\alpha$ in some indexing set $I$.
	
	\begin{def1}
		A $\check h$-twisted sheaf of left $\mathscr A^*$ cdg-modules $M^*$ is given by the following data:
		\begin{enumerate}
			\item For each $\alpha \in I$ a sheaf $M^*_{\alpha}$ of left $\mathscr A^*|_{U_\alpha}$ cdg-modules on $U_\alpha$.
			\item For each $\alpha, \beta \in I$ an isomorphism $\phi_{\alpha \beta} : M^*_{\alpha}|_{U_\alpha \cap U_\beta} \to M^*_{\beta}|_{U_\alpha \cap U_\beta}$ of left $\mathscr A^*|_{U_\alpha \cap U_\beta}$ cdg-modules on $U_\alpha \cap U_\beta$.
		\end{enumerate}
		The isomorphisms have to fulfill the following property: For each $\alpha, \beta, \gamma \in I$ we have the identity
		\begin{align*} \phi_{\alpha \beta} \circ \phi_{\beta \gamma}  = \check h_{\alpha \beta \gamma}  \phi_{\alpha \gamma}\end{align*}
		on $U_\alpha \cap U_\beta \cap U_\gamma$, i.e., $\check d \phi = \check h$.
	\end{def1}
	
	\begin{def1}
		We define the dg-category $(\mathscr A^*, h)$-Mod$^{\check h}$. Its objects are $\check h$-twisted sheaves of left $\mathscr A^*$ cdg-modules. Given two such objects $M^*$ and $N^*$, a homogeneous morphism $f \in \text{Hom}_{(\mathscr A^*, h)\text{-Mod}^{\check h}}^n(M^*, N^*)$ is given by a collection $(f_{\alpha})_{\alpha \in I}$ such that
		\begin{enumerate}
			\item each $f_{\alpha}$ is a degree $n$ map of graded sheaves $f_{\alpha} : M^*_{\alpha} \to N^{*+n}_\alpha$, compatible with the left action of $\mathscr A^*$,
			\item for each $\alpha$ and $\beta$ we have the identity $\phi^N_{\alpha \beta} \circ f_\alpha|_{U_\alpha \cap U_\beta}  = f_\beta|_{U_\alpha \cap U_\beta} \circ \phi^M_{\alpha \beta}$
		\end{enumerate}
		The differential is given componentwise by $df = d((f_{\alpha})_{\alpha \in I}) = (d_\alpha f_\alpha)_{\alpha \in I}$ where 
		\begin{align*}d_\alpha f_\alpha = d_{N_\alpha^*} \circ f_\alpha - (-1)^{n} f_\alpha \circ d_{M_\alpha^*}.\end{align*} As each $d_\alpha$ squares to $0$, we obtain that $d$ also squares to $0$. Furthermore one easily verifies that $df$ indeed defines a morphism of degree $n+1$ as defined above.
	\end{def1}
	\begin{remark}
		\label{remarkRelate}
		If we have another representative $\check \eta$ with $[\check \eta] = [\check h]$, where $\check \eta$ possibly belong to a different cover, then the dg-categories $(\mathscr A^*, h)$-Mod$^{\check h}$ and $(\mathscr A^*, h)$-Mod$^{\check \eta}$ are dg-equivalent. The proof of this involves a few steps but is straightforward. First, consider some $\check \eta$ belonging to some cover and take a cover which is refining it. The sheaf condition ensures that both dg-categories are dg-equivalent (this would not be true if were dealing with presheaves). If may now assume that $\check \eta$ and $\check h$ belong to the same cover and are cohomologous with respect to this cover, i.e., $\check h_{\alpha \beta \gamma} = \check d(p)_{\alpha \beta \gamma} \check \eta_{\alpha \beta \gamma}$ for some 2-cochain $p$. Then we can use $p$ to twist the gluing isomorphisms $\phi_{\alpha \beta}$ which works because each $p_{\alpha \beta}$ is central and closed. The differentials and morphisms remain unchanged. Thus, up to non-unique dg-equivalence, $(\mathscr A^*, h)$-Mod$^{\check h}$ only depends on the class in  $\check H^2(M, Z(\mathscr A_{cl}^0)^*)$. In particular, for $\check h = 1$ we may choose the cover only consisting of $X$, thus recovering the usual dg-category of sheaves of left $\mathscr A^*$ cdg-modules.
	\end{remark}
	\begin{remark}
		Given an object $M^*$ and an $x \in X$ we can consider the stalks $(M_\alpha)_x$ for each $\alpha$ with $x \in U_\alpha$. These are all isomorphic for different $\alpha \in I$, but we cannot glue them together (in a useful manner).\\ \\
		If $\mathscr A^*$ is a sheaf of dg-algebras, we can view it as a sheaf of cdg-algebras $(\mathscr A^*, 0)$. In this case, we can say an object $M^*$ is acyclic if all $M_{\alpha}^*$ are acyclic complexes of sheaves (meaning that the complexes of stalks are acyclic). Similarly, we can define a notion of quasi-isomorphism. Localising the homotopy category $H^0((\mathscr A^*, 0)\text{-Mod}^{\check h})$ along quasi-isomorphisms, we obtain the derived category $D(\mathscr A^*)^{\check h}$. 
		\\ \\
		Furthermore, if $\mathscr R$ is a sheaf of $k$-algebras, we can consider the full subcategory of $H^0((\mathscr R, 0)\text{-Mod}^{\check h})$ consisting of objects concentrated in degree $0$ and denote it by $\mathscr R$-Mod$^{\check h}_0$. This category is abelian, its category of chain complexes up to homotopy $K(\mathscr R$-Mod$^{\check h}_0)$ recovers $H^0((\mathscr R, 0)\text{-Mod}^{\check h})$ and its derived category is $D(\mathscr R)^{\check h}$.
	\end{remark}
	
	\begin{def1}
		Let $X$ and $Y$ be topological spaces. Let $\mathcal O_Y$ be a sheaf of $k$-algebras on $Y$ and $\mathcal O_X$ a sheaf of $k$-algebras on $X$. Let $[\check h] \in \check H^2(\mathscr U, Z(\mathcal O_Y)^\times)$ for some cover $\mathscr U$ of $Y$. Furthermore let $f : (X, \mathcal O_X) \to (Y, \mathcal O_Y)$ be a morphism of ringed spaces. Pulling back $\check h$ along $f$ we obtain an element $[\check h_X] \in \check H^2(f^*(\mathscr U), \mathcal O_X^\times)$ where $f^*(\mathscr U)$ is the pullback of the cover $\mathscr U$ to $X$. We assume that $\check h_X$ has values in the center of $\mathcal O_X$. We define the direct image and inverse image functors $f_*$ and $f^*$.
		Let $M$ be an object of $\mathcal O_X$-Mod$^{\check h_X}_0$. We define $f_*(M) \in \mathcal O_Y$-Mod$^{\check h}_0$ by 
		\begin{align*} f_*(M)_{\alpha} = (f|_{U_\alpha})_*(M_\alpha), \phi_{\alpha \beta}^Y = (f|_{U_\alpha \cap U_\beta})_*(\phi_{\alpha \beta}^X)\end{align*}
		where $f_*|_U$ denotes the usual direct image functor $f_*|_U : (f^{-1}(U), O_{f^{-1}(U)}) \to (U, O_U)$ It is straightforward to verify that this indeed defines an object of $\mathcal O_Y$-Mod$^{\check h}_0$ and it naturally extends to morphisms. Similarly we can define the inverse image functor $f^* : \mathcal O_Y$-Mod$^{\check h}_0 \to \mathcal O_X$-Mod$^{\check h_X}_0$.  We know that for the untwisted versions of $f^*$ and $f_*$ we have that $f^*$ is left adjoint to $f_*$. By spelling out definitions, it follows that this remains true in the twisted case (by reducing to the untwisted case). Furthermore, by comparing stalks, it is clear that $f^*$ is exact if $f$ is a flat morphism (i.e., if $f^{-1}(\mathcal O_Y)$ is flat over $\mathcal O_X$). In particular, this is the case when $X$ is a subspace of $Y$ and $\mathcal O_X = \mathcal O_Y|_X$.
	\end{def1}
	
	\begin{lemma}
		Let $\mathscr R$ be a sheaf of $k$-algebras. The abelian category $\mathscr R$-Mod$^{\check h}_0$ of twisted $\mathscr R$-modules has enough injectives.
		\begin{proof}
			This is  \cite[Lemma 4.3]{Toda}.
		\end{proof}
	\end{lemma}
	In particular this implies that for every complex $E^*$ in $(\mathscr R, 0)$-Mod$^{\check h}$, there is a $K$-injective complex $I^*$ and a quasi-isomorphism $I^* \to E^*$. 
	
	\begin{def1}
		Let $\mathscr R$ be a sheaf of commutative $k$-algebras. Let $\omega_1$ and $\omega_2$ be twists and assume they are taken to be relative to the same cover (otherwise one can choose a common refinement). Let $M$ be a $\omega_1$-twisted sheaf of $\mathscr R$-modules and $N$ a $\omega_2$-twisted sheaf of $\mathscr R$-modules. We define their tensorproduct, which will be a $\omega_1 \cdot \omega_2$ twisted sheaf $N \otimes_{\mathscr R} M$. For each $\alpha$ we set
		\begin{align*} (N \otimes_{\mathscr R} M)_\alpha := N_\alpha \otimes_{\mathscr R|_{U_\alpha}} M_\alpha,\end{align*}
		being the ordinary tensor-product of sheaves. For each $\alpha, \beta \in I$ the corresponding isomorphism is $\phi^M_{\alpha \beta} \otimes \phi^N_{\alpha \beta}$. Thus it is clear that $N \otimes_{\mathscr R} M$ defines a $\omega_1 \cdot \omega_2$-twisted sheaf.
	\end{def1} 
	The only case of interest for us will be when $M$ is $\omega$-twisted and $N$ untwisted, so that $N \otimes_{\mathscr R} M$ is $\omega$-twisted.
	
	\begin{def1}
		Let $\mathscr A^*$ be a sheaf of cdg-algebras. We say that $\mathscr A^*$ exponentiates if there is a morphism of sheaves of groups $\exp : (\mathscr A^0, +) \to (Z(\mathscr A^0)^\times, \cdot)$ such that for each $U \subset X$ and section $s \in \mathscr A^0(U)$ we have 
		\begin{align*}d\exp(s) = d(s) \exp(s). \end{align*}
	\end{def1}
	\subsection{Twisted Riemann-Hilbert correspondence}
	
	From now on we will fix the following data and assumptions:
	\begin{enumerate}
		\label{ListAssumptions}
		\item A connected, Hausdorff and paracompact topological space $X$ of finite covering dimension.
		\item A sheaf of graded commutative dg-algebras $\mathscr A^*$ satisfying the following properties:
		\begin{itemize}
			\item $\mathscr A^i = 0$ for all $i < 0$ and $\mathscr A^n = 0$ for $n \gg 0$.
			\item $\mathscr A^0$ is a soft sheaf (and thus each $\mathscr A^i$ is soft).
			\item $\mathscr A^*$ exponentiates.
		\end{itemize}
		\item A closed global section $h \in \mathscr A^2(X)$.
		\item A sheaf of commutative $k$-algebras $\mathscr R$ together with a quasi-isomorphism
		\begin{align*} \rho : \mathscr R \to \mathscr A^*\end{align*} of sheaves of dg-algebras. 
		\item For every open $U \subset X$, $x \in U$ and a left $\mathscr A^*|_U$ dg-module $M^*$ whose underlying graded module is of the form $\mathscr A^*|_U \otimes_{\mathscr A^0} V^*$ for a finitely generated projective, graded $\mathscr A^0$-module $V^*$, there exists an open subset $x \in U' \subset U$, a finitely generated, graded $\mathscr R|_{U'}$-module $W^*$, consisting of direct summands of free $\mathscr R|_{U'}$-modules in each degree so that $\mathscr A^0  \otimes_{\mathscr R|_{U'}} W^* \cong V^*$ (as graded objects), and a homotopy equivalence \begin{align*}M^* \to (\mathscr A^*|_{U'} \otimes_{\mathscr R|_{U'} } (\mathscr R|_{U'}  \otimes_k V), d_{\mathscr A^*|_{U'}}  \otimes \mathbbm 1 + \mathbbm 1 \otimes d),\end{align*} where $d$ is some $\mathscr R|_{U'}$-linear differential on $W^*$.
	\end{enumerate}
	\begin{example}
		There are two main examples satisfying the above assumptions. The first one is given by a real smooth manifold $M$ with $k = \mathbb K$, $\mathcal R = \underline{\mathbb K}$ and $\mathscr A^* = \underline{\Omega}^*(M, \mathbb K)$. The other one is given by a complex manifold $X$ with $k = \mathbb C$, $\mathcal R = \mathcal O_X$ and $\mathscr A^*$ the Dolbeault algebra sheaf on $X$.
	\end{example}
	\begin{def1}
		Note that $d \circ \exp \circ \rho = 0$. As $0 \to \mathscr R \to \mathscr A^0 \to \mathscr A^1$ is exact, we obtain that $\exp \circ \rho$ restricts to an exponential map $\exp_{\mathscr R} : \mathscr R \to \mathscr R$. \\ \\
		Reconstructing the isomorphisms
		\begin{align*}H^2(\mathscr A^*(X)) \cong \check H^1(X, \mathscr A^1_{cl}) \cong \check H^2(X, \mathscr A^0_{cl}) = \check H^2(X, \mathscr R)\end{align*}
		we furthermore obtain and fix the following data:
		\begin{enumerate}
			\item An open cover $\mathscr U$ of $X$, indexed by $I$.
			\item For each $\alpha \in I$, sections $h_\alpha^0 \in \mathscr A^1(U_\alpha)$ such that $d(h_\alpha^0) = h|_{U_\alpha}$.
			\item For each $\alpha, \beta \in I$, sections $h^1_{\alpha \beta} \in \mathscr A^0(U_\alpha \cap U_\beta)$, such that \[d(h^1_{\alpha \beta} ) = h_\beta^0  - h_\alpha^0 = \check d(h^0)_{\alpha \beta} \text{ on }U_\alpha \cap U_\beta.\]
			\item For each $\alpha, \beta, \gamma \in I$ sections $h^2_{\alpha \beta \gamma} \in \mathscr R(U_\alpha \cap U_\beta \cap U_\gamma)$ such that \[\rho(h^2_{\alpha \beta \gamma})  = h^1_{\beta \gamma} - h^1_{\alpha \gamma} + h^1_{\alpha \beta} = \check d(h^1)_{\alpha \beta \gamma} \text{ on }U_\alpha \cap U_\beta \cap U_\gamma.\]
			\item Moreover, we define the  \v Cech 2-cocycle $\check h_{\alpha \beta \gamma} := \exp_{\mathscr R}(h^2_{\alpha \beta \gamma}) \in \mathscr R(U_\alpha \cap U_\beta \cap U_\gamma)^\times$, defining a twist.
		\end{enumerate}
		Since $h$ is a closed and central, $(\mathscr A^*, h)$ defines a sheaf of cdg-algebras. We will consider the dg-category $(\mathscr A^*, h)\text{-Mod} := (\mathscr A^*, h)$-Mod$^{1}$. Its homotopy category will be denoted by $K_h( \mathscr A^*)$. In particular the twist $\check h$ appears only as the curvature $h$, but not itself as a twist of sheaves. However, note that in virtue of remark \ref{remarkRelate}  we have that $(\mathscr A^*, h)$-Mod$^{1}$ is dg-equivalent to $(\mathscr A^*, h)\text{-Mod}^{\rho(\check h)}$, since $H^2(X, \mathscr A^0) = 0$ and $\rho(\check h)$ is in the image of $\exp : \check H^2(X, \mathscr A^0) \to \check H^2(X, (\mathscr A^0)^\times)$. Furthermore, we consider the abelian category $\mathscr R\text{-Mod}^{\check h}_0$, its associated dg-category $\mathscr R\text{-Mod}^{\check h}$, its homotopy category $K(\mathscr R)^{\check h}$ and its derived category $D(\mathscr R)^{\check h}$. In the following, we will define a dg-functor $(\mathscr A^*, h)\text{-Mod}  \to \mathscr R\text{-Mod}^{\check h}$.
	\end{def1}
	\begin{def1}
		\label{DefJFunctor}
		The dg-functor $J : (\mathscr A^*, h)\text{-Mod}  \to \mathscr R\text{-Mod}^{\check h}$ is defined as follows. Let $M^*$ be an object of $(\mathscr A^*, h)\text{-Mod}$. We set
		$J(M^*) = (M^*_\alpha, \exp(h^1_{\alpha \beta}))$, where, as a graded $\mathscr R|_{U_\alpha}$-module, $M^*_\alpha = M^*|_{U_\alpha}$. Denote by $d|_{U_\alpha}$ the restriction of the differential of $M^*$ to $U_\alpha$. The differential $d_\alpha$ on $M^*_\alpha$ is given by $d_\alpha := d|_{U_\alpha} - h^0_{\alpha}$. As $\mathscr A^*$ is graded commutative, we see that $d_\alpha$ squares to zero, using that $d(h_\alpha^0) = h|_{U_\alpha} = (d|_{U_\alpha})^2$. Since $\rho(h^2_{\alpha \beta \gamma})  = h^1_{\beta \gamma} - h^1_{\alpha \gamma} + h^1_{\alpha \beta}$, this sheaf indeed defines an object of $\mathscr R\text{-Mod}^{\check h}$. $J$ is defined on morphisms by restriction. This is indeed compatible with the differentials because the $h_\alpha^0$ terms cancel. It is also compatible with the isomorphisms $\exp(h^1_{\alpha \beta})$ because they are central.
	\end{def1}
	
	\begin{remark}
		\label{remarkUntwist}
		Similarly, we can define a functor in the other direction $\mathscr R\text{-Mod}^{\check h} \to (\mathscr A^*, h)\text{-Mod}$. Let $M^*$ be an object of $\mathscr R\text{-Mod}^{\check h}$. Then we can consider the $\check h$-twisted sheaf $\mathscr A^* \otimes_{\mathscr R} M^*$ over $\mathscr A^*$, whose differential squares to $0$. We obtain a $1$-twisted sheaf by letting $(\mathscr A^* \otimes_{\mathscr R} M^*)^{untw}_\alpha := (\mathscr A^* \otimes_{\mathscr R} M^*)_\alpha$ with the differential $d^{untw}_\alpha := d_\alpha + h_\alpha^0$ and $\phi_{\alpha \beta}^{untw} = \exp(-h^1_{\alpha \beta}) \phi_{\alpha \beta}$. The differential then precisely squares to $h$. By Remark \ref{remarkRelate}, this corresponds to an untwisted sheaf over $(\mathscr A^*,h)$.
	\end{remark}

	We would like $J$ to induce some kind of equivalence. However, since $\rho : \mathscr R \to \mathscr A^*$ is only a quasi-isomorphism, we cannot expect this to be a dg-quasi equivalence. 
	\begin{def1}We define the full dg-subcategory  $ (\mathscr A^*, h)\text{-Mod}^{coh}$ of $(\mathscr A^*, h)\text{-Mod}$ consisting of objects whose underyling graded $\mathscr A^*$-module is of the form $\mathscr A^* \otimes_{\mathscr A^0} E^*$ where $E^*$ is a graded, $\mathscr A^0$-module which is a direct summand of a finitely generated free graded $\mathscr A^0$-module. 
	\end{def1}
	Note that $\mathscr A^0$ is projective as an object in the abelian category $\mathscr A^0$-Mod$_0$ since the global sections functor is exact (because $\mathscr A^0$ is soft). Moreover any $\mathscr A^0$-module is generated by its global sections because $\mathscr A^0$ is soft. Thus, $E^*$ is equivalently just a finitely graded object consisting of finitely generated projective $\mathscr A^0$-modules in each degree. We also define the cdg-algebra $(A^*, h) := (\mathscr A^*(X), h)$. 
	\begin{def1}
		Similarly we define the full dg-subcategory $(A^*, h)$-Mod$^{coh}$ of $(A^*, h)$-Mod  consisting of objects whose underlying graded $A^*$-module is of the form $A^* \otimes_{A^0} E^*$ where $E^*$ is a finitely generated, graded projective $A^0$-module.
	\end{def1}
	Note that both dg-categories can be obtained as dg-categories of twisted complexes of a cdg-category (as in Remark \ref{remarkCohesive}). In particular, this allows us to use the results of section \ref{quasiEqSection}.
	
	\begin{lemma}
		\label{simpleDGEquivalence}
		There is a dg-equivalence $(A^*, h)\text{-Mod}^{coh} \to (\mathscr A^*, h)\text{-Mod}^{coh}$.
		\begin{proof}
			The functor is the restriction of the pullback functor along the map $X \to pt$ to cohesive modules. Note that we do not have to sheafify because of finitely generated projectivity. It is then straightforward to see that the functor is essentially surjective. To show that it induces bijections on the morphism sets, we may forget about all differentials and gradings. Furthermore it suffices to show the claim in the case that $E^*$ is free, for in the general case one finds a finitely generated projective module $E'$ such that $E \oplus E'$ is free. Let $n$ be the number of generators of $E$. Then we have 
			\begin{nalign} \text{Hom}_{A^*}(A^* \otimes_{A^0} E^*, A^* \otimes_{A^0} F^*) &= \bigoplus_{i = 1}^n \text{Hom}_{A^*}(A^*, A^* \otimes_{A^0} F^*) \\ &= \bigoplus_{i = 1}^n A^* \otimes_{A^0} F^*\\&=  \bigoplus_{i = 1}^n \text{Hom}_{\mathscr A^*}(\mathscr A^*, \mathscr A^* \otimes_{A^0} F^*) \\&= \text{Hom}_{\mathscr A^*}(\mathscr A^* \otimes_{A^0} E^*, \mathscr A^* \otimes_{A^0} F^*)  
			\end{nalign}
		\end{proof}
	\end{lemma}
	\begin{lemma}
		\label{LemmaAcyclic}
		Let $M^*$ be an acyclic, bounded below complex of sheaves of abelian groups and assume that each $M^i$ is acyclic (with respect to $X$). Then the complex $M^*(X)$ is acyclic.
		\begin{proof}
			This follows from general category theory. By definition, the $n$-th sheaf cohomology of $M^i$ is the $n$-th right derived functor of the global sections functor. Then we can apply \cite[\href{https://stacks.math.columbia.edu/tag/015E}{Tag 015E}]{stacks-project}.
		\end{proof}
	\end{lemma}
	
	\begin{lemma}
		\label{HighProjective}
		Let $\mathscr C$ be an abelian category. Let $B^*$ be a bounded above complex consisting of projectives in each degree. For any $n \in \mathbb N$ such that $H^i(B^*) = 0$ for $i \geq n$ we have that $\text{Im}(B^{n-1} \to B^n)$ is projective.
		
		\begin{proof}
			We truncate our complex iteratively. Suppose $B^i = 0$ for $i > m > n$, then $B^{m-2} \to B^{m-1} \to B^m \to 0$ is exact. Since $B^m$ is projective, in the truncation $B^{m-2} \to \text{Im}(B^{m-2} \to B^{m-1}) \to 0$ the middle term is a direct summand of $B^m$ and hence projective.
		\end{proof}
	\end{lemma}
	\begin{lemma}
		\label{l11}
		
		Let $\mathscr C$ be an abelian category. Consider a quasi-isomorphism of complexes $A^*$ and $B^*$ of the following form:
		\[\begin{tikzcd}
			... \arrow[r] & 0 \arrow[r] \arrow[d] & A^1 \arrow[r] \arrow[d] & A^2 \arrow[r] \arrow[d] & ... \arrow[r] & A^{n-2} \arrow[r] \arrow[d] & A^{n-1} \arrow[r] \arrow[d] & 0 \arrow[r] \arrow[d] & ... \\
			... \arrow[r] & B^0 \arrow[r]         & B^1 \arrow[r]           & B^2 \arrow[r]           & ... \arrow[r] & B^{n-2} \arrow[r]           & B^{n-1} \arrow[r]           & B^{n} \arrow[r]       & ...
		\end{tikzcd}\]
		such that $A^*$ is K-projective and $B^*$ consists of projective objects and $B^*$ is bounded above. Then $\text{Im}(B^0 \to B^1)$ and $\text{Im}(B^{n-1} \to B^n)$ are projective.
		\begin{proof}
			$\text{Im}(B^{n-1} \to B^n)$ is projective by the previous Lemma.\\ \\
			$\text{Im}(B^0 \to B^1)$  is projective: Let $I : = \text{Im}(B^0 \to B^1)$. We will show that $\text{Ext}^1(I,M) = 0$ for all $M$. First recall that $\text{Ext}^1(I,M) = \text{Hom}_{D(\mathscr C)}(I[0], M[-1])$. Consider the complex $\tilde B^*$ which is given by $\tilde B^i = B^i$ for $1 \leq i \leq n-1$, 0 for $i > n$ or $i \leq 0$ and equal to $\text{Im}(B^{n-1} \to B^n)$ for $i = n$. Next let $C^*$ be the mapping cone of $A^* \to \tilde B^*$. Then $C^*$ is still $K$-projective and we can calculate
			\begin{align*}\text{Hom}_{D(\mathscr C)}(I[0], M[-1]) = \text{Hom}_{D(\mathscr C)}(C^*, M[-1]) = \text{Hom}_{K(\mathscr C)}(C^*, M[-1]) = 0\end{align*}
			since $C^*$ is concentrated in non-negative degrees.
		\end{proof}
	\end{lemma}
	
	\begin{def1}
		\label{PerfectDefinition}
		Let $\mathscr Q$ be a sheaf of rings.
		\begin{enumerate}
			\item We say that a complex of $\mathscr Q$-modules $M^*$ is \emph{strictly perfect} if it is bounded and consists of direct summands of finitely generated free modules in each degree.
			\item
			We say that a complex $M^*$ of $\mathscr Q$-modules is \emph{perfect} if for every $x \in X$, there exists an open neighbourhood $x \in U \subset X$ such that $M^*_{|U}$ is quasi-isomorphic to a strictly perfect complex. We denote the corresponding full subcategory of $D(\mathscr R)$ by $D_{\text{perf}}(\mathscr R)$.
			\item
			We say that a complex $M^*$ of $\mathscr Q$-modules is \emph{globally bounded perfect with bounding constants $a < b \in \mathbb Z$ and $N \in \mathbb N$} if there exists an open neighbourhood  $U$ of $x$ such that $M^*_{|U}$ is quasi-isomorphic to a strictly perfect complex which is concentrated in degrees $[a, b]$ and has at most $N$ generators.  
			\item We say that a complex $M^*$ of $\mathscr Q$-modules is \emph{globally bounded perfect} if $M^*$ is globally bounded perfect for some bounding constants.
			We denote the corresponding full subcategory of $D(\mathscr Q)$ by $D^B_{\text{perf}}(\mathscr Q)$.
			\item Furthermore, if $\omega$ is a twist, then we say that a complex $M^*$ of twisted $\mathscr Q$-modules is \emph{perfect} if each $M^*_{\alpha}$ is perfect. We say that $M^*$ is \emph{globally bounded perfect} if there are numbers $a < b$ and $N \in \mathbb N$ such that each $M^*_{\alpha}$ is globally bounded perfect with bounding constants $a < b$ and $N$. The corresponding derived categories will be denoted by $D_{\text{perf}}(\mathscr Q)^{\omega}$ and $D^B_{\text{perf}}(\mathscr Q)^{\omega}$.
		\end{enumerate}
		
	\end{def1}
	\begin{lemma}
		\label{p8}
		Let $Y$ be Hausdorff paracompact topological space of finite covering dimension. Consider the category $C$ generated by 2 objects $t$ and $f$ generated by a morphism $f \to t$. Let $PSh(Y,C)$ be the category of $C$-valued presheaves on $Y$. Let $s \in PSh(Y,C)$ such that for every $x \in Y$ there exists an open neighbourhood $U$ of $x$ with $s(U) = t$. Additionally, assume that for every collection of pairwise disjoint opens $(U_i)$ with $s(U_i) = t$ for all $i$ it follows that $s(\amalg U_i) = t$.  Then there is a finite cover $(U_i)$ of $Y$ such that $s(U_i) = t$ for all elements of the cover $U_i$.
		\begin{proof}
			The proof closely follows \cite[Proposition III.4.1]{Wells}. By assumption we can find a cover $(U_\alpha)$ such that $s(U_\alpha) = t$ for all $\alpha$. Since there is no morphism $t \to f$ we can assume that $(U_\alpha)$ has finite dimension because $Y$ is required to have finite covering dimension. As $Y$ is Hausdorff and paracompact we can find a partition of unity $\varepsilon_\alpha$ with values in $[0,1]$ subordinate to $(U_\alpha)$. We define $A_i := \{ \{\alpha_0,...,\alpha_i\} | \alpha_i \neq \alpha_j \text{ for } i \neq j\}$. For $a = \{\alpha_0,...,\alpha_i\} \in A_i$ we let 
			\begin{align*}
				W_{ia} = \{x \in Y | \varepsilon_\alpha(x) < \text{min}(\varepsilon_{\alpha_0}(x),...,\varepsilon_{\alpha_i}(x)) \text{ for every } \alpha \notin a\}.
			\end{align*}
			Each $W_{ia}$ is open and one can see that the union $X_i := \cup_a W_{ia}$ is disjoint. We also have the finite union $\cup X_i = Y$ and each $W_{ia}$ is contained, for example in $U_{\alpha_0}$. By our assumption this means that $s(X_i) = t$ for all $i$.
		\end{proof}
	\end{lemma}
	\begin{corollary}
		\label{c6}
		
		Let $Y$ be as above and let $\mathscr S$ be a sheaf of rings on $Y$. Let $M$ be a $\mathscr S$-module such that there exists $N \in \mathbb N$ so that for every $x \in Y$ there exists an open $x \in U$ and a surjection $\mathscr S^{\oplus N}|_U \to M|_U$. Then there exist $N' \in \mathbb N$ and a surjection $\mathscr S^{\oplus N'} \to M$.
		\begin{proof}
			Apply the Lemma above to the presheaf \begin{align*}U \mapsto  \left\{
				\begin{array}{ll}
					t & \text{if there exists a surjection } \mathscr S^{\oplus N}|_U \to M|_U,  \\
					f & \, \textrm{else} .\\
				\end{array}
				\right. \end{align*}
		\end{proof}
	\end{corollary}
	
	\begin{corollary}
		\label{c7}
		Let $Y$ be as before and let $\mathscr S$ be a sheaf of rings on $Y$. Let $M$ be a $\mathscr S$-module such that there exists $N \in \mathbb N$ so that for every $x \in Y$ there exists an open $x \in U$ and an isomorphism $\mathscr S^{\oplus N}|_U \to M|_U$. Then there is a finite open cover $(U_i)$ such that $M|_{U_i}$ is isomorphic to $\mathscr S^{\oplus N}|_{U_i}$.
		\begin{proof}
			Apply the Lemma above to the presheaf \begin{align*}U \mapsto  \left\{
				\begin{array}{ll}
					t & \text{if there exists an isomorphism } \mathscr S^{\oplus N}|_U \to M|_U,  \\
					f & \, \textrm{else.} \\
				\end{array}
				\right. \end{align*}
		\end{proof}
	\end{corollary}
	\begin{lemma}
		Let $\mathscr S$ be a soft sheaf of commutative rings on any topological space $Y$. Let $E$ be a $\mathscr S$-module, such that for every $x \in Y$ there exists an open neighbourhood $U$ of $x$ such that $E|_U$ is projective over $\mathscr S|_U$. Then $E$ is projective.
		\begin{proof}
			We show that $\text{Ext}^1(E, M) = 0$ for all $\mathscr S$-modules $M$. Let $r : M \to I^*$ be an injective resolution of $M$. Then $\text{Ext}^1(E,M)=  \text{Hom}_{K(\mathscr S)}(E[1],I^*)$. Denote by $ \text{\underline {Hom}}_{\mathscr S}(E[1], I^*)$ the sheaf of complexes $U \mapsto \text{Hom}_{\mathscr S}(E[1]|_U, I^*|_U)$. The map
			\begin{align*} r_* : \text{\underline {Hom}}_{\mathscr S}(E[1], M[0]) \to \text{\underline {Hom}}_{\mathscr S}(E[1], I^*)\end{align*}
			is a quasi-isomorphism of complexes of sheaves since $E$ is locally projective. Both complexes are bounded below and consist of acyclic objects in each degree (as $\mathscr S$ is soft and commutative), thus we obtain that 
			\begin{align*} H^0(r_*(X)) : \text{{Hom}}_{K(\mathscr S)}(E[1], M[0]) \to \text{{Hom}}_{K(\mathscr S)}(E[1], I^*)\end{align*}
			is a quasi-isomorphism by Lemma \ref{LemmaAcyclic}. As the left side is $0$, we are done.
			
		\end{proof}
	\end{lemma}
	\begin{lemma}
		\label{SoftStrict}
		Let $\mathscr S$ be a soft sheaf of commutative rings on a paracompact Hausdorff topological space of finite covering dimension. Any globally bounded perfect complex of sheaves $M^*$ over $\mathscr S$ is quasi-isomorphic to a strictly perfect complex.
		\begin{proof}
			Let $M^*$ be a globally bounded perfect complex with bounding constants $a < b$ and $N$. We first construct a bounded above complex. By (\cite{Thomason}, Lemma 1.9.5) it is enough to show that if $H^i(M^*) = 0$ for $i > n$ then for any surjection $A \to H^n(M^*)$ there exists a finite free $\mathscr S$-module $F$ and a map $F \to A$ such that the composite $F \to H^n(M^*)$ is surjective. Furthermore, since free modules are projective, as $\mathscr S$ is soft, it suffices to find a surjection $F \to H^n(M^*)$, i.e. $H^n(M^*)$ is finitely generated by global sections. Locally $H^n(M^*)$ can be computed as the highest non-vanishing cohomology group of a strictly perfect complex  $S^*$  which has no more than $N$ generators. The proof of Lemma \ref{HighProjective} shows that there is a surjection $F \to H^n(S^*)|_U = H^n(M^*)|_U$ where $F$ is a free $\mathscr S|_U$-module of rank $\leq N$. By Corollary \ref{c6} we obtain the claim.\\ \\
			We truncate our found complex $P^*$ below in degrees of vanishing cohomology. We apply Lemma \ref{l11} locally and obtain that the lowest non-vanishing object is projective and finitely generated by at most $N$ generators. Again using corollary \ref{c6}, there exists a finite free $\mathscr S$-module which surjects onto it globally. By the previous Lemma on the other hand, we have that the considered object is (globally) projective.
		\end{proof}
	\end{lemma}
	
	Recall that we had constructed a functor $J : (\mathscr A^*, h)\text{-Mod}  \to \mathscr R\text{-Mod}^{\check h}$. We obtain an induced functor $H^0(J)$ of the corresponding homotopy categories. We further compose this functor with the localization functor $K(\mathscr R)^{\check h} \to D(\mathscr R)^{\check h}$. We further restrict it to cohesive modules and thus obtain a functor
	\begin{align*} \mathscr J : H^0((\mathscr A^*, h)\text{-Mod}^{coh}) \to D(\mathscr R)^{\check h}.\end{align*}
	\begin{lemma}
		$\mathscr J$ factors through $D^B_{perf}(\mathscr R)^{\check h}$.
		\begin{proof}
			Let $\mathscr A^* \otimes_{\mathscr A^0} E^*$ be an object of $(\mathscr A^*, h)\text{-Mod}^{coh}$ where $E^*$ is a finitely generated projective, graded $\mathscr A^0$-module. Suppose that $E^*$ is concentrated in degrees $[a,b]$ and has no more than $N$ generators. For every $\alpha \in I$ and $x \in U_\alpha$ we can find by assumption an open neighbourhood $U \subset U_\alpha$ of $x$ and a finitely generated, graded $\mathscr R|_{U}$-module $V^*$, concentrated in degrees $[a,b]$ which consists of direct summands of free modules in each degree and has no more than $N$ generators such that there is a homotopy equivalence 
			\begin{align*}J(\mathscr A^* \otimes_{\mathscr A^0} E^*)_{\alpha}|_U \to (\mathscr A^*|_{U} \otimes_{\mathscr R|_{U} } V^* , d \otimes \mathbbm 1 + \mathbbm 1 \otimes d).\end{align*}
			$V^*$ is $K$-flat so $\mathscr A^*|_{U} \otimes_{\mathscr R|_{U} } V^*$ is quasi-isomorphic to $V^*$. Thus $J(\mathscr A^* \otimes_{\mathscr A^0} E^*)_{\alpha}$ is perfect globally bounded with bounding constants $a,b$ and $N$ for every $\alpha$. By definition, this means that $\mathscr J(\mathscr A^* \otimes_{\mathscr A^0} E^*)$ is perfect globally bounded.
		\end{proof}
	\end{lemma}		
	\begin{proposition}
		$\mathscr J : H^0((\mathscr A^*, h)\text{-Mod}^{coh}) \to D^B_{perf}(\mathscr R)^{\check h}$ is essentially surjective.
		\begin{proof}
			Let $M^*$ be an object of $D^B_{perf}(\mathscr R)^{\check h}$ with bounding constants $a < b$ and $N \in \mathbb N$ and choose $M^*$ to be flat over $\mathscr R$. Then $\mathscr A^* \otimes_{\mathscr R} M^*$ is a $\check h$-twisted sheaf over $\mathscr A^*$.  Furthermore, $\mathscr A^0 \otimes_{\mathscr R} M^*$ is a globally bounded perfect complex over $\mathscr A^0$. As $\mathscr A^0$ is soft, it is quasi-isomorphic to a strictly perfect complex $S^*$ by Lemma \ref{SoftStrict} (the twist is irrelevant in this case because $\mathscr A^0$ is soft and the twist has logarithmic type). Again using that $\mathscr A^0$ is soft, it follows that $S^*$ is K-projective, which means that the quasi-isomorphism can be represented by a map of complexes $f^0 : S^* \to \mathscr A^0 \otimes_{\mathscr R} M^*$. By Proposition \ref{ExtendLemma} (again using that $S^*$ is $K$-projective), there exists a differential on $\mathscr A^* \otimes_{\mathscr A^0} S^*$ and an extension of $f^0$ to a closed morphism $f : \mathscr A^* \otimes_{\mathscr A^0} S^* \to \mathscr A^* \otimes_{\mathscr R} M^*$. A straightforward spectral sequence argument shows that $f$ is a quasi-isomorphism, using the filtrations $F^i  \mathscr A^* \otimes_{\mathscr A^0} S^* = \mathscr A^{\geq i} \otimes_{\mathscr A^0} S^*$ and $F^i \mathscr A^* \otimes_{\mathscr R} M^* = \mathscr A^{\geq i} \otimes_{\mathscr R} M^*$. But $\mathscr A^* \otimes_{\mathscr A^0} S^*$ is clearly in the image of $\mathscr J$ after untwisting the sheaves by the sections $\exp(-h^0_{\alpha \beta}) \in \mathscr A^0(U_\alpha \cap U_\beta)^*$ and twisting the differentials by $h^0_{\alpha} \in \mathscr A^1(U_\alpha)$ (compare Remark \ref{remarkUntwist}). Since $M^*$ is flat, and $R \to A^*$ is a quasi-isomorphism, we have an $\mathscr R$-linear quasi-isomorphism $M^* \simeq \mathscr A^* \otimes_{\mathscr A^0} S^*$.
		\end{proof}
	\end{proposition}
	
	\begin{proposition}
		$\mathscr J : H^0((\mathscr A^*, h)\text{-Mod}^{coh}) \to D^B_{perf}(\mathscr R)^{\check h}$ is fully faithful.
		\begin{proof}
			Let $\mathscr A^* \otimes_{\mathscr A^0} E^*$ and $\mathscr A^* \otimes_{\mathscr A^0} F^*$ be objects of $(\mathscr A^*, h)\text{-Mod}^{coh}$. Let $\mathscr J(\mathscr A^* \otimes_{\mathscr A^0} F^*) \to I^*$ be a quasi-isomorphism of twisted $\mathscr R$-modules such that $I^*$ is bounded below and $K$-injective. We consider the complex of sheaves $U \mapsto \text{Hom}_{\mathscr R|_U}(J(\mathscr A^* \otimes_{\mathscr A^0} E^*), I^*)$ which we will denote by $\text{\underline{Hom}}_{\mathscr R}(J(\mathscr A^* \otimes_{\mathscr A^0} E^*), I^*)$. Notice that this indeed defines a (untwisted) sheaf as the twists cancel each other. Furthermore note that this sheaf is endowed with an action of $\mathscr A^0$, thus it is soft. On the other hand we obtain a complex of sheaves $\text{\underline{Hom}}_{\mathscr A^*}(\mathscr A^* \otimes_{\mathscr A^0} E^*, \mathscr A^* \otimes_{\mathscr A^0} F^*)$ which is also soft. Note that \begin{align*}H^0(\text{\underline{Hom}}_{\mathscr R}(J(\mathscr A^* \otimes_{\mathscr A^0} E^*), I^*)(X)) = \text{{Hom}}_{D(\mathscr R)^{\check h}}(J(\mathscr A^* \otimes_{\mathscr A^0} E^*), J(\mathscr A^* \otimes_{\mathscr A^0} F^*))\end{align*}
			and 
			\begin{align*}H^0(\text{\underline{Hom}}_{\mathscr A^*}(\mathscr A^* \otimes_{\mathscr A^0} E^*, \mathscr A^* \otimes_{\mathscr A^0} F^*)(X)) = \text{Hom}_{H^0((\mathscr A^*, h)\text{-Mod}^{coh})}(\mathscr A^* \otimes_{\mathscr A^0} E^*, \mathscr A^* \otimes_{\mathscr A^0} F^*).\end{align*}
			Furthermore, both complexes are bounded below. By Lemma \ref{LemmaAcyclic} it suffices to show that the map 
			\begin{align*}\text{\underline{Hom}}_{\mathscr A^*}(\mathscr A^* \otimes_{\mathscr A^0} E^*, \mathscr A^* \otimes_{\mathscr A^0} F^*) \to \text{\underline{Hom}}_{\mathscr R}(J(\mathscr A^* \otimes_{\mathscr A^0} E^*), I^*)\end{align*} is a quasi-isomorphism of complexes of sheaves. Let $\alpha \in I$ and $x \in U_\alpha$. By our assumptions, we can find an open neighbourhood $U \subset U_\alpha$ of $x$ such that there is a strictly perfect complex $V^*$ of $\mathscr R|_U$-modules and a homotopy equivalence 
			\begin{align*} \mathscr A^* \otimes_{\mathscr R|_U} V^* \to J(\mathscr A^* \otimes_{\mathscr A^0}  E^*.)\end{align*}
			We obtain:
			\begin{align*}
				&\text{\underline{Hom}}_{\mathscr R|_U}(J(\mathscr A^* \otimes_{\mathscr A^0} E^*), I^*)\\
				&\simeq \text{\underline{Hom}}_{\mathscr R|_U}(\mathscr A^* \otimes_{\mathscr R|_U} V^*, I^*) &&\\
				&\simeq \text{\underline{Hom}}_{\mathscr R|_U}(V^*, I^*) && R \simeq \mathscr A^*, V^* \text{ is K-flat, and } I^* \text{ is K-injective} \\
				&\cong \text{Hom}_{\mathscr R(U)}(V^*(U), I^*(U)) && V^* \text{ is strictly perfect}\\
				&\simeq \text{Hom}_{\mathscr R(U)}(V^*(U), \mathscr A^* \otimes_{\mathscr A^0} F^*(U)) && V^*(U) \text{ is K-projective and Lemma \ref{LemmaAcyclic}}\\
				&\simeq \text{Hom}_{\mathscr A^*(U)}(\mathscr A^* \otimes_{\mathscr R|_U} V^*(U), \mathscr A^* \otimes_{\mathscr A^0} F^*(U)) &&\text{no sheafification as } V^* \text{ is strictly perfect}\\
				& && \text{and tensor-hom adjunction}\\
				&\cong \text{\underline{Hom}}_{\mathscr A^*|_U}(\mathscr A^* \otimes_{\mathscr R|_U} V^*, \mathscr A^* \otimes_{\mathscr A^0} F^*)\\
				&\simeq \text{\underline{Hom}}_{\mathscr A^*|_U}(\mathscr A^* \otimes_{\mathscr A^0} E^*, \mathscr A^* \otimes_{\mathscr A^0} F^*)
			\end{align*}
		\end{proof}
	\end{proposition}
	\begin{theorem}
		\label{MainEq}
		Let $\mathscr R$ be a sheaf of $k$-algebras and let $\mathscr R \to \mathscr A^*$ be a quasi-isomorphism of sheaves of dg-algebras on a topological space $X$. Set $A^* := \mathscr A^*(X)$ and let $h \in A^2(X)$ be closed. Under the assumptions \ref{ListAssumptions} we obtain an equivalence of categories
		\begin{align*}H^0((A^*, h)\text{-Mod}^{coh}) \to D^B_{perf}(\mathscr R)^{\check h}\end{align*}
		induced by the functor $J$ as in Definition \ref{DefJFunctor}.
		\begin{proof}
			By the previous two propositions, we have that $\mathscr J : H^0((\mathscr A^*, h)\text{-Mod}^{coh}) \to D^B_{perf}(\mathscr R)^{\check h}$ is an equivalence.
			Then we can apply Lemma \ref{simpleDGEquivalence}.
		\end{proof}
	\end{theorem}
	\subsection{de Rham algebra}
	Let $X = M$ be a real smooth manifold. Set $k = \mathbb R$ or $k = \mathbb C$ and let $\mathscr R$ be the constant sheaf $\underline k$ associated to $k = \mathbb R, \mathbb C$. Let $\mathscr A^* := \underline \Omega^*(M, k)$ be the sheaf of $k$-valued forms on $M$. Let $h$ be any closed 2-form on $M$ with values in $k$. Further note that a finitely generated projective module over $\mathscr A^0$ is the same as a locally free sheaf by the Serre-Swan correspondence. Assumption 5 in \ref{ListAssumptions} follows from homotopy invariance as discussed in \ref{HomotopyInvarianceSection}. The other assumptions are immediate. Thus, Theorem \ref{MainEq} applies. Therefore, we obtain an equivalence between the homotopy category of supervector bundles and $\check h$-twisted globally bounded perfect complexes over $k$. We give further characterizations of the latter category.
	\begin{def1}
		Let $R$ be a commutative ring and $Y$ a topological space. Let $M^*$ be a complex of $\underline R$-modules. We say that $M^*$ is finite cohomologically locally constant (fclc) if the cohomology sheaves $H^i(M^*)$ are zero for almost all $i$ and each $H^i(M^*)$ is a locally constant sheaf of finitely generated $R$-modules. In the twisted case, we mean that each $H^i(M^*)_\alpha$ is locally constant.
	\end{def1}
	\begin{lemma}
		Let $R$ be a PID and $Y$ a connected, locally path connected topological space. Let $M^*$ be a complex of sheaves of $R$-modules. The following are equivalent.
		\begin{enumerate}
			\item $M^*$ is fclc.
			\item $M^*$ is perfect.
			\item $M^*$ is globally bounded perfect.
		\end{enumerate}
		\begin{proof}
			First note the following. For every $x \in Y$ we can choose a path-connected open neighbourhood $U$ of $x$. In particular, we have $\underline R(U) = R$ and for all $y \in Y$ we have $\underline R_y = R$. Thus, $p : (U, R|_U) \mapsto (*, R)$ is a flat map of ringed spaces and the pullback functor $p^*$ is exact. \\ \\
			Suppose that $M^*$ is fclc. We show that $M^*$ is perfect globally bounded. We will do an induction by the number of non-zero cohomology groups of $M^*$. Starting with $n = 1$, $M^*$ is quasi-isomorphic (using truncations) to its only non-vanishing cohomology group. Thus, we can assume $M^i = 0$ for all $i \neq 0$ and $M^0$ is locally constant and locally finitely generated by at most $b$ generators. For any $x \in M$, choose an open neighbourhood $U$ of $x$ such that $U$ is path-connected and $M^0|_U$ is constant so that we have $p_*(N) = M^0|_U$ for some finitely generated $R$-module $N$. As $R$ is a PID, there is an exact sequence
			$0 \to R^a \to R^b \to N \to 0$, where $b$ is as above. Using that $p^*$ is exact, we obtain a strictly perfect complex quasi-isomorphic to $M^*|_U$. By construction $M^*$ is perfect globally bounded with bounding constants $-1 < 0$ and $N = 2b$. Now suppose that $M^*$ has $(n+1)$ non-vanishing cohomology groups. Choose some $a < b \in \mathbb Z$ such that $H^a(M^*), H^b(M^*) \neq 0 $. By \cite[\href{https://stacks.math.columbia.edu/tag/08J5}{Tag 08J5}]{stacks-project} we obtain a distinguished triangle in the derived category
			\begin{align*} \tau_{\leq a} M^* \to M^* \to \tau_{\geq a+1}M^*.\end{align*}
			The first truncation kills the cohomology group in degree $b$ and the second one kills the cohomology group in degree $a$. The induction hypothesis thus applies and we obtain that both outer terms are perfect globally bounded. Since perfect complexes are closed under extensions, we obtain that $M^*$ is perfect. To see that it is globally bounded, note that locally we have that $M^*$ is quasi-isomorphic to the cone of a map between strictly perfect complexes (because $\tau_{\geq a+1} M^*$ is perfect) which are bounded by the bounding constants of $\tau_{\leq a} M^*$ and $\tau_{\geq a+1}M^*$. Thus, $M^*$ is globally bounded by bounding constants only depending on those of $\tau_{\leq a} M^*$ and $\tau_{\geq a+1}M^*$.\\ \\
			Now let $M^*$ be a perfect complex. We have to show that $M^*$ is fclc. By definition, and since $Y$ is connected, it suffices to show that this is the case when $M^*$ is strictly perfect. Since $M^*$ is bounded, its cohomology is also bounded. Thus, we only have to show that each $H^i(M^*)$ is locally constant and finitely generated. Furthermore it is easy to see that $p^* \circ p_*$ is the identity when restricted to finitely generated $R|_U$-modules that are direct summands of free modules. Therefore, we have $H^i(M^*)|_U = p^*(H^i(M^*(U)))$ (using exactness of $p^*$ as discussed above) and so $H^i(M^*)|_U$ is constant. As $R$ is a PID, it is also finitely generated.\\ \\
		\end{proof}
	\end{lemma}
	
	\begin{theorem}
		\label{MainEqM}
		Let $M$ be a real, connected smooth manifold. Let $\mathbb K = \mathbb R$ or $\mathbb K = \mathbb C$ and let $h \in \Omega_{cl}^2(M, \mathbb K)$ be a closed 2-form. Then there are equivalences
		\begin{align*}H^0(\mathcal P(M)^{\infty}[h])) \to H^0((\underline \Omega^*(M, \mathbb K), h)\text{-Mod}^{\text{coh}}) \to D^B_{\text{perf}}(\underline{\mathbb K})^{\check h} = D_{\text{perf}}(\underline{\mathbb K})^{\check h} = D_{\text{fclc}}(\underline{\mathbb K})^{\check h} \end{align*}
		\begin{proof}
			Combine the previous Lemma with Theorem \ref{MainEq}.
		\end{proof}
	\end{theorem}
	
	\subsection{Dolbeault algebra}
	Let $X$ be a smooth, connected complex manifold. Let $k = \mathbb C$ and $\mathscr A^* := \underline{\Omega}^{0,*}(M)$ be the Dolbeault-algebra sheaf on $M$ and let $h \in \Omega^{0,*}(M)$ be a closed $(0,2)$-form. Furthermore let $\mathscr R := \mathcal O_X$ be the sheaf of holomorphic functions on $M$. Assumption 5 in \ref{ListAssumptions} follows by \cite[Lemma 4.1.5]{Block} and the others are immediate. Once again we obtain that Theorem \ref{MainEq} applies. In the case that $X$ is compact we can give another description of globally bounded perfect modules over $\mathcal O_X$.
	
	\begin{lemma}
		Let $M^*$ be a globally bounded perfect complex over $\mathcal O_X$. Then the cohomology of $M^*$ is concentrated in finitely many degrees and each $H^i(M^*)$ is a coherent sheaf. Conversely, let $M^*$ be a complex such that the cohomology of $M^*$ is concentrated in finitely many degrees and each $H^i(M^*)$ is a coherent sheaf, then $M^*$ is a perfect complex.
		\begin{proof}
			Let $M^*$ be a globally bounded perfect complex. By definition this means that the cohomology of $M^*$ is concentrated in finitely many degrees. Furthermore, since coherency is a local property we may assume that $M^*$ is strictly perfect. It is well known that $\mathcal O_X$ itself is a coherent sheaf for a complex manifold. Furthermore, the full subcategory of coherent sheaves is abelian. Thus, it suffices to observe that direct summands of coherent sheaves are coherent which is straightforward.\\ \\
			Now let $M^*$ be a complex whose cohomology is concentrated in finitely many degrees and consists of coherent sheaves. In \cite[p.696]{Griffiths} it is shown that every coherent sheaf is perfect when seen as a complex concentrated in degree $0$. We proceed by an induction over the number of non-zero cohomology groups of $M^*$. For $n = 1$ we can truncate our complex so that we can assume $M^i = 0$ for all $i \neq 0$ and $M^0$ is coherent. By the above mentioned fact we obtain that $M^*$ is perfect. Now suppose that $M^*$ has $(n+1)$ non-vanishing cohomology groups. Choose some $a < b \in \mathbb Z$ such that $H^a(M^*), H^b(M^*) \neq 0 $. By \cite[\href{https://stacks.math.columbia.edu/tag/08J5}{Tag 08J5}]{stacks-project} we obtain a distinguished triangle in the derived category
			\begin{align*} \tau_{\leq a} M^* \to M^* \to \tau_{\geq a+1}M^*.\end{align*}
			The first truncation kills the cohomology group in degree $b$ and the second one kills the cohomology group in degree $a$. The induction hypothesis thus applies and we obtain that both outer terms are perfect. Since perfect complexes are closed under extensions, we obtain that $M^*$ is perfect.
		\end{proof}
	\end{lemma}
	\begin{corollary}
		Let $X$ be a compact, complex manifold and $M^*$ a complex of $\mathcal O_X$-modules. Then the following are equivalent:
		\begin{itemize}
			\item The cohomology of $M^*$ is concentrated in finitely many degrees and each $H^i(M^*)$ is coherent.
			\item $M^*$ is globally bounded perfect.
			\item $M^*$ is perfect.
		\end{itemize}
	\end{corollary}
	Denote by $D_{\text{coh}}^B(X)^{\check h}$ the full subcategory of $D(X)^{\check h}$ given by those complexes $M^*$ whose cohomologies are concentrated in finitely many degrees and such that each $H^i(M^*)$ is coherent. A version of the following theorem appeared in \cite[Theorem 4.2.1]{Block} in the case that $X$ is compact.
	\begin{theorem}
		\label{MainEqX}
		Let $X$ be a complex, connected smooth manifold. Let $h \in \Omega_{cl}^{0,2}(M)$ be a $\bar \partial$-closed $(0,2)$-form. Then there are equivalences
		\begin{align*}H^0((\Omega^{0,*}(X), h)\text{-Mod}^{\text{coh}}) \to H^0((\underline \Omega^{0,*}(X), h)\text{-Mod}^{\text{coh}}) \to D^B_{\text{perf}}(X)^{\check h}.\end{align*}
		The latter category embeds fully faithfully into the bounded derived category $D^B_{\text{coh}}(X)^{\check h}$ of objects whose cohomology sheaves are coherent. This embedding is an equivalence if $X$ is compact. 
	\end{theorem}
	\begin{proof}
		The first part follows from Theorem \ref{MainEq}. The second statement follows from the previous Lemma and Corollary.
	\end{proof}
	\nocite{BlockMukai}

	\newcommand{\link}[2]{{\color{blue}\href{#1}{#2}}}
	
	\subsection*{Acknowledgements}
	I would like to sincerely thank my advisor, Julian Holstein, for his continuous guidance and support throughout the development of this paper. In many instances, his encouragement has pushed me in the right direction, and his numerous fruitful suggestions shaped the quality of my work. The author has also benefited from suggestions by Vicente Cortés. Furthermore, I thank Jim Stasheff for advising me and sharing his thoughts.\\ \\
	The results of the present paper have been obtained in the author's master's thesis and have been extended, and revised during his PhD studies funded by the Deutsche Forschungsgemeinschaft (DFG, German Research Foundation) – SFB-Geschäftszeichen 1624 – Projektnummer 506632645.

	\bibliographystyle{alphaurl}
	\bibliography{AdArxiv}
	
\end{document}